\documentclass[12pt,oneside,openany,article]{memoir}
\usepackage{mempatch}

\nouppercaseheads 
\usepackage[amsthm,thmmarks,hyperref]{ntheorem}
\usepackage[noTeX]{mmap}
\usepackage{etex}

\usepackage[english]{babel}
\usepackage[leqno]{mathtools}

\usepackage{mathrsfs}
\usepackage{upgreek}
\usepackage[compress,square,comma,numbers]{natbib}
\usepackage{hypernat} 

\usepackage{sfmath}

\usepackage{microtype}

\usepackage{subfig}
\usepackage{wrapfig}
\usepackage{array}

\usepackage{graphicx,color}
\definecolor{gray}{gray}{0}
\pagecolor{white}

\usepackage{hyperxmp}
\usepackage{xr-hyper}
\usepackage{nameref}
\usepackage[pdftex,bookmarks,pdfnewwindow,plainpages=false,unicode]{hyperref}

\usepackage{bookmark}

\usepackage{enumitem}

\usepackage{amssymb} 

\hypersetup{
colorlinks=true,
linkcolor=black,
citecolor=black,
urlcolor=blue,
pdfauthor={Victor Ivrii},
pdftitle={Local Spectral Asymptotics for  2D-Schroedinger Operator with Strong Magnetic Field near Boundary},
pdfsubject={Sharp Spectral Asymptotics},
pdfkeywords={Microlocal Analysis,  Sharp Spectral Asymptotics, Schroedinger Operator, Strong Magnetic Field, Boundary},
bookmarksdepth={4}
}

\numberwithin{equation}{chapter}

\theoremstyle{plain}
\newtheorem{theorem}{Theorem}[chapter]

\newtheorem{proposition}[theorem]{Proposition}
\newtheorem{corollary}[theorem]{Corollary}
\newtheorem{conjecture}[theorem]{Conjecture}

\theoremstyle{definition}

\theoremstyle{remark}
\newtheorem{remark}[theorem]{Remark}
\newtheorem{example}[theorem]{Example}
\newtheorem{problem}[theorem]{Problem}

\numberwithin{equation}{chapter}

\DeclareMathAlphabet{\mathpzc}{OT1}{pzc}{m}{it}

\newcommand{\cA}{\mathcal{A}}

 \newcommand{\cL}{\mathcal{L}}
 
 \newcommand{\cN}{\mathcal{N}}
 
 \newcommand{\cP}{\mathcal{P}}

 \newcommand{\cX}{\mathcal{X}}
 
 \newcommand{\cZ}{\mathcal{Z}}

 \newcommand{\sC}{\mathscr{C}}

 \newcommand{\sL}{\mathscr{L}}

\newcommand{\inn}{{\mathsf{inn}}}

\newcommand{\bound}{{\mathsf{bound}}}

\newcommand{\corr}{{\mathsf{corr}}}

\newcommand{\D}{{\mathsf{D}}}

\newcommand{\MW}{{\mathsf{MW}}}
\newcommand{\N}{{\mathsf{N}}}

\newcommand{\R}{{\mathsf{R}}}
\renewcommand{\S}{{\mathsf{S}}}
\newcommand{\T}{{\mathsf{T}}}
\newcommand{\trans}{{\mathsf{trans}}}

\newcommand{\W}{{{\mathsf{W}}}}
\newcommand{\w}{{\mathsf{w}}}

\newcommand{\const}{{\mathsf{const}}}
\newcommand{\dist}{{{\mathsf{dist}}}}

\newcommand{\eff}{{{\mathsf{eff}}}}

\newcommand{\new}{{{\mathsf{new}}}}
\newcommand{\old}{{{\mathsf{old}}}}
\newcommand{\reg}{{{\mathsf{reg}}}}

\newcommand{\bR}{{\mathbb{R}}}

\newcommand{\bZ}{{\mathbb{Z}}}

\newcommand{\fz}{{\mathfrak{z}}}

\def\1{\boldsymbol {|}}
\newcommand{\boldgamma}{{\boldsymbol{\gamma}}}

\newcommand{\Def}{\mathrel{\mathop:}=}

\newcommand{\WF}{\operatorname{WF}}

\newcommand{\degen}{{\operatorname{deg}}}     

\newcommand{\mes}{\operatorname{mes}}

\renewcommand{\Re}{\operatorname{Re}}       

\newcommand{\sign}{\operatorname{sign}}

\newcommand{\supp}{\operatorname{supp}}

\newenvironment{claim}[1][{\textup{(\theequation)}}]{\refstepcounter{equation}\vglue10pt
\begin{trivlist}
\item[{\hskip\labelsep#1}]}{\vglue10pt\end{trivlist}}

\newenvironment{claim*}[1][{}]{\vglue10pt
\begin{trivlist}
\item[{\hskip\labelsep#1}]}{\vglue10pt\end{trivlist}}

\newenvironment{phantomequation}[1][]{\refstepcounter{equation}}{}
\newcounter{note}

\DeclareTextCommand{\textinfty}{PU}{\9042\036}

\DeclareTextCommand{\textge}{PU}{\9042\145}
\DeclareTextCommand{\textle}{PU}{\9042\144}
\DeclareTextCommand{\texthat}{PD1}{\136}

\begin{document}
\title{Local Spectral Asymptotics for  $2D$-Schr\"{o}dinger Operator with Strong Magnetic Field near Boundary}
\author{Victor Ivrii}

\maketitle
{\abstract%
We consider 2-dimensional Schr\"odinger operator with the non-degenerating magnetic field in the domain with the boundary and under certain non-degeneracy assumptions we derive  spectral asymptotics with the remainder estimate better than $O(h^{-1})$, up to $O(\mu^{-1}h^{-1})$ and the principal part 
$\asymp h^{-2}$ where $h\ll 1$ is Planck constant and $\mu \gg 1$ is the intensity of the magnetic field; $\mu h \le 1$.

We also consider generalized Schr\"odinger-Pauli operator in the same framework albeit with $\mu h\ge 1$ and derive  spectral asymptotics with the remainder estimate up to $O(1)$ and with the principal part $\asymp \mu h^{-1}$, or, under certain special circumstances with the principal part $\asymp \mu^{\frac{1}{2}} h^{-\frac{1}{2}}$
\endabstract}

\setcounter{section}{-1}
\section{Introduction}
\label{sect-15-0}

Our goal is to derive spectral asymptotics of 2-dimensional Schr\"odinger operator 
\begin{equation}
A=\sum_{j,k} P_jg^{jk}P_k+V,\qquad \text{with\ \ } P_j=hD_j-\mu V_j
\label{13-1-1}
\end{equation}
near the boundary
where
\begin{claim}\label{13-1-2}
$g^{jk}=g^{kj}$, $V_j$ and $V$ are real-valued functions, $ h\in (0,1]$, $\mu \in [1,\infty )$.
\end{claim}
and in $B(0,1)\subset X$ the following conditions are fulfilled:
\begin{phantomequation}\label{13-1-3}\end{phantomequation}
\begin{equation}
|D^\alpha g^{jk}|\le c,\quad |D^\alpha V_j|\le c,\quad |D^\alpha V|\le c
\quad \qquad \forall \alpha :|\alpha |\le K,
\tag*{$\textup{(\ref*{13-1-3})}_{1-3}$}
\end{equation}
\begin{equation}
\epsilon _0\le \sum_{j,k} g^{jk}\eta _j\eta _k\cdot |\eta |^{-2}\le c
\quad \forall \eta \in \bR^d\setminus 0 \quad \forall x \in B(0,1)
\label{13-1-4}
\end{equation}

So we basically want  to generalize results of Chapter~\ref{book_new-sect-13}~\cite{futurebook}\footnote{\label{foot-15-0} This article is a rather small part of the huge project to write a book and is just Chapter~15 consisting entirely of newly researched results. Chapter~\ref{book_new-sect-13} corresponds to Chapter~6 of its predecessor V.~Ivrii~\cite{Ivr1}. External references by default are to  \cite{futurebook}.} $d=2$ to the case of $\supp \psi \cap\partial X\ne\emptyset$. We assume that condition (\ref{13-2-1}) is fulfilled: $F\asymp 1$.

However it is not a simple generalization as propagation near the boundary is completely different from one inside of the domain. While classical dynamics inside is a normal speed cyclotron movement combined with a slow (with the speed $O(\mu^{-1})$) magnetic drift, it is not the case near boundary: when cyclotron hits the boundary it reflects from it and we arrive to a normal speed (with the speed $O(1)$) hop movement along the boundary.

The really difficult part is that hop movement is not separated from cyclotron plus magnetic drift movement: first, as we move away from the boundary the former is replaced by the latter; second, during some hop the hop movement can be torn away from the boundary and become cyclotron plus magnetic drift movement and v.v.: cyclotron plus magnetic drift movement can collide with the boundary and become hop movement.

The main goal is to investigate the generic case as $\nabla VF^{-1}\ne 0$ and 
$\nabla_{\partial X} VF^{-1}= 0\implies \nabla^2_{\partial X} VF^{-1}\ne 0$ on $\partial X$.

\section*{Plan of the article}

Section~\ref{sect-15-1} is preliminary: first, we consider a classical dynamics, described above, in details. Then we consider a model operator in the half-plane and derive precise formula for it.

In section~\ref{sect-15-2} we consider a weak magnetic field case $\mu h \ll 1$ when classical dynamics defines everything. Then hop movement breaks periodic cyclotron movement which allows us to prove that the contribution of this zone to the remainder is $O(\mu^{-1}h^{-1})$ where factor $\mu^{-1}$ is the width of the boundary zone and $1$ is the time for which we typically follow classical dynamics. Recall that the contribution of inner zone to the remainder (under appropriate non-degeneracy conditions) also is $O(\mu^{-1}h^{-1})$ albeit there factor $\mu^{-1}$ comes from time $\asymp \mu$ for which we typically follow a classical dynamics. Sure, there is a transitional zone between boundary and inner zones but as magnetic field is not strong, it is very thin.

In section~\ref{sect-15-3} we study a strong magnetic field case. In subsections~\ref{sect-15-3-1}--\ref{sect-15-3-3} we establish Tauberian remainder estimates under different non-degeneracy assumptions. 

In subsection~\ref{sect-15-3-4} we study propagation of singularities in the transitional zone and find that under Dirichlet boundary condition we are able to prove better results than under Neumann boundary condition. This difference is not technical as it is the first manifestation of the fact that as magnetic field grows stronger the classical dynamics loses its value as predictor of the propagation and in the case of Neumann boundary condition some singularities propagate along the boundary in the direction opposite to hops; this is not related to a magnetic drift but rather to a different behavior of the eigenvalues of the model operator.

In subsection~\ref{sect-15-3-5} in the case of the strong magnetic field we pass from estimates to calculations and derive our final results.

In section~\ref{sect-15-4} we consider the cases of  superstrong magnetic field ($\mu h\asymp 1$ and $\mu h \gg 1$ respectively; in the latter case we need to take potential $-\fz \mu hF +V$ rather than $V$ to prevent pushing bottom of the spectrum to high. Here $\fz \ge 1$ under Dirichlet boundary condition but could be smaller under Neumann boundary condition; then spectral asymptotics are concentrated near boundary). As we know from Chapter~\ref{book_new-sect-13} \cite{futurebook} cyclotrons are no more observable due to uncertainty principle but magnetic drift preserves its sense. The same remains true near boundary: we do not see hops but we observe propagation along the boundary and it can be torn away from the boundary and became magnetic drift and v.v.

Finally, in section~\ref{sect-15-5} we consider generalizations: first, we get rid off condition $V\asymp 1$ and then we get rid off  condition $F\asymp 1$. In the latter case as we know from Chapter~\ref{book_new-sect-14} \cite{futurebook} our remainder estimate \emph{must\/} deteriorate to $O(\mu^{-\frac{1}{2}}h^{-1})$.

Appendix~\ref{sect-15-A} is devoted to auxiliary $1$-dimensional harmonic oscillator $-\partial_x^2 +x^2$ considered on $(-\infty,\eta]$ with Dirichlet or Neumann boundary condition as $x=\eta$. We study its eigenvalues $\lambda_{\D,n}(\eta)$ and $\lambda_{\N,n}(\eta)$ reproducing results of B.~Helffer, C.~Bolley and M.~Dauge. These properties are crucial for our formulae and for analysis of section~\ref{sect-15-4}.

\chapter{Preliminary discussion}
\label{sect-15-1}

\section{Inner and boundary zones}
\label{sect-15-1-1}
Let us consider $2$-dimensional Magnetic Schr\"odinger operator near the boundary. Let as usual $\gamma(x)=\frac{1}{2}\dist(x, \partial X)$ and consider disk (i.e. $2$-dimensional ball) $B(y,\gamma(y))$. Scaling it  $x\mapsto (x-y)\gamma^{-1}$ to the unit disk we get
\begin{multline}
h\mapsto h_\new=h\gamma^{-1},\quad \mu \mapsto \mu_\new=\mu \gamma,\quad 
\nu \mapsto \nu_\new=\nu\gamma,\\
\bar{\nu}\mapsto \bar{\nu}_\new=\bar{\nu}\gamma
\label{15-1-1}
\end{multline}
where recall $\nu$ and $\bar{\nu}$ are introduced in (\ref{book_new-13-3-11}) and \ref{book_new-13-2-77} \cite{futurebook}. We assume that  originally $\nu=\bar{\nu}=1$. Then according to Chapter~\ref{book_new-sect-13} \cite{futurebook} under reasonable conditions contribution of $B(y,\gamma(y))$ to the remainder $\R^\MW_{\infty}$ does not exceed 
\begin{equation}
C\nu_\new \mu_\new^{-1}h_\new^{-1}=C\mu^{-1}h^{-1}\gamma
\label{15-1-2}
\end{equation}
as $\gamma\ge \mu h$ and dividing by $\gamma^2$ and integrating we conclude that the contribution of the \emph{inner zone\/}\index{inner zone}\index{zone!inner} 
\begin{equation}
X_\inn\Def \{ x, \gamma(x)\ge \bar{\gamma}\Def C_0\mu^{-1}\}
\label{15-1-3}
\end{equation}
to $\R^\MW_\infty$ does not exceed 
\begin{equation}
C\mu^{-1}h^{-1}\int \gamma^{-1}\,dx\asymp C\mu^{-1}h^{-1}\log \mu
\label{15-1-4}
\end{equation}
as $\mu\le h^{-\frac{1}{2}}$.

\begin{remark}\label{rem-15-1-1}
We \emph{must\/} take $C_0$ large enough in (\ref{15-1-2}) to ensure that 
$\mu_\new\ge \mu_0$ where  $\mu_0$ is large enough constant.
\end{remark}

One can get rid off logarithmic factor in (\ref{15-1-4}) using Seeley' approach (R.~Seeley, \cite{S1,S2}; see also section~\ref{book_new-sect-7-4})~\cite{futurebook}. One also should consider case $\mu \ge h^{-\frac{1}{2}}$ but then zone $\{\mu h \ge \gamma \ge C_0\mu^{-1}\}$ appears and its contribution to remainder is estimated differently, as 
\begin{equation}
C\int \gamma(x)^{-1}T(x)^{-1}\, dx.
\label{15-1-5}
\end{equation}
We will postpone this until real proofs. 

In the \emph{boundary zone\/}\index{boundary zone}\index{zone!boundary} 
\begin{equation}
X_\bound \Def \{ x, \gamma(x)\le 2\bar{\gamma}\}
\label{15-1-6}
\end{equation}
one can apply standard Weyl remainder estimate after rescaling. Then contribution of $B(y,\bar{\gamma})$ to $\R^\W$ does not exceed $Ch_\new^{-1}=Ch^{-1}\bar{\gamma}$ and the total contribution of $X_\bound$ to the remainder does not exceed 
$Ch^{-1}\bar{\gamma} \times \bar{\gamma}^{-1}=Ch^{-1}$. This is much worse than the contribution of the inner zone we derived and needs to be fixed. To do it we need to extend dynamics from time  $T\asymp \bar{\gamma}\asymp \mu^{-1}$ ($T\asymp 1$ after rescaling) to 
$T\asymp 1$ ($T\asymp \mu$ after rescaling) or even more in the strong, very strong and superstrong magnetic field cases.

\section{Classical dynamics near the boundary}
\label{sect-15-1-2}
To understand the role of the boundary, consider classical dynamics. Let us consider first a half-plane $X=\bR^2_+=\{x,\ x_1>0\}$, $g^{jk}=\updelta_{jk}$, $F=1$ and $V=0$. Let us write the model operator  in the form
\begin{equation}
\bar{A}\Def h^2D_1^2+ (hD_2-\mu x_1)^2,
\label{15-1-7}
\end{equation}
as $F_{12}\Def \partial _2V_1-\partial _1V_2=-1$ according to (\ref{book_new-13-1-7})~\cite{futurebook}.

Then we have a Hamiltonian circular trajectory 
\begin{multline}
x_1=\mu^{-1}\bar{\xi}_2+ a\mu^{-1}\cos 2\mu t, \qquad 
x_2=\bar{x}_2-a\mu^{-1}\sin 2\mu t,\\
\xi_1=-a\sin2\mu t,\qquad \xi_2=\bar{\xi}_2
\label{15-1-8}
\end{multline}
where $a=\tau^{\frac{1}{2}}$ and $\tau$ is an energy level. So we got circular counter-clock-wise trajectories of the radius $\tau^{\frac{1}{2}}$ centered at $\bar{x}$ with $\bar{x}_1=\mu^{-1}\bar{\xi}_2$ and depending on $\bar{\xi}_2$ these trajectories behave differently:

\medskip\noindent
(a) As $\bar{\xi}_2\ge \tau^{\frac{1}{2}}$ trajectory does not intersect $\partial X$ or just touches it and remains circular. 

\medskip\noindent
(b) As $\bar{\xi}_2< \tau^{\frac{1}{2}}$ trajectory reflects from $\partial X$ and we get a ``hop''-movement: 
\begin{figure}[h]
\includegraphics[scale=1.15]{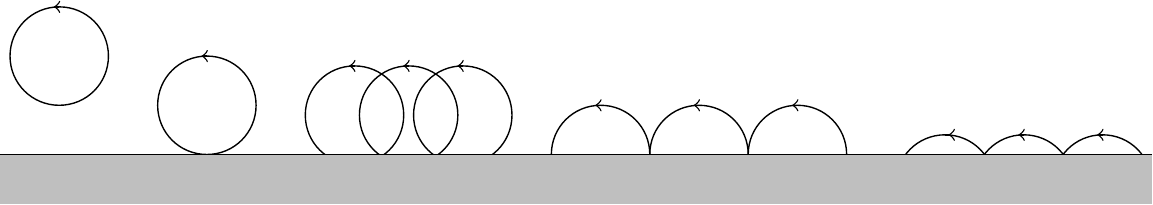}
\caption{\label{fig-hop} Different classical trajectories in half-plane for model operator}
\end{figure}

\medskip\noindent
(c) As $\bar{\xi}_2\searrow -\tau^{\frac{1}{2}}$ trajectory stays closer and closer to $\partial X$ and becomes a kind of gliding ray in the limit.  And we are interested only at zone   $\{\xi_2\ge - \tau^{\frac{1}{2}}\}$. 

\medskip
So, trajectories described in (b) are not periodic even for the model operator. 

\begin{figure}[h!]
\centering
\subfloat[$\xi_2>0$]{\includegraphics[width=0.4\linewidth]{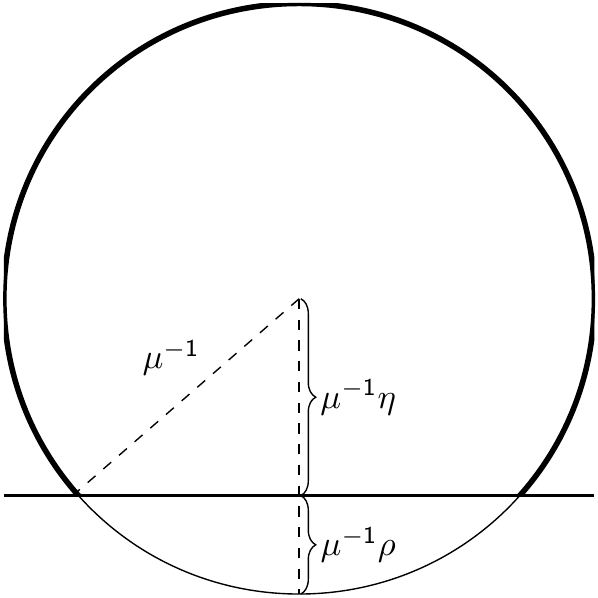}} \qquad
\subfloat[$\xi_2<0$]{\includegraphics[width=0.4\linewidth]{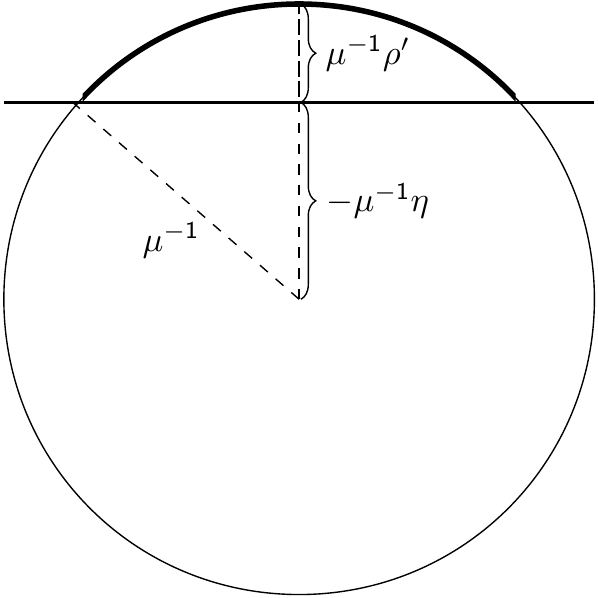}}
\caption{\label{fig-hop2} Calculating the length of the hop; $\rho=1-\eta$, $\rho'=1+\eta$}
\end{figure}

One can calculate easily that
\begin{claim}\label{15-1-9}
The \emph{length of the hop\/} (i.e. the distance between hop's start and end points) is $2a\mu^{-1}(1-\eta^2)^{\frac{1}{2}}$ with $\eta=\xi_2/a$ while the length of the arc is $2a\mu^{-1}(\pi -\arccos \eta )$ and thus the time of the hop is  $\mu^{-1}(\pi -\arccos \eta )$. 
\end{claim}\index{hop}\index{hop!length}

Therefore 

\begin{claim}\label{15-1-10} 
As $|\eta|<1$ \emph{the average hop-speed\/} along $x_2$ is $\frac{\updelta x_2}{\updelta t}=-2av(\eta)$ with\index{hop!speed}
\begin{equation}
v(\eta)\Def \frac{(1-\eta^2)^{\frac{1}{2}}}{\pi-\arccos (\eta)}
\label{15-1-11}
\end{equation}
\end{claim}

\begin{wrapfigure}[4]{l}[4.pt]{3truecm}
\includegraphics[scale=1]{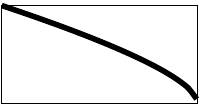}
\end{wrapfigure}
\noindent
This is a plot of $v(\eta)$. One can see easily that $v(\eta)$ is defined on $(-1,1)$ where it decays from $0$ to $1$;   $\eta=1$ is a threshold between circular and hop-movement and $\eta=-1$ corresponds to gliding rays. 
\index{hop!movement}

\medskip
As we cannot get remainder estimate better than $O(\mu h^{-1})$ for model operator we need to consider a perturbation by a potential: 
\begin{equation}
A\Def h^2D_1^2+ (hD_2-\mu x_1)^2+V(x);
\label{15-1-12}
\end{equation}
a classical dynamics for a general operator (\ref{13-1-1}) in dimension $d=2$ will be not different in our assumptions.

Then one should use the same classification as before with
\begin{equation}
\eta \Def \frac{\xi_2}{W_0(x)^{\frac{1}{2}}},\qquad W= (\tau -V(x))F(x)^{-1}, \ W_0=W|_{\partial X}
\label{15-1-13}
\end{equation}
where \emph{so far\/} $F(x)=1$. 

Actually the last statement is not completely true in the \emph{transitional zone\/}\index{zone!transitional} 
(where now we consider zones in $(x,\xi)$-space)
\begin{gather}
\cX_\trans= \bigl\{(x,\xi):\ |\xi_2 - W_0(x)^{\frac{1}{2}}|\le  2\bar{\rho}\bigr\}
\label{15-1-14}
\intertext{or equivalently $\{|\eta - 1|\le C_0\bar{\rho}\}$ where \emph{so far}} 
\bar{\rho}=C_0\mu^{-1}
\label{15-1-15}
\end{gather}
but later it may be increased due to uncertainty principle.

So, let us introduce an \emph{inner zone\/}\index{zone!inner}
\begin{equation}
\cX_\inn= \bigl\{(x,\xi):\ \xi_2 - W_0(x)^{\frac{1}{2}}\ge \bar{\rho}\bigr\}
\label{15-1-16}
\end{equation}
and a \emph{boundary zone\/}\index{zone!boundary}
\begin{equation}
\cX_\bound= \bigl\{(x,\xi):\ \xi_2 - W_0(x)^{\frac{1}{2}}\le -\bar{\rho}\bigr\}.
\label{15-1-17}
\end{equation}
There is no need to consider \emph{gliding zone\/}\index{zone!gliding} 
\begin{equation}
\cX_{\mathsf{glid}}= \bigl\{(x,\xi):\ \xi_2 + W_0(x)^{\frac{1}{2}}\le \bar{\rho}\bigr\}
\label{15-1-18}
\end{equation}
separately from $\cX_\bound$.

Recall that inside of the domain potential causes magnetic drift
\begin{phantomequation}\label{15-1-19}\end{phantomequation}
\begin{equation}
\frac{d\ }{dt}x_1=\mu^{-1}\partial_{x_2}W,\qquad 
\frac{d\ }{dt}x_2=-\mu^{-1}\partial_{x_1}W.
\tag*{$\textup{(\ref*{15-1-19})}_{1,2}$}\label{15-1-19-*}
\end{equation}

Let us analyze what happens near the boundary. Note first that

\begin{claim}\label{15-1-20}
Billiards do not branch as $\mu\ge \mu_0$ where $\mu_0$ is large enough.
\end{claim}
Really,  one can prove easily that
\begin{claim}\label{15-1-21}
With respect to Hamiltonian trajectories $\partial X$ is strongly  concave in the gliding zone and strongly convex in the transitional zone (and domain $X$ has opposite property) as $\mu\ge \mu_0$.
\end{claim}
Recall that (according to Figure~\ref{fig-hop2}) $\rho=1-\eta$, $\rho'=1+\eta$, $\eta =\xi_2W_0^{-1}$. Then along Hamiltonian trajectories 
\begin{multline*}
\frac{d\rho' }{dt} =  \frac{d\ }{dt}\bigl(\xi_2 W_0^{-\frac{1}{2}}\bigr)=
W_{x_2}W_0^{-\frac{1}{2}}-
\frac{1}{2}\xi_2 W_0^{-\frac{3}{2}}W_{0x_2}\frac{dx_2}{dt} \equiv\\
W_{0x_2}W_0^{-\frac{1}{2}}\Bigl(1-\frac{1}{2}\xi_2W_0^{-1}\frac{dx_2}{dt}\Bigr)\qquad
\mod O\bigl(\mu^{-1}\rho'\bigr)
\end{multline*}
as $x_1=O(\mu^{-1}\rho')$ and therefore for one hop
\begin{gather*}
\frac{\updelta\rho' }{\updelta x_2} \equiv
W_{0x_2}W_0^{-\frac{1}{2}}\Bigl(1-\frac{1}{2}\xi_2W_0^{-1}
\frac{\updelta x_2}{\updelta t}\Bigr)\\
\intertext{with $W_0,W_{0x_2},\xi_2$ calculated in the middle of it}
\equiv 
W_{0x_2}W_0^{-\frac{1}{2}}\Bigl(1+\eta v(\eta)\Bigr)v(\eta)^{-1}\equiv \frac{2}{3} \rho' W_{0x_2}W_0^{-\frac{1}{2}}
\quad
\mod O\bigl(\rho'(\mu^{-1}+\rho')\bigr)
\end{gather*}
as $\eta=\rho'-1$ and $v(\eta)=1-\frac{1}{3}\rho'\mod \rho^{\prime\,2}$.

Therefore 
\begin{claim}\label{15-1-22}
Along trajectories of the length $\le 1$  \ $\rho' \exp(-\frac{4}{3}W_0^{\frac{1}{2}})$ remains constant modulo $O\bigl(\rho'(\mu^{-1}+\rho')\bigr)$.
\end{claim}

\begin{remark}\label{rem-15-1-2}
The similar statement would be completely wrong for $\rho$ because as $\rho\approx 0$ $\frac{\updelta\rho}{\updelta t}\approx -W_{0x_2}W_0^{-\frac{1}{2}}$ and therefore as $W_{0x_2}<0$ hop-trajectories\footnote{\label{foot-15-1} In the positive time direction (and negative $x_2$-direction).} will be torn out of the boundary and begin magnetic drift movement. Meanwhile as $W_{0x_2}>0$ trajectories drifting in the inner zone may collide with the boundary and begin hop-movement\footref{foot-15-1}. 

In other words hops move away from the boundary (to the boundary) in the direction along the boundary, in which $W_0$ decreases (increases).
\end{remark}

\begin{example}\label{ex-15-1-3}
Meanwhile $W_{x_1}$ has more subtle effect. As hop-speed is larger than $C_0\mu^{-1}$ (i.e. in $\cX_\bound$) magnetic drift with respect to $x_2$ has no qualitative effect. However there are no hops in $\cX_\inn$. Therefore as $W_{x_1}\asymp 1$ we have two rather different cases:

\medskip\noindent
(i) $W_{x_1}>0$. Then according to $\textup{(\ref{15-1-19})}_2$ magnetic drift is to the left, in the same direction as hops. Then all dynamics is to the left\footref{foot-15-1}. In particular as $W_{x_2}>0$ hop-trajectories are torn from the boundary and begin drift movement (see figure~\ref{fig-traj}(a)) while as $W_{x_2}<0$ drift-trajectories collide with the boundary and begin hop-movement (see figure~\ref{fig-traj}(b)). 

\medskip\noindent
(ii)  $W_{x_1}<0$. Then according to $\textup{(\ref{15-1-19})}_2$ magnetic drift is to the right, in the opposite direction to the hops. So direction of dynamics (with respect to $x_2$) in $\cX_\inn$ is opposite to the hop-movement.  In particular as $W_{x_2}\le -\epsilon_0$ hop-trajectories are torn from the boundary and begin drift movement (see figure~\ref{fig-traj}(c))  while as 
$W_{x_2}\ge \epsilon_0$ drift-trajectories collide with the boundary and begin hop-movement (see figure~\ref{fig-traj}(d)). 
\begin{figure}[h!]
\centering
\subfloat[$W_{x_1}>0,W_{x_2}>0$]{\includegraphics[width=0.23\linewidth]{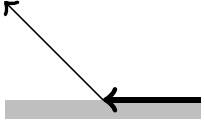}} \;
\subfloat[$W_{x_1}>0,W_{x_2}<0$]{\includegraphics[width=0.23\linewidth]{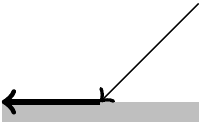}} \;
\subfloat[$W_{x_1}<0,W_{x_2}>0$]{\includegraphics[width=0.23\linewidth]{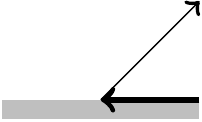}} \;
\subfloat[$W_{x_1}<0,W_{x_2}<0$]{\includegraphics[width=0.23\linewidth]{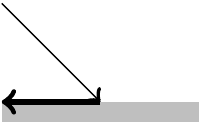}} 
\caption{\label{fig-traj} To example~\ref{ex-15-1-3}. Bold lines show hop movement and thin lines show drift movement which is along level lines of $W$.}
\end{figure}

This is consistent with the fact that drift trajectories are level curves of $W$. 
\end{example}

\begin{example}\label{ex-15-1-4}
Assume now that $W_{x_2}$ vanishes at some point but $W_{x_2x_2}\ne 0$. Then repeating analysis of the previous example we arrive to the following four pictures:
\begin{figure}[h!]
\centering
\subfloat[$W_{x_1}>0,W_{x_2x_2}>0$]{\includegraphics[width=0.25\linewidth]{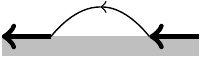}}\quad
\subfloat[$W_{x_1}>0,W_{x_2x_2}<0$]{\includegraphics[width=0.25\linewidth]{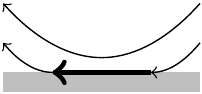}}\\ 
\subfloat[$W_{x_1}<0,W_{x_2x_2}>0$]{\includegraphics[width=0.25\linewidth]{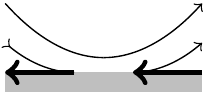}}\quad
\subfloat[$W_{x_1}<0,W_{x_2x_2}<0$]{\includegraphics[width=0.25\linewidth]{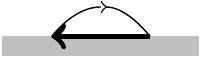}} 
\caption{\label{fig-traj3} To example~\ref{ex-15-1-4}. Bold lines show hop movement and thin lines show drift movement which is along level lines of $W$.}
\end{figure}

Again this is consistent with the fact that drift trajectories are level curves of $W$: in cases (a), (d), (b)--(c) point $\bar{x}$ is a local minimum,  maximum 
and minimax respectively.
\end{example}
However the following observation basically remains true:
\begin{claim}\label{15-1-23}
The  speed of magnetic drift is $O(\mu^{-1})$ (the typical speed is $\asymp \mu^{-1}$) while the speed of hop-movement is $O(1)$ (and the typical speed is $\asymp 1$).
\end{claim}

\section{Spectrum of the model operator}
\label{sect-15-1-3}
Consider model operator (\ref{15-1-7}) as $x_1>0$ with the Dirichlet or Neumann boundary condition $u|_{x_1=0}$ or $D_1u|_{x_1=0}$. Making $h$-Fourier transform we arrive to $1$-dimensional operator 
\begin{gather}
 h^2 D_1^2+ (\xi_2-\mu x_1)^2 
\label{15-1-24}\\
\intertext{which after transformations $x_1\mapsto \mu x_1$ or 
$x_1\mapsto  \hslash^{-\frac{1}{2}}x_1$ with $\hslash\Def  \mu ^{-1}h$ becomes}
\hbar^2 D_1^2 + (\xi_2-x_1)^2\qquad \text{with\ \ }\hbar=\mu h
\label{15-1-25}\\
\shortintertext{or}
L(\eta)\Def \hbar \bigl(D_1^2 + (\eta -x_1)^2\bigr)\qquad \text{with\ \ }
\eta = \hbar ^{-\frac{1}{2}}\xi_2\label{15-1-26}
\end{gather}
at $\bR^+\ni x_1$ respectively, again with the Dirichlet or Neumann boundary condition at $x_1=0$. As in the previous chapter one needs to distinguish between $\hbar=\mu h$ and $\hslash=\mu^{-1}h$.

Obviously 
\begin{claim}\label{15-1-27}
Each of these operators has a simple discrete spectrum, let 
$0\le \lambda_{*,0}(\eta)<\lambda_{*,1}(\eta) <\lambda_{*,2}(\eta)<\dots $ 
be eigenvalues of operator $L_*(\eta)$ defined by (\ref{15-1-26}) as $\hbar=1$ where $*$ means either $\D$ or $\N$ and as usual $\D$ and $\N$ denote Dirichlet and Neumann respectively.
\end{claim}
Further
\begin{claim}\label{15-1-28}
Let $\upsilon_{*,j}(x_1,\eta)$ be real-valued orthonormal eigenfunctions of operator $L_*(\eta)$ (as $\hbar=1$) corresponding to eigenvalues $\lambda_{*,j}(\eta)$.
\end{claim}

We will analyze them in details later (see Appendix~\ref{sect-15-A}), so far let us notice only that $\lambda_{*,j}(\eta)\to +\infty$ as $\eta\to -\infty$ and as $\lambda_{*,j}(\eta)$ are analytic we conclude that 

\begin{proposition}\label{prop-15-1-5}
(i) Spectrum of model operator \textup{(\ref{15-1-7})} in $\{x_1>0\}$ with Dirichlet or Neumann boundary conditions is absolutely continuous and occupies $[\hbar \cdot \inf _\eta \lambda_{*,0}(\eta),+\infty)$;

\medskip\noindent
(ii) Schwartz kernel of its spectral projector is
\begin{multline}
e(x,y,\tau,\mu, h)=\\[2pt]
(2\pi )^{-1} \hslash^{-1} \sum _{j\ge 0}\int 
e^{i\hslash^{-\frac{1}{2}} (x_2-y_2)\eta} 
\uptheta\bigl(\tau -\hbar \lambda_{*,j}(\eta)\bigr)
\upsilon_{*,j}(\hslash^{-\frac{1}{2}}x_1,\eta ) 
\upsilon_{*,j}(\hslash^{-\frac{1}{2}}y_1,\eta )\,d\eta.
\label{15-1-29}
\end{multline} 
\end{proposition}

Setting $x=y$, subtracting $h^{-2}\cN^\MW(\tau,\hbar)$ with $\cN^\MW(\tau,\hbar)$ defined by $\textup{(\ref{book_new-13-1-9})}_2$~\cite{futurebook}:
\begin{multline}
\cN^\MW(\tau ,\mu h)=\\
\#\{j\in\bZ^+:(2j+1)\mu hF+V\le\tau \}\cdot(2\pi )^{-1}\sqrt g(\mu h)F,
\label{13-1-9-2}
\end{multline}
and integrating with respect to $x_1$ we arrive to the following \emph{formal\/} (at least for now) equality
\begin{multline}
\int_0^\infty\Bigl(e(x,x,\tau,\mu, h)-h^{-2}\cN^\MW(\tau,\hbar)\Bigr)\,dx_1=
\\[2pt]
(2\pi)^{-1} \mu h^{-1} \int_0^\infty \sum_{j\ge 0}\Bigl( 
\int \uptheta\bigl(\tau -\hbar \lambda_{*,j}(\eta)\bigr)
\upsilon_{*,j}^2 (\hslash^{-\frac{1}{2}}x_1,\eta )\,d\eta -\\ 
\uptheta \bigl(\tau -(2j+1)\hbar\bigr)\Bigr)\, dx_1=
h^{-1}\cN^\MW_{*,\bound}(\tau, \mu h)
\label{15-1-30}
\end{multline}
with
\begin{multline}
\cN^\MW_{*,\bound}(\tau, \hbar)\Def\\
(2\pi)^{-1} \int_0^\infty \sum_{j\ge 0}\Bigl( 
\int \uptheta\bigl(\tau -\hbar \lambda_{*,j}(\eta)\bigr)
\upsilon_{*,j}^2 (x_1,\eta)\,d\eta - 
\uptheta \bigl(\tau -(2j+1)\hbar\bigr)\Bigr)\hbar^{\frac{1}{2}}\, dx_1
\label{15-1-31}
\end{multline}
where in the last transition we rescaled $x_1$. 

Recall that $\bar{\lambda}_j=(2j+1)$ are eigenvalues of operator (\ref{15-1-26}) on the whole line $\bR\ni x_1$ as $\hbar=1$.

The following properties of $\lambda_{*,j}(\eta)$ are useful to know

In virtue of proposition~\ref{prop-15-A-1}
\begin{phantomequation}\label{15-1-32}\end{phantomequation}
$\tau > (2j+1)\hbar \iff \tau > \lambda_{\D j}(\eta)$ as $0<\eta $ is large enough and thus 
\begin{equation}
\uptheta\bigl(\tau -\lambda_{\D j}(\eta)\hbar\bigr)=
\uptheta\bigl(\tau -(2j+1)\hbar\bigr)\impliedby 0< \eta \gg 1
\tag*{$\textup{(\ref*{15-1-32})}_\D$}\label{15-1-32-D}
\end{equation}
However $\tau \ge (2j+1)\hbar \iff \tau > \lambda_{\N j}(\eta)$ as $0<\eta $ is large enough and in this case 
\begin{equation}
\uptheta\bigl(\tau -\lambda_{\N j}(\eta)\hbar\bigr)=
\uptheta\bigl(\tau -(2j+1)\hbar+\bigr)\impliedby 0< \eta \gg 1
\tag*{$\textup{(\ref*{15-1-32})}_\N$}\label{15-1-32-N}
\end{equation} 

Recall that $\uptheta(\tau)$ is semi-continuous from the left and therefore $\uptheta(0)=0$. Meanwhile $\uptheta(\tau+)$ is semi-continuous from the right and therefore $\uptheta(0)=1$. Thus for Dirichlet boundary problem formulae (\ref{15-1-30})--(\ref{15-1-31}) should be retained
\begin{equation}
\int_0^\infty\Bigl(e_\D(x,x,\tau,\hbar)-h^{-2}\cN^\MW(\tau,\hbar)\Bigr)\,dx_1=\\
h^{-1} \cN^\MW_{\D,\bound} (\tau, \mu h)
\tag*{$\textup{(\ref*{15-1-30})}_\D$}\label{15-1-30-D}
\end{equation}
with
\begin{multline}
\cN^\MW_{\D,\bound}(\tau, \hbar)\Def\\
(2\pi)^{-1}  \int_0^\infty \sum_{j\ge 0}\Bigl( 
\int \uptheta\bigl(\tau -\hbar \lambda_{\D,j}(\eta)\bigr)
\upsilon_{\D ,j}^2 ( x_1,\eta )\,d\eta - 
\uptheta \bigl(\tau -(2j+1)\hbar\bigr)\Bigr)\hbar^{\frac{1}{2}}\, dx_1
\tag*{$\textup{(\ref*{15-1-31})}_\D$}\label{15-1-31-D}
\end{multline}
but for Neumann boundary problem they need to be adjusted to
\begin{equation}
\int_0^\infty\Bigl(e_\N(x,x,\tau,\hbar)-h^{-2}\cN^\MW(\tau+,\hbar)\Bigr)\,dx_1=
h^{-1} \cN^\MW_{\N,\bound}(\tau, \mu h)
\tag*{$\textup{(\ref*{15-1-30})}_\N$}\label{15-1-30-N}
\end{equation}
with
\begin{multline}
\cN^\MW_{\N,\bound}(\tau, \hbar)\Def\\
(2\pi)^{-1} \int_0^\infty \sum_{j\ge 0}\Bigl( 
\int \uptheta\bigl(\tau -\hbar \lambda_{\N,j}(\eta)\bigr)
\upsilon_{\N,j}^2 ( x_1,\eta )\,d\eta - 
\uptheta \bigl(\tau -(2j+1)\hbar+\bigr)\Bigr)\hbar^{\frac{1}{2}}\, dx_1.
\tag*{$\textup{(\ref*{15-1-31})}_\N$}\label{15-1-31-N}
\end{multline}
 
\begin{proposition}\label{prop-15-1-6}
(i) Equalities \ref{15-1-30-D}--\ref{15-1-31-D} and \ref{15-1-30-N}--\ref{15-1-31-N} hold and all integrals converge;

\medskip\noindent
(ii) Alternatively 
\begin{phantomequation}\label{15-1-33}\end{phantomequation}
\begin{multline}
\cN^\MW_{\D,\bound}(\tau, \hbar) =\\
(2\pi)^{-1}\sum_{j\ge 0}\int \Bigl(
\uptheta\bigl(\tau -\hbar \lambda_{\D,j}(\eta)\bigr)- 
\uptheta\bigl(\tau -(2j+1)\hbar\bigr)\uptheta (\eta) \Bigr) \hbar^{\frac{1}{2}}\,d\eta
\tag*{$\textup{(\ref*{15-1-33})}_\D$}\label{15-1-33-D}
\end{multline}
and
\begin{multline}
\cN^\MW_{\N,\bound}(\tau, \hbar) =\\
(2\pi)^{-1}\sum_{j\ge 0}\int \Bigl(
\uptheta\bigl(\tau -\hbar \lambda_{\N ,j}(\eta)\bigr)- 
\uptheta\bigl(\tau -(2j+1)\hbar+\bigr)\uptheta (\eta) \Bigr) \hbar^{\frac{1}{2}}\,d\eta
\tag*{$\textup{(\ref*{15-1-33})}_\N$}\label{15-1-33-N}
\end{multline}
\end{proposition}

\begin{proof}
One can prove easily that 
\begin{equation}
|x_1-\eta|\ge \epsilon (x_1+|\eta|) \implies |\upsilon_{*,j}(x_1,\eta)|\le 
C_j (1+x_1+|\eta|)^{-s}.
\label{15-1-34}
\end{equation}
By definition 
\begin{equation}
\int_0^\infty \upsilon_{*,j}^2(x_1,\eta)\,dx_1=1
\label{15-1-35}
\end{equation}
and based on these two facts and $\textup{(\ref{15-1-32})}_{\D,\N}$ one can prove (i) easily. 

\medskip
Inserting $\eta$ in the second term in the big parenthesis  in $\textup{(\ref{15-1-31})}_{\D,\N}$ through factor 
$1=\int \updelta(x_1-\eta)\,d\eta$, changing order of integration and calculating integral with respect to $x_1$ we prove (ii).
\end{proof}

\begin{remark}\label{rem-15-1-7}
From forthcoming weak magnetic field arguments it follows  that in the weak sense 
\begin{equation}
\cN_{*,\bound}^\MW (\tau,\hbar)\sim \sum_{n\ge 0}\kappa_{*n}(\tau) \hbar^{2n}
\label{15-1-36}
\end{equation}
with $\kappa_{*0}=\mp (4\pi)^{-1}\tau_+^{\frac{1}{2}}$ as 
$*=\D,\N$ respectively.
\end{remark}

\chapter{Weak magnetic field}
\label{sect-15-2}

\section{Precanonical form}
\label{sect-15-2-1}
We will consider a general magnetic Schr\"odinger operator (\ref{13-1-1}) in $X\subset\bR^2$
\begin{equation}
A=\sum_{1\le j,k\le 2} P_jg^{jk}P_k+V,\qquad \text{with\ \ } P_j=hD_j-\mu V_j
\label{15-2-1}
\end{equation}
satisfying in $B(0,1)\subset X$ assumptions (\ref{13-1-2})--(\ref{13-1-4}) but in contrast to (\ref{book_new-13-1-5}) \cite{futurebook}  we assume that 
\begin{equation}
X\cap B(0,1)=\{x_1>0\}\cap B(0,1).
\label{15-2-2}%
\end{equation}
Then in contrast to Chapter~\ref{book_new-sect-13} we cannot unleash a full power of Fourier integral operators as we must preserve boundary, but we can assume that $x_1$ is small, in fact as small as $\mu^{-1}P_j$ are: not exceeding 
$C\max\bigl(\mu^{-1},\mu^{-\frac{1}{2}}h^{\frac{1}{2}-\delta}\bigr)$.

Without any loss of the generality one can assume that 
\begin{phantomequation}\label{15-2-3}\end{phantomequation}
\begin{equation}
g^{11}=F,\qquad g^{12}=0, \qquad \bigl(g^{22}-F\bigr)\bigl|_{x_1=0}=0.
\tag*{$\textup{(\ref*{15-2-3})}_{1-3}$}
\label{15-2-3-*}
\end{equation}
To achieve subsequently $\textup{(\ref*{15-2-3})}_{1-3}$ we just reintroduce $x_1\Def \dist_{g/F}(x,\partial X)$ in the given metrics $(g^{jk}F^{-1})$, then reintroduce  $x_2\Def \alpha(x)+\beta(x_2)$ with an appropriate function $\alpha(x)$ and arbitrary function $\beta(x_2)$ and finally chose $\beta(x_2)$. 

\begin{remark}\label{rem-15-2-1}
In this construction $F$ is any arbitrary positive function but we select it to be a scalar intensity of the magnetic field. We cannot however choose $\partial_{x_1}g^{22}|_{x_1=0}$ as this is defined by the curvature of the metrics $(g^{jk}F^{-1})$. 
\end{remark}

Then  
\begin{equation}
F_{12}|_{x_1=0}=-1
\label{15-2-4}%
\end{equation}
(actually it may be $1$ but then we change $x_2\mapsto -x_2$) and without any loss of the generality one can assume that $V_1=0$ and $V_2=x_1 + O(x_1^2)$; we can always reach it by a gauge transformation.

So, changing $V$ by $O(h^2)$
\begin{multline}
A= (hD_1\bigr) F(x)(hD_1\bigr)+\\
\bigl(hD_2 - \mu x_1 - \mu x_1^2 b(x) \bigr)g^{22}(x)
\bigl(hD_2 - \mu x_1 - \mu x_1^2 b(x) \bigr)+V(x)
\label{15-2-5}%
\end{multline}
with $V(x)$ satisfying $\textup{(\ref{15-2-3})}_3$. One can easily generalize to such operator results of the previous section.

\section{Propagation of singularities}
\label{sect-15-2-2}

\subsection{No-critical point case}
\label{sect-15-2-2-1}

In the case of a \emph{weak magnetic field}\index{weak magnetic field}
\begin{equation}
\mu_0\le \mu \le h^{\delta-1}
\label{15-2-6}
\end{equation}
we know that inside of the domain the shift during the first winding is microlocally observable under condition (\ref{book_new-13-3-54}) \cite{futurebook} i.e. as 
\begin{equation}
|\nabla VF^{-1}|\ge \epsilon _0.
\label{15-2-7}
\end{equation}
We are going to prove that near boundary the same is true as
\begin{equation}
|\nabla_{\partial X} VF^{-1}|\ge \epsilon _0
\label{15-2-8}
\end{equation}
where $\nabla_{\partial X}$ is a derivative along boundary (i.e. $\partial_{x_2}$).

To prove this assertion one needs just to prove the following statement:

\begin{proposition}\label{prop-15-2-2}
(i) The propagation speed with respect to $(x,\xi_2)$ does not exceed $C$;

\medskip\noindent
(ii) Under condition \textup{(\ref{15-2-8})} for operator in the form \textup{(\ref{15-2-5})} propagation speed with respect to $\xi_2$ is of magnitude $1$ as $x_1 \le C\mu^{-1}$. Namely,  for time $t$ shift with respect to $\xi_2$ is of magnitude $|t|$ and has the same sign as $t W_{x_2}$ with $W=(\tau-V)F^{-1}$.
\end{proposition}

\begin{proof}
Proof follows the very standard way of the proof propagation of singularities as in theorems~\ref{book_new-thm-2-1-2} and~\ref{book_new-thm-3-1-2} \cite{futurebook}. Here we are using an auxiliary function $\varphi=\varphi (x,\xi_2,t)$ and not invoking actually reflections from the boundary. 

Then proof of statement (i) is then completely trivial.

To prove (ii) note that modulo $O(x_1)$
\begin{equation*}
\{a,\xi_2\}\equiv \{ F,\xi_2\}( a-V)+ \{V,\xi_2\} -
2\bigl(\xi_2- \mu x_1 - \mu x_1^2 b \bigr)F \mu x_1^2 \{b,\xi_2\}
\end{equation*}
and the last term is  $O(\mu^{-1})$ as $|a|=O(1)$; so 
\begin{equation}
\{a,\xi_2\}\equiv \{ F,\xi_2\}(a-\tau) + F^2\{(V-\tau)F^{-1},\xi_2\} 
+O(\mu^{-1})
\label{15-2-9}
\end{equation}
and as $(a-\tau)$ is small everything is defined by the second term in the right hand expression. 

Further details of this very simple proof are left to the reader.
\end{proof}

Then we immediately arrive to

\begin{corollary}\label{cor-15-2-3}
Let $U=U(x,y,t)$ be a Schwartz kernel of $e^{-ih^{-1}tA}$. Let 
$\psi (x)= \psi'(x_2)\psi''(\mu x_1)$ with fixed functions 
$\psi',\psi'' \in \sC_0^\infty$ .

Finally, let condition \textup{(\ref{15-2-8})} be fulfilled on $\supp \psi'$.   Then under condition \textup{(\ref{15-2-6})} both 
\begin{gather}
F_{t\to h^{-1}\tau}\chi_T(t) \Gamma (U \psi)
\label{15-2-10}
\shortintertext{and}
F_{t\to h^{-1}\tau}\bigl(\bar{\chi}_T(t)-\bar{\chi}_{T'}(t)\bigr) 
\Gamma (U \psi)
\label{15-2-11}
\end{gather}
are negligible as 
$\epsilon_1 \mu^{-1} \le T'< T\le \epsilon _2$ with arbitrarily small constant $\epsilon_1$ and small enough constant $\epsilon_2$~\footnote{\label{foot-15-2} Recall that as usual $\chi\in \sC^\infty_0([-1,-\frac{1}{2}]\cap [\frac{1}{2},1])$,
$\bar{\chi}\in \sC^\infty_0([-1,1])$ and equals $1$ at $[-\frac{1}{2},\frac{1}{2}]$.}.
\end{corollary}

Meanwhile rescaling $x\mapsto x\mu$, $h\mapsto \mu h$ we immediately arrive to

\begin{proposition}\label{prop-15-2-4}
Let condition \textup{(\ref{15-2-6})} be fulfilled and let condition \textup{(\ref{book_new-13-2-45})} \cite{futurebook}  i.e. 
\begin{equation}
V\le -\epsilon _0
\label{15-2-12}
\end{equation}
be fulfilled on $\supp \psi'$. Let $\psi$ be specified in corollary~\ref{cor-15-2-3}.

Then as $|\tau|\le \epsilon$ and $T=\epsilon \mu^{-1}$
\begin{multline}
F_{t\to h^{-1}\tau}\bar{\chi}_T(t) \Gamma (U \psi)\equiv\\
\sum_{n\ge 0, m\ge 0} \partial_\tau \kappa_{nm}(\tau) \mu^{2n-1}h^{2n+2m-1}
+
\sum_{n\ge 0, m\ge 0} \partial_\tau \kappa'_{nm}(\tau) \mu^{2n}h^{2n+2m}
\label{15-2-13}
\end{multline}
where $\kappa_{nm}(\tau)$ and $\kappa'_{nm}(\tau)$ are smooth coefficients. 
\end{proposition}

Combining proposition~\ref{prop-15-2-4} and corollary~\ref{cor-15-2-3} we immediately conclude that 

\begin{corollary}\label{cor-15-2-5}
Under conditions \textup{(\ref{15-2-6})}, \textup{(\ref{15-2-8})} and 
\textup{(\ref{15-2-12})} decomposition \textup{(\ref{15-2-13})} holds with  $T=\epsilon_2$.
\end{corollary}

Therefore applying Tauberian theorem we arrive to  estimate (\ref{15-2-14}) below as $\psi $ specified in corollary~\ref{cor-15-2-3}.

\begin{theorem}\label{thm-15-2-6}
Let conditions \textup{(\ref{13-1-1})}--\textup{(\ref{13-1-4})}, \textup{(\ref{15-2-2})}, \textup{(\ref{15-2-6})}, \textup{(\ref{15-2-8})} and 
\textup{(\ref{15-2-12})} be fulfilled on $\supp \psi$ where  
$\psi(x)\in \sC_0^\infty(\bR^2)$.  Then 
\begin{multline}
\R^\W_\infty\Def |\int_X \Bigl(e(x,x,0) -
\sum_{n\ge 0, m\ge 0}  \kappa_{nm}(0) \mu^{2n}h^{2n+2m-2}\Bigr)\psi(x)\,dx-\\
\int _{\partial X}\sum_{n\ge 0, m\ge 0} 
\kappa'_{nm}(0) \mu^{2n}h^{2n+2m-1}\psi (x)\,ds_g|\le
C\mu^{-1}h^{-1}
\label{15-2-14}
\end{multline}
where $ds_g$ is a measure on the boundary corresponding to metrics $g$.
\end{theorem}

\begin{proof}
So, as $\psi =\psi'(x_2)\psi''(\mu x_1)$ estimate (\ref{15-2-14}) is proven; thus contribution of zone $\{x,\ x_1\le 2C_0\mu^{-1}\}$ to the remainder is $O(\mu^{-1}h^{-1})$. This is rather a broad zone as $C_0$ is arbitrarily large and in zone  $\{x,\ x_1\ge C_0\mu\}$ analysis is almost as if there was no boundary. We need to use the time direction in which $\xi_2$ increases and therefore trajectories drift inside of $X$. Still it is not quite as without boundary as we are forced to use cut-off functions which are scaled with respect to $x_1$.

However, this is technicality and we overcome it in the following way: first note that such scaling functions are admissible as long as microlocal uncertainty principle  $\mu^{-1}\times \mu^{-2}\ge \mu^{-1}h^{1-\delta}$ holds as we reduced our operator to microlocal canonical form inside of domain. Here the first factor $\mu^{-1}$ is due to the scale. So, as $\mu \le h^{\delta-\frac{1}{2}}$ we can appeal to theory of Chapter~\ref{book_new-sect-13} \cite{futurebook}.

However we can change variables $x_{1\new}= x_1+b x_1^2$ so that 
\begin{equation*}
\xi_1= \bigl(1+O(x_1)\bigl)\xi_{1\new},\quad
\xi_2-\mu x_1-\mu bx_1^2= \xi_{2\new} -\mu x_{1\new} +O(x_1^2)\xi_1
\end{equation*}
and then bad term does not appear at all; in old coordinates it would amount to replacing $\xi_2$ by $\xi_2 + O(x_1^2\xi_1)$.

This proves that in the inner zone we can take $T\asymp \mu$ in the correct time direction and most importantly, we do not scale with respect to $x_2$ so microlocal uncertainty principle would be $1\times \mu^{-2}\ge \mu^{-1}h^{1-\delta}$ i.e. in our frames $\mu \le h^{\delta-1}$. Formula (\ref{15-2-14}) is proven.
\end{proof}

Finally we need to pass to magnetic Weyl formula:

\begin{theorem}\label{thm-15-2-7}
In frames of theorem~\ref{thm-15-2-6}
\begin{multline}
\R^\MW\Def 
|\int_X \Bigl( e(x,x,0)-h^{-2}\cN^\MW (x,0, \mu h)\Bigr)\psi(x)\,dx -\\ h^{-1}\int_{\partial X} \cN_{*,\bound}^\MW(x,0, \mu h )\,ds_g|\le 
C\mu^{-1}h^{-1}
\label{15-2-15}
\end{multline}
where $\cN^\MW_{*,\bound}$ is introduced by \ref{15-1-31-D} or  \ref{15-1-31-N} for Dirichlet or Neumann boundary condition respectively with 
$\hbar=\mu h F(x)$ and $\tau$ replaced by $-V(x)$.
\end{theorem}

\begin{proof}
To pass from (\ref{15-2-14}) to (\ref{15-2-15}) let us recall that coefficients 
are obtained by successive approximation method and notice that as $\psi$ is fixed (non-scaled) function running successive approximation method to derive (\ref{15-2-14}) leads to an extra factor $h$ rather than $\mu h$ if we differentiate $\psi$ or $g^{jk}$ or $V$ or differentiate twice $V_j$; however exactly one such differentiation leads to $0$ in the final calculations, so we conclude that at least $2$ ``losing'' differentiations must be there, extra factor is $O(h^2)$ and the term is $O(1)$. 

Thus to derive $\R^\W_\infty$ modulo $O(1)$ we should not differentiate $g^{jk},V,\psi$ at all and to differentiate $V_j$ only once. However it means exactly considering model operator $A=A_y$ at each point $y$.

The microhyperbolicity with respect to $x_2$ implies that with integration (with the scale $1$) over $x_2$ we can replace $\cN^\MW$ by its decomposition with respect to powers of $\hbar=\mu h$ which justifies the transition from (\ref{15-2-14}) to (\ref{15-2-15}). 
\end{proof}

\subsection{Analysis in the boundary zone}
\label{sect-15-2-2-2}
In what follows ``formula'' and ``remainder''  mean \emph{Tauberian formula with 
$T=\epsilon \mu^{-1}$\/} and the corresponding \emph{Tauberian remainder\/} $\R^\T$ until we will pass to Weyl and magnetic Weyl formula and remainder. 

To cover the case  when condition (\ref{15-2-8}) is violated (with an extreme case $W\Def VF^{-1}|_{\partial X}=\const$) let us consider first boundary zone $\cX_\bound$ defined by (\ref{15-1-17}) where we chose a small parameter $\bar{\rho}$ later. To do so let us consider a stripe
\begin{equation}
\cX_{\bound,\rho}=\cX_\bound\cap \{\rho\le - \xi_2 + \bigl(VF^{-1}\bigr)^{\frac{1}{2}}\le 2\rho\}.
\label{15-2-16}
\end{equation}
Then the length of the hop is $\asymp \mu^{-1}\rho^{\frac{1}{2}}$ with the possible perturbation $O(\mu^{-2})$ due to the magnetic drift  (so $\rho \gg \mu^{-2}$ would suffice, but we request $\rho \ge \mu^{-1}$ anyway).

Now, weak magnetic field approach would mean that uncertainty principle is fulfilled after the first hop:
$\rho \times \mu^{-1}\rho^{\frac{1}{2}} \ge h^{1-\delta}$ or equivalently 
$\rho \ge \mu ^{\frac{2}{3}}h^{\frac{2}{3}-\delta}$. Combining with restriction $\rho \ge \mu^{-1}$ we arrive to
\begin{equation}
\rho \ge \bar{\rho}\Def 
\max\bigl(C_0\mu^{-1},\mu ^{\frac{2}{3}}h^{\frac{2}{3}-\delta}\bigr)
\label{15-2-17}
\end{equation}
and our goal is to prove that under this the contribution of $\cX_\bound$ to the remainder does not exceed $C\mu^{-1}h^{-1}$.

To achieve this goal let note first that 
\begin{equation}
\int _{ \cX_{\bound,\rho}\cap \{x_2=\const\}} dx_1 d\xi :da \le C \rho
\label{15-2-18}
\end{equation}
which enables us to estimate 
\begin{equation}
|F_{t\to h^{-1}\tau}\bar{\chi}_T(t)\Gamma (u\,^t\!Q_y \psi)|\le 
C\mu^{-1}h^{-1}\rho \gamma
\label{15-2-19}
\end{equation}
as $|\tau|\le \epsilon$, $Q=Q(x_2,hD_2)\in \S_{h,\rho,\gamma}(\bR)$ and with the symbol supported in $\cX_{\bound,\rho}\cap \{|x_2-\bar{x}_2|\le \gamma\}$, 
$\rho \ge C\max(\gamma,\mu^{-1})$, $\psi=\psi''(\mu x_1)$ and 
$T=T_*=\epsilon \mu^{-1}$.

\begin{remark}\label{rem-15-2-8}
Note that factor $\mu^{-1}$ appears here and in all similar estimates because width $\asymp\rho$ with respect to $\xi_2$ matches to the width $\asymp\mu^{-1}\rho$ with respect to $x_1$ on energy levels below $c$.
\end{remark}

What we need is to investigate propagation until time $T=T^*(\rho)$ and prove that
\begin{gather}
|F_{t\to h^{-1}\tau}\bigl(\bar{\chi}_{T^*}(t)-\bar{\chi}_{T_*}(t)\bigr)
\Gamma (u\,^t\!Q_y\psi)|\le C'h^s
\label{15-2-20}\\
\shortintertext{with}
\int T^{*\,-1}(\rho)\,d\rho\le C
\label{15-2-21}
\end{gather}
then as the main part is given by Tauberian formula with $T=T_*$, remainder does not exceed 
\begin{equation*}
C\mu^{-1}h^{-1}\int T^{*\,-1}(\rho)\,d\rho dx_2\le C\mu^{-1}h^{-1}.
\end{equation*}
In what follows $T^*(\rho)=\rho^{1-\delta'}$  with arbitrarily small exponent $\delta'>0$.

To prove (\ref{15-2-20}) with $T^*(\rho)=\rho^{1-\delta'}$ note first that 
\begin{claim}\label{15-2-22}
If $|\partial_{x_2} VF^{-1}|\le c\rho^{\delta'}$ then propagation in \emph{any\/} time direction remains in the zone
\begin{equation}
\bigl\{\frac{1}{2}\rho\le -\xi_2 + 
(VF^{-1})^{\frac{1}{2}}\le \rho'\bigr\} \cap 
\bigl\{|x_2-\bar{x}_2|\le 2\gamma\bigr\}
\label{15-2-23}
\end{equation}
for time $T=T^*(\rho)$ 
\end{claim}
and 
\begin{claim}\label{15-2-24}
If $|\partial_{x_2} VF^{-1}|\ge \tfrac{1}{2}c\rho^{\delta'}$ then propagation in \emph{an appropriate\/}\footnote{\label{foot-15-3} In which $(VF^{-1})(0,x_2)$ decays; as rotation is counter-clock-wise we take $\pm t>0$ as $\pm  (VF^{-1})(0,x_2)>0$.} time direction remains in (\ref{15-2-23}) for time $T=T^*(\rho)$;
\end{claim}
in both cases $\gamma =C\mu^{-1}T^*(\rho')\le \epsilon\rho$ for sure where 
in the former case  $\rho' =3\rho$ and in the latter case $\rho'=\rho^{1-\delta''}$.

After assertions (\ref{15-2-22}) and (\ref{15-2-24}) are proven, we need to prove that singularities really propagate. Figure~\ref{fig-hop2}(a) shows that as $VF^{-1}=1$ and the movement is strictly circular the length of the hop is exactly 
$2\mu^{-1}\bigl(1- (1-\rho)^2\bigr)^{\frac{1}{2}}$ while in general case it will be of the same magnitude $\mu^{-1}\rho^{\frac{1}{2}}$ as $\rho\ll 1$. But obviously in the most critical zone movement in $x_2$ is not monotone.
However for sure $x_2$ increases with each hop as we bound ourselves with $x_1\le \epsilon' \mu^{-1}\rho$ with small enough constant $\epsilon'$.

The main problem are trajectories which are almost tangent to the boundary. There are two kind of them: with $\rho'\ll 1$ and with $\rho \ll 1$. Trajectories of of the first kind are not actually difficult (see proof of proposition~\ref{prop-15-2-11}). So we start from trajectories of the second kind. 

\begin{proposition}\label{prop-15-2-9}
Consider zone  $\cZ\Def \{\xi_2 \ge - \epsilon \}$. Consider point $(\bar{x},\bar{\xi}_2)\in \cZ$ and its $(\boldgamma,\sigma)$-vicinity $\Omega$ with $\boldgamma =(\gamma_1,\gamma_2)$ and\,\footnote{\label{foot-15-4} After rescaling $x\mapsto \mu x$.}
\begin{align}
&\gamma_1=\sigma = \hbar^{\frac{1}{2}-\delta},
&&\text{as\ \ } \bar{x}_1\ge  \hbar^{\frac{1}{2}-\delta},\label{15-2-25}\\
&\gamma_1  = \sigma =
\bigl(\hbar \gamma^{-\frac{3}{2}}\bigr)^{\frac{1}{2}-\delta}\gamma,
&&\text{as\ \ } \bar{x}_1\asymp \gamma \le  \hbar^{\frac{1}{2}-\delta},
\label{15-2-26}\\
& \gamma_2= \hbar^{\frac{1}{3}-\delta} \label{15-2-27}
\end{align}
and operator $Q=Q(x,hD_2)\in \S_{h,\gamma_1,\gamma_2,\sigma}$ with symbol supported in $\Omega$.

Further, consider point $(\bar{y},\bar{\eta}_2)\in \cZ$ and its $(\boldgamma,\sigma)$-vicinity $\Omega'$ with $\boldgamma$ reintroduced for this point according to \textup{(\ref{15-2-25})}--\textup{(\ref{15-2-27})} and also operator $Q'=Q'(x,hD_2)\in \S_{h,\gamma_1,\gamma_2,\sigma}$ with symbol supported in $\Omega'$. 

Let $T\asymp \mu^{-1}$ and 
\begin{multline}
|x_1-y_1|\gamma_1^{-1}+|x_2-y_2|\gamma_2^{-1}+|\xi_2-\eta_2|\sigma^{-1}\ge \epsilon\\
\forall x\in \Psi_{t,\tau}(\Omega), \ y\in \Omega'\qquad\forall t\in \supp \chi_T
\label{15-2-28}
\end{multline}
where $\Psi_{t,\tau}$ is a Hamiltonian flow with reflections and $\xi_1$ is defined so that $a(x,\xi)=\tau$. Then 
\begin{equation}
F_{t\to h^{-1}\tau}\chi_T(t) Q'U \,^t\!Q'
\label{15-2-29}
\end{equation}
is negligible.
\end{proposition}

\begin{proof}
As we assume $\bar{\xi}_2\ge -\epsilon_0$ we do not need to consider ``short and low'' hops as on figure~\ref{fig-hop2}(b) and the time of the hop is 
$\asymp \mu^{-1}$.

Obviously as $\mu \le h^{-\delta}$ one can apply results of sections~\ref{book_new-sect-2-4} and~\ref{book_new-sect-3-5} \cite{futurebook} and justify our final conclusion.

Let us consider larger $\mu$. To do so we need to understand how small vicinity of Hamiltonian billiard flow with reflections we must take to contain propagation. 

To do so consider propagation in the different zones, step by step, and we do consider $\sigma,\gamma_1,\gamma_2$ not necessarily defined by (\ref{15-2-25})--(\ref{15-2-27}).

\medskip\noindent
(a) First of all, as $x_1\ge \gamma \ge \hbar^{\frac{1}{2}-\delta}$ one can take any $\gamma_j\ge \gamma$ and $\rho_j\ge \gamma$. So, as 
$x_1\ge \hbar^{\frac{1}{2}-\delta}$ we can take 
$\gamma_1= \hbar^{\frac{1}{2}-\delta}$ and other scales different but larger.
 Thus,

\begin{claim}\label{15-2-30}
Let $\bar{x}_1\ge \gamma \ge \hbar^{\frac{1}{2}-\delta}$ and let $\gamma_j\ge \gamma$, $\sigma\ge \gamma$. Let $t$ be such that $\Psi_{t'} (\supp Q)$ does not intersect $\gamma$-vicinity of $\partial X$ as $0<\pm t'<\pm t$ (if $\pm t>0$).

Then (\ref{15-2-29}) is negligible, provided $\Psi_{t,\tau} (\supp Q)$ does not intersect $(\boldgamma,\sigma)$-vicinity of $\supp Q'$.
\end{claim}

We refer to this as \emph{inner propagation\/} (see figure~\ref{fig-prop1}(a)).

\begin{figure}[h!]
\centering
\subfloat[Inner propagation]{\includegraphics[width=0.4\linewidth]{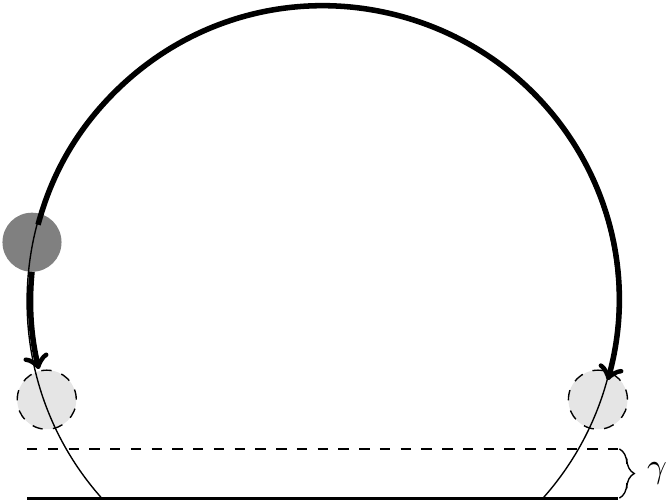}} \qquad
\subfloat[Squeezed inner propagation]{\includegraphics[width=0.4\linewidth]{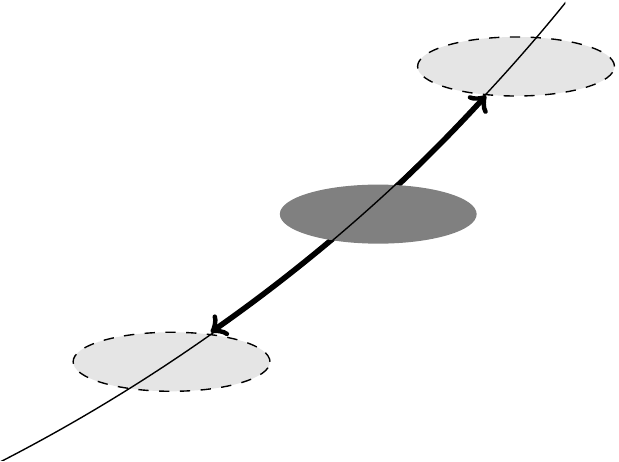}}
\\[5pt]
\subfloat[Squeezed propagation with reflection]{\includegraphics[width=0.6\linewidth]{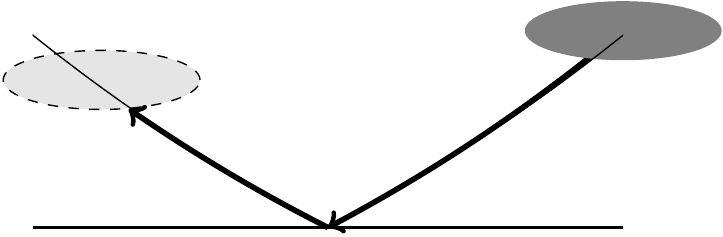}}
\caption{\label{fig-prop1}Inner, squeezed inner (and similar mirror symmetric), squeezed with reflection propagations}
\end{figure}

\medskip\noindent
(b) Consider now zone $C\rho \le x_1\asymp \gamma \le \hbar^{\frac{1}{2}-\delta}$. Let us scale $x_1\mapsto x_1\gamma^{-1}$, $x_2\mapsto x_2\gamma^{-\frac{1}{2}}$ and $\hbar \mapsto \hbar'\Def \hbar\gamma^{-\frac{3}{2}}$. As after original rescaling $x\mapsto \mu x$ we had that derivatives of all coefficients were less than $C\mu^{-1}$ we conclude that after this rescaling they are less then $C\mu^{-1}\gamma^{\frac{1}{2}}$ and after division by $\gamma $ they are still bounded. 

More precisely: we recalled that ``up to perturbation'' operator was 
\begin{equation}
\bar{A}=(\hbar D_1)^2 - (\hbar D_2- x_1)^2 -1
\label{15-2-31}
\end{equation}
with $| \xi_2-1 |\le \rho$; so in the zone in question rescaling as described is justified. And therefore respectively we need to scale $\xi_1\mapsto \xi_1\gamma^{-\frac{1}{2}}$ and 
$\xi_2-1 \mapsto \xi_2 \gamma^{-1}$. 

Now we have rather regular situation and can take 
$\hbar^{\prime\, \frac{1}{2}-\delta}$ vicinity. In other words, we can take 
\begin{gather}
\gamma_1 ,\sigma \ge \hbar^{\prime\, \frac{1}{2}-\delta} \gamma = 
\bigl(\hbar \gamma^{-\frac{3}{2}}\bigr)^{\frac{1}{2}-\delta}\gamma,
\label{15-2-32}\\
\gamma_2  \ge \hbar^{\prime\, \frac{1}{2}-\delta} \gamma ^{\frac{1}{2}}= 
\bigl(\hbar \gamma^{-\frac{3}{2}}\bigr)^{\frac{1}{2}-\delta}
\gamma ^{\frac{1}{2}}
\label{15-2-33}
\end{gather}
and we can replace these right hand expressions by $\hbar^{\frac{1}{2}-\delta}\gamma^{\frac{1}{4}}$ and $\hbar^{\frac{1}{2}-\delta}\gamma^{-\frac{1}{4}}$ respectively and the latter does not exceed its value as $\gamma =\bar{\rho}$ and it is $\hbar^{\frac{1}{3}-\delta}$. This is \emph{squeezed inner propagation\/} (see figure~\ref{fig-prop1}(b)). 

Furthermore, we can apply long-range propagation of section~\ref{book_new-sect-2-4} \cite{futurebook}  and replace condition $x_1\asymp \gamma$ by a weaker condition $\gamma^{1+\delta'}\le x_1\le \gamma^{1-\delta'}$ with $\delta' >0$ small enough. But then we need no more than $C(\delta')$ jumps to reach from $\gamma \asymp \hbar^{\frac{1}{2}-\delta}$ to $\gamma\asymp \hbar^{{2}{3}-\delta}$ or inversely. 

\medskip\noindent
(c) Finally as $x_1\le C\rho$ we can apply the same scaling with $\gamma=\rho$ and note that after rescaling trajectories meet the boundary under angle disjoint from $0$. So we have a standard reflection situation. Scaling back we arrive to the same conclusion as in (b). This is \emph{squeezed propagation with reflection\/} (see figure~\ref{fig-prop1}(c)).
\end{proof}

Repeating $N\asymp\mu T$ times we arrive to

\begin{corollary}\label{cor-15-2-10}
Consider $\mu^{-1}\le T \le \epsilon \rho$. Then conclusion of proposition~\ref{prop-15-2-9} remains true if we redefine 
$\sigma\Def \sigma_\old N$, $\gamma_j \Def \gamma_{j,\old} N$ with 
$N\asymp \mu T$ number of rotations.
\end{corollary}

Now we can prove our main claim

\begin{proposition}\label{prop-15-2-11}
Estimate \textup{(\ref{15-2-20})} holds with $T_*=\epsilon \mu^{-1}$ and $T^*(\rho)=\rho^{1-\delta'}$.
 \end{proposition}

\begin{proof} We consider a pseudo-differential partition of unity. Let $Q'=Q'(x_2,hD_2)\in \S_{h,\rho,\gamma}$ with symbol supported in 
$\epsilon (\gamma, \rho)$-vicinity of $(\bar{x}_2,\bar{\xi_2})$ 
and $\gamma =\mu^{-1}\rho^{\frac{1}{2}}$ satisfy \textup{(\ref{15-2-17})}. 

\medskip\noindent
(i) Consider first $\bar{\xi}_2\le -\epsilon_0$.  

Note first that 
\begin{claim}\label{15-2-34}
As $\xi_2\le -\epsilon $ both operator and boundary value problem are $\xi_2$-microhyperbolic as $x_1\ge 0$.
\end{claim}

Really, $ih^{-1}[A,x_2]= hD_2-\mu x_1\le -\epsilon$.

\medskip

Then due to results of chapter~\ref{book_new-sect-3} \cite{futurebook} $U\,^t\!Q_y$ is negligible as 
$\pm t>0$ and  $\pm (x_2- \bar{x}_2+\epsilon\gamma ) <\pm \epsilon_0  t$, which implies (\ref{15-2-20}).

\medskip\noindent
(ii) Consider now $\bar{\xi}_2\ge -\epsilon_0$.  Note that $\Psi_t$ after exactly one turn moves $x_2$ to the left by at least $\epsilon \rho^{\frac{1}{2}}$; we assume that $\rho \ge C_0\mu^{-1}$ to counter possible perturbations. Further, as $\rho\ge (\mu h)^{\frac{2}{3}-\delta}$ it exceeds $\gamma_2$ expansion due to uncertainty principle. This justifies conclusion of the proposition as $T^*=\epsilon_0 \rho$ as propagation speed with respect to $\xi_2$, $x_2$ does not exceed $C_0$.

\medskip\noindent
(iii) To increase $T^*(\rho)$ notice that we can take $T^*(\rho)=\rho^{1-\delta'}$ unless $\ell\ge \rho^{\delta'}$ where 
\begin{equation}
\ell \Def  \epsilon_0|\nabla_{\partial X} VF^{-1}|
\label{15-2-35}
\end{equation}
as the speed of propagation with respect to $\rho$ does not exceed $\ell$ and dynamics remains in the same $(\ell,\rho)$-element with respect to $(x_2,\xi_2)$. 

Meanwhile as $\ell \ge \rho^{\delta'}$ we can take time direction in which $\rho$ increases and then take $T^*(\rho)=\rho^{1-\delta'}$ anyway.

Uncertainty principle $\rho\ell\ge h^{1-\delta}$ is obviously satisfied in both cases.
\end{proof}

\begin{remark}\label{rem-15-2-12}
Surely one can improve arguments of (iii) and we will do it studying strong magnetic field. However at this moment this leads to no improvements in the remainder estimate.
\end{remark}

Then immediately we arrive to

\begin{corollary}\label{cor-15-2-13}
Contribution of zone $\cX_\bound$ defined by \textup{(\ref{15-1-17})}, \textup{(\ref{15-2-17})}  to the remainder does not exceed $C\mu^{-1}h^{-1}$.
\end{corollary}

\begin{proof}
Contribution of $\cX_{\bound,\rho}$ to the remainder does not exceed
\begin{equation}
C\mu^{-1}\rho h^{-1}\times T^{*,-1}\asymp C\mu^{-1}h^{-1}\rho^{\delta'}
\label{15-2-36}
\end{equation}
and summation over $\rho\ge \bar{\rho}$ (i.e. integration with respect to $\rho^{-1}d\rho$) results in $O(\mu^{-1}h^{-1})$.
\end{proof}

\subsection{Analysis in the transitional zone}
\label{sect-15-2-2-3}

Consider now \emph{transitional zone\/} defined by (\ref{15-1-14}). Obviously we get a rather rough estimate

\begin{proposition}\label{prop-15-2-14}
Contribution of the transitional zone \textup{(\ref{15-1-14})} to the remainder does not exceed 
\begin{equation}
C\bar{\rho}h^{-1}\asymp C\mu^{-1}h^{-1}+ C(\mu h)^{\frac{2}{3}}h^{-1-\delta}.
\label{15-2-37}
\end{equation}
In particular, as $\mu \le h^{-\frac{2}{5}+\delta}$ it does not exceed $C\mu^{-1}h^{-1}$.
\end{proposition}

Now we want to improve this estimate under some non-degeneracy condition invoking $VF^{-1}|_{\partial X}$.

Let us introduce $\ell$-admissible partition with $\ell$ defined by (\ref{15-2-35}).

\begin{proposition}\label{prop-15-2-15} On $\ell$-element with
\begin{equation}
\ell \ge \bar{\ell}\Def \max\bigl(C\mu^{-1}, \mu^{\frac{1}{2}}h^{\frac{1}{2}-\delta}\bigl)
\label{15-2-38}
\end{equation}
expressions \textup{(\ref{15-2-10})} and \textup{(\ref{15-2-11})} are negligible as 
$T_*\le T\le T^*$ with $T_*=\epsilon \mu^{-1}$ and 
$T^*(\ell)=\ell^{1-\delta'}$.
\end{proposition}

\begin{proof}
As in subsection \ref{sect-15-2-2-1} consider shift with respect to $\xi_2$ and it will be $\asymp\ell T$. So, uncertainty principle requests $\ell T\times \ell \ge h^{1-\delta}$ which is exactly our restriction to $\ell$ as 
$T=\epsilon \mu^{-1}$. Meanwhile $x_1\le C_0\mu^{-1} + C\mu^{-1}\ell T$ and it remains less than $C\mu^{-1}$ as $T\le 1$. On the other hand shift with respect to $x_2$ does not exceed $C\mu^{-1} + C\rho^{\frac{1}{2}}T$ in virtue of proposition \ref{prop-15-2-17} below and it remains less than $\epsilon \ell$ as $T=\epsilon \ell ^{1-\delta'}$ unless 
$\rho \ge \ell^{\delta'}$ which is impossible in the transitional zone.

We surely need to keep $\rho \ell \ge h^{1-\delta}$ but one can check easily that $\bar{\rho}\bar{\ell}\ge \mu^{-\frac{1}{2}}h^{\frac{1}{2}-\delta}\ge h^{1-\delta}$.
\end{proof}

\begin{corollary}\label{cor-15-2-16}
(i) Contribution of the part of the transitional zone \textup{(\ref{15-1-14})} 
where \textup{(\ref{15-2-38})} is fulfilled to the remainder does not exceed $C\mu^{-1}h^{-1}$;

\medskip\noindent
(ii) Contribution of the  transitional zone \textup{(\ref{15-1-14})} 
to the remainder does not exceed
\begin{multline}
C\mu^{-1}h^{-1} + \\
C(\mu h)^{\frac{2}{3}}h^{-1-\delta}
\mes_{\partial X} \bigl( \bigl\{x\in \partial X, \ 
|\nabla_{\partial X} V|\le C(\mu h)^{\frac{1}{2}-\delta}\bigr\}\bigr) h^{-1-\delta},
\label{15-2-39}
\end{multline}
(iii) In particular, it does not exceed
\begin{equation}
C\mu^{-1}h^{-1} + C(\mu h)^{\frac{7}{6}}h^{-1-\delta},
\label{15-2-40}
\end{equation}
provided\begin{phantomequation}\label{15-2-41}\end{phantomequation}
\begin{equation}
|\nabla_{\partial X}VF^{-1}|\le \epsilon \implies \pm 
\nabla^2_{\partial X}VF^{-1} \ge \epsilon
\tag*{$\textup{(\ref*{15-2-41})}^\pm$}\label{15-2-41-pm}
\end{equation}
and expression \textup{(\ref{15-2-41})} does not exceed $C\mu^{-1}h^{-1}$ as
$\mu \le h^{-\frac{7}{13}+\delta}$.
\end{corollary}

While either sign in \ref{15-2-41-pm} assures remainder estimate (\ref{15-2-40}), in the future we will need to distinguish between different signs in this condition as dynamics will be different (see example~\ref{ex-15-1-4}).

We finish this subsubsection by 

\begin{proposition}\label{prop-15-2-17}
An ``average'' propagation speed with respect to $x_2$ in $\cX_\trans$ does not exceed $C\bar{\rho}^{\frac{1}{2}}$.
\end{proposition}

\begin{proof}
Proof coincides with one of proposition~\ref{prop-15-2-9} and corollary~\ref{cor-15-2-10}.

Surely trajectories in the squeezed reflection zone are not transversal to boundary anymore after rescaling but we do not need it as instead of $\hbar^{\prime\,\frac{1}{2}-\delta}$-vicinity we can take $\epsilon$-vicinity and appeal f.e. to results of section~\ref{book_new-sect-3-4} \cite{futurebook}.
\begin{figure}[h]
\centering
{\includegraphics[width=0.6\linewidth]{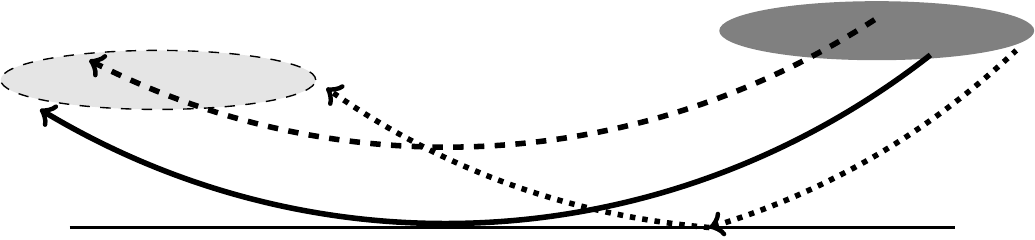}}
\caption{\label{fig-prop2}Squeezed propagation with tangency, reflections or miss the boundary}
\end{figure}
\end{proof}

\subsection{Analysis in the inner zone}
\label{sect-15-2-2-4}

Consider now \emph{inner zone\/} defined by (\ref{15-1-16}). As we know, in this zone we must assume some non-degeneracy condition to get remainder estimate better than $O(\mu h^{-1})$. In Chapter~\ref{book_new-sect-13} \cite{futurebook} we used few of such conditions rendering different remainder estimates; all of them boiled down to the remainder estimate $O(\mu^{-1}h^{-1})$ under different restrictions to $\mu$. 

Consider first propagation of singularities.

\begin{proposition}\label{prop-15-2-18}
Let  $Q\in \S_{h,\rho_1,\rho_2,\gamma_1,\gamma_2}$ and let $\Omega$ be $(\gamma_1,\gamma_2,\rho_1,\rho_2)$-vicinity (with respect to $(x,\xi)$) of its symbol. Then long as trajectories started from $\Omega$ remain in $\cX_\inn$ defined by \textup{(\ref{15-1-16})}, \textup{(\ref{15-2-17})} 
\begin{equation}
U \,^t\!Q_y \equiv U^0 \,^t\!Q_y
\label{15-2-42}
\end{equation}
where $U^0$ is a propagator for operator $A$ considered in $X^0$ the fixed vicinity of $\bar{X}$ (i.e. in the whole plane effectively).
\end{proposition}

\begin{proof}
Proof is trivial as 
\begin{equation}
\rho \ge \bar{\rho}\Def 
\max\bigl(C\mu^{-1},\mu ^{\frac{1}{2}}h^{\frac{1}{2}-\delta}\bigr)
\label{15-2-43}
\end{equation}
as then we can apply Fourier integral operators as described in Chapter~\ref{book_new-sect-13} \cite{futurebook}.

As $\bar{\rho}$ is defined by (\ref{15-2-17}) one needs to apply the same arguments as in the proof of proposition~\ref{prop-15-2-9}. Albeit in that proof trajectories ``comfortably met'' $\partial X$ and now they ``comfortably miss'' it.
\end{proof}

\begin{remark}\label{rem-15-2-19}
This proposition however does not imply that results of section~\ref{book_new-sect-13-3} \cite{futurebook}  are automatically valid in our case as in that section we used partition elements which could be rather large in some directions (sometimes as large as $\asymp 1$ which could be replaced by $h^{\delta'}$ as we apply microlocal standard uncertainty principle rather than logarithmic uncertainty principle). Now boundary prevents us from doing this. 
\end{remark}

Still, however, as $\mu \le h^{-\delta}$ with sufficiently small exponent $\delta>0$ these arguments of section~\ref{book_new-sect-13-3} \cite{futurebook}  remain automatically valid and leaving easy details to the reader arriving to

\begin{proposition}\label{prop-15-2-20}
Under non-degeneracy condition~$\textup{(\ref{book_new-13-3-57})}_m$ \cite{futurebook}  with $\nu=1$ i.e. \begin{phantomequation}\label{15-2-44}\end{phantomequation}
\begin{equation}
\sum_{1\le |\alpha|\le m} |D^\alpha (VF^{-1})|\ge \epsilon
\tag*{$\textup{(\ref*{15-2-44})}_m$}\label{15-2-44-*}
\end{equation}
contribution of the inner zone $\cX_\inn$ to the remainder does not exceed $C\mu^{-1}h^{-1}$ as $\mu \le h^{-\delta}$ with sufficiently small exponent $\delta=\delta(m)>0$.
\end{proposition}

Combining with results in inner and transitional zones where the same contributions were derived even without non-degeneracy condition we conclude that

\begin{theorem}\label{thm-15-2-21}
Let conditions \textup{(\ref{13-1-1})}--\textup{(\ref{13-1-4})}, \textup{(\ref{15-2-2})} and \textup{(\ref{15-2-12})} and non-degeneracy condition~\ref{15-2-44-*} be fulfilled on 
$\supp \psi$ where   $\psi(x)\in \sC_0^\infty(\bR^2)$. 

Then  the remainder does not exceed $C\mu^{-1}h^{-1}$ as $\mu \le h^{-\delta}$ with sufficiently small exponent $\delta=\delta(m)>0$.
\end{theorem}

\begin{proof}
As in the Chapter~\ref{book_new-sect-13}   (see $\textup{(\ref{book_new-13-3-49})}_\vartheta$) \cite{futurebook} the contribution of $\cX_\inn$ to the remainder does not exceed
\begin{equation}
C \mu^{-1} h^{-1}\int_{\cX_\inn} T^{-1}\,dx_2 d\xi_2  + C\mu h^{-1} \mes_2 (\Omega_{2})
\label{15-2-45}
\end{equation}
with the same definition (\ref{book_new-13-3-50}) \cite{futurebook}  of $\Omega_2$; here  $T =\epsilon \mu^{\frac{1}{2}} \rho ^{\frac{1}{2}}$ and $d\xi_2=d\rho$ and integral is bounded as $\rho \le c\mu$.
\end{proof}

In almost all other cases we however need to take remark~\ref{rem-15-2-19} into account. Still there is an exception: arguments linked to evolution with respect to $\xi_2$ work in a bit larger zone, namely $\{(x,\xi): \ \xi_2\le C_0\}$ and thus  in zone $\{x,\ x_1\le C_1\mu^{-1}\}$ (which is a definition of $X_\bound$). 

\subsection{No-critical point case} 

Assume temporarily that condition (\ref{15-2-8}) is fulfilled. Then, as we mentioned,  theorem~\ref{thm-15-2-6} and proposition~\ref{prop-15-2-15} cover this zone. So we need to consider zone 
\begin{equation}
\{x:\  C_0\mu^{-1}\le x_1\le \epsilon\}.
\label{15-2-46}
\end{equation}

However under condition (\ref{15-2-8})  we can take $\asymp 1$ scale with respect to $x_2$ and consider drift in the time direction in which $\xi_2$ increases; then we remain in the zone in question for time $T^* =\epsilon '\mu$. Meanwhile under this condition shift with respect to $\xi_2$ is observable as $|t|\ge T_*=\epsilon \mu^{-1}$. This leads us to to conclusion that 
\begin{claim}\label{15-2-47}
Under assumption (\ref{15-2-8}) contribution of this (\ref{15-2-46}) zone to the Tauberian remainder does not exceed $C\mu^{-1}h^{-1}$
\end{claim}
(i.e. the same as the contribution of zone $\{x:\ x_1\le C_0\mu^{-1}\}$ albeit factor $\mu^{-1}$ now comes as $T^{*\,-1}$ rather than as the width of the zone).

So we arrive to

\begin{theorem}\label{thm-15-2-22} 
Let conditions \textup{(\ref{13-1-1})}--\textup{(\ref{13-1-4})}, \textup{(\ref{15-2-2})} and \textup{(\ref{15-2-12})} and non-degeneracy condition~\textup{(\ref{15-2-8})} be fulfilled on $\supp \psi$ where   $\psi(x)\in \sC_0^\infty(\bR^2)$. 

Then contribution of zone 
$\{x:\ \dist (x,\partial X)\le \epsilon_1\}$ to the remainder does not exceed $C\mu^{-1}h^{-1}$ as $\mu \le h^{\delta-1}$ with arbitrarily small exponent $\delta>0$.
\end{theorem}

\paragraph{Non-degenerate critical point case.}
Consider now case when non-degeneracy condition (\ref{15-2-41}) is fulfilled. Then obviously contribution of zone $\{x_1\le C_0\mu^{-1}\}$ to the remainder does not exceed 
\begin{equation}
C\mu^{-1}h^{-1}+ C(\mu h)^{\frac{1}{2}-\delta}h^{-1}
\label{15-2-48}
\end{equation}
where the second term comes from zone 
\begin{equation}
\bigl\{x: \ x_1\le C_0\mu^{-1}, \ 
|W_{x_2}|\le C_0\max\bigl(\mu^{-1}, (\mu h)^{\frac{1}{2}-\delta}\bigr)
\bigr\}
\label{15-2-49}
\end{equation}
and this is not as good as (\ref{15-2-40}) because $\bar{\rho} \ll 1$.

Now in the zone 
\begin{equation}
\bigl\{x: \  C_0\mu^{-1}\le x_1\le \epsilon_1, \ 
|W_{x_2}|\ge C_0\max\bigl(\mu^{-1}, (\mu h)^{\frac{1}{2}-\delta}\bigr)
\bigr\}
\label{15-2-50}
\end{equation}
we can use $T^*=\epsilon \mu \ell$ with $\ell = \epsilon |W_{x_2}|$ and we can even take $T^* = \ell^{1-\delta'}$; then contribution of this zone to the remainder does not exceed $C\mu^{-1}h^{-1}$. 

Meanwhile contribution of zone 
\begin{equation}
\bigl\{x: \  C_0\mu^{-1}\le x_1\le \epsilon_1, \ 
|W_{x_2}|\le C_0\max\bigl(\mu^{-1}, (\mu h)^{\frac{1}{2}-\delta}\bigr)
\bigr\}
\label{15-2-51}
\end{equation}
to the Tauberian remainder under this condition does not exceed 
\begin{equation}
C\mu^{-1}h^{-1}+ C\mu (\mu h)^{\frac{1}{2}-\delta}h^{-1}.
\label{15-2-52}
\end{equation}

So we arrive to 

\begin{theorem}\label{thm-15-2-23} 
Let conditions \textup{(\ref{13-1-1})}--\textup{(\ref{13-1-4})}, \textup{(\ref{15-2-2})} and \textup{(\ref{15-2-12})} and non-degeneracy condition~\textup{(\ref{15-2-41})} be fulfilled on $\supp \psi$ where   $\psi(x)\in \sC_0^\infty(\bR^2)$. 

Then contribution of zone 
$\{x:\ \dist (x,\partial X)\le \epsilon_1\}$ to the remainder does not exceed  \textup{(\ref{15-2-52})}; in particular, it does not exceed $C\mu^{-1}h^{-1}$ as $\mu \le h^{\delta-\frac{1}{5}}$.
\end{theorem}

This is exactly the same estimate as was derived in Chapter~\ref{book_new-sect-13} \cite{futurebook}  without boundary under condition 
\begin{equation}
|\nabla VF^{-1}|+|\nabla^2 VF^{-1}|\ge \epsilon.
\label{15-2-53}
\end{equation}

Now let us consider derivatives of $VF^{-1}$ with respect to $x_1$ which leads to the drift in the direction $x_2$.

Our goal is to prove

\begin{theorem}\label{thm-15-2-24}
(i) Let conditions \textup{(\ref{13-1-1})}--\textup{(\ref{13-1-4})}, \textup{(\ref{15-2-2})} and \textup{(\ref{15-2-12})} and non-degeneracy condition~\textup{(\ref{15-2-7})} be fulfilled on $\supp \psi$ where   $\psi(x)\in \sC_0^\infty(\bR^2)$.  Then the remainder does not exceed
\begin{equation}
C\mu^{-1}h^{-1} + C\mu^2 h^{-\delta}
\label{15-2-54}
\end{equation}
as $\mu \le h^{\delta-\frac{1}{2}}$; in particular remainder does not exceed $C\mu^{-1}h^{-1}$ as $\mu \le h^{\delta-\frac{1}{3}}$.

\medskip\noindent
(ii) Additionally assume that condition \textup{(\ref{15-2-41})} is also fulfilled on $\supp \psi$. Then the remainder does not exceed
\begin{equation}
C\mu^{-1}h^{-1} + C\mu^{\frac{5}{2}}h^{\frac{1}{2}-\delta}
\label{15-2-55}
\end{equation}
as $\mu \le h^{\delta-\frac{1}{2}}$; in particular remainder does not exceed $C\mu^{-1}h^{-1}$ as $\mu \le h^{\delta-\frac{3}{7}}$.
\end{theorem}

\begin{proof}
(i) First assume that 
\begin{equation}
\rho \ge \bar{\rho} \Def \max\bigl(C\mu^{-1}, \mu^{\frac{1}{2}}h^{\frac{1}{2}-\delta}\bigr).
\label{15-2-56}
\end{equation}
Then along the whole trajectory 
$x_1\ge \mu^{-1}\rho \ge \mu^{-\frac{1}{2}}h^{\frac{1}{2}-\delta}$ and using $\mu^{-1}h$-Fourier integral operators we can reduce operator to the canonical form of section~\ref{book_new-sect-13-2} \cite{futurebook}. We can then notice that the shift with respect to ``new'' $x_1$ is $\mu^{-1} T k$ where $k \asymp |\nabla V F^{-1}|$ but the scale with respect to the dual variable is $\mu^{-1}\rho$ and we need to write uncertainty principle 
\begin{equation}
\mu^{-1}k T \times \mu^{-1}\rho \ge \mu^{-1}h^{1-\delta}
\label{15-2-57}
\end{equation}
or plugging $T\asymp \mu^{-1}$ and $k \asymp 1$ we get 
\begin{equation}
\rho \ge \bar{\rho}\Def \max\bigl(C\mu^{-1}, \mu^2 h^{1-\delta}\bigr)
\label{15-2-58}
\end{equation}
and this restriction is stronger than (\ref{15-2-56}) which in turn is stronger than (\ref{15-2-17}). 

Therefore 
\begin{claim}\label{15-2-59}
Under condition (\ref{15-1-8}) contribution to the remainder of $\cX_\inn$ defined by (\ref{15-1-16}), (\ref{15-2-58}) does not exceed $C\mu^{-1}h^{-1}$.
\end{claim}
Meanwhile with this choice of $\bar{\rho}$ contribution to the remainder of $\cX_\trans$ defined by (\ref{15-1-14}) does not exceed 
$C \bar{\rho} h^{-1}$ which is exactly the second term in (\ref{15-2-54}). 

Statement (i) is proven. 

\medskip\noindent
(ii) To prove (ii) we need to note that in $\cX_\trans$ shift with respect to $\xi_2$ is observable as $T\asymp \mu^{-1}$ and $|W_{x_2}|\asymp \ell$ satisfies (\ref{15-2-38}). Then the contributions to the remainder of all zones save $\{\cX_\trans, \ |W_{x_2}|\le \bar{\ell}\}$ are $O(\mu^{-1}h^{-1})$ and under condition (\ref{15-2-41}) the contribution of this zone does not exceed $C\bar{\rho}\bar{\ell}h^{-1}$ which is exactly the second term in (\ref{15-2-55}). 
\end{proof}

\begin{remark}\label{rem-15-2-25}
(i) Conditions (\ref{15-1-8}) and (\ref{15-2-41}) are fulfilled generically at the boundary.

\medskip\noindent
(ii) Obviously in the setting of theorem~\ref{thm-15-2-24} the inner zone provides the worst contribution to the remainder and we will need to improve it in frames of condition (\ref{15-2-24}) using the strong magnetic field approach there.
\end{remark}

\section{From Tauberian to  magnetic Weyl formula}
\label{sect-15-2-3}

Now our goal is to pass from Tauberian with $T=\epsilon \mu^{-1}$ to magnetic Weyl formula and estimate remainder $\R^\MW$. We also consider extended Weyl formula and estimate remainder $\R^\W_\infty$.

\begin{theorem}\label{thm-15-2-26}
Let $\psi\in \sC^\infty (\bar{X})$ be a fixed function with a compact support contained in the small vicinity of $\partial X$ and let conditions  \textup{(\ref{13-1-1})}--\textup{(\ref{13-1-4})},
\begin{gather}
F\ge \epsilon _0
\label{13-2-1}\\
\shortintertext{and}
V\le -\epsilon _0\qquad \forall x\in B(0,1);
\label{13-3-45}
\end{gather}
be fulfilled there. Further, let condition \textup{(\ref{15-2-6})} be fulfilled.

Then

\medskip\noindent
(i) Under condition \textup{(\ref{15-2-8})} both $\R^\W_\infty$ defined by \textup{(\ref{15-2-14})} and $\R^\MW$ defined by \textup{(\ref{15-2-15})}
do not exceed $C\mu^{-1}h^{-1}$ as $\mu \le h^{\delta-1}$;

\medskip\noindent
(ii) Under condition \ref{15-2-44-*} both $\R^\W_\infty$ and $\R^\MW$ do not exceed $C\mu^{-1}h^{-1}$ as $\mu \le h^{-\delta}$;

\medskip\noindent
(iii) Under condition \textup{(\ref{15-2-41})} both $\R^\W_\infty$ and $\R^\MW$ do not exceed \textup{(\ref{15-2-52})} as $\mu \le h^{\delta-1}$; in particular remainder does not exceed $C\mu^{-1}h^{-1}$ as $\mu \le h^{\delta-\frac{1}{5}}$;

\medskip\noindent
(iv) Under condition \textup{(\ref{15-2-7})} both $\R^\W_\infty$ and $\R^\MW$ do not exceed \textup{(\ref{15-2-54})} as $\mu \le h^{\delta-1}$; in particular remainder does not exceed $C\mu^{-1}h^{-1}$ as $\mu \le h^{\delta-\frac{1}{3}}$;

\medskip\noindent
(v) Under conditions \textup{(\ref{15-2-7})} and \textup{(\ref{15-2-41})} both $\R^\W_\infty$ and $\R^\MW$ do not exceed \textup{(\ref{15-2-55})} as 
$\mu \le h^{\delta-1}$;  in particular remainder does not exceed $C\mu^{-1}h^{-1}$ as $\mu \le h^{\delta-\frac{3}{7}}$.
\end{theorem}

\begin{proof}
(a) To go from $\R^\T$ to $\R^\W_\infty$ is easy: using condition (\ref{13-3-45}) (i.e. $|V|\asymp 1$) we can apply the standard results of Chapters~\ref{book_new-sect-4} and~\ref{book_new-sect-7} \cite{futurebook}  after rescaling $x\mapsto \mu x$.

Going from $\R^\W_\infty$ to $\R^\MW$ is more subtle. Note first that we can drop all terms with $m>0$ in (\ref{15-2-14}). Therefore only surviving terms are those with $h^{-2}(\mu h)^{2n}$ with integration over $X$ and 
$h^{-1}(\mu h)^{2n}$ with integration over $\partial X$. 

\medskip\noindent
(b) Consider first terms with integration over $X$. 

 Note first that under condition (\ref{15-2-8}) their sum is equal to 
\begin{equation}
\int_X \cN^\MW (x,\mu h, \tau)\psi(x)\, dx
\label{15-2-60}
\end{equation}
modulo $O\bigl(h^{-2}(\mu h)^s\bigr)$; actually we go in the opposite direction: from expression (\ref{15-2-60}) to its decomposition into powers of 
$\hbar=\mu h$. 

Consider expression (\ref{15-2-60}) under condition (\ref{15-1-8}); again we can replace $\uptheta (\tau-V -  (2j+1)F\mu h)$ by derivative with respect to $x_1$ (of high order and with smooth coefficient) of smooth $\sC^s$-function; so integrating by parts we again arrive to the sum of the terms in question plus the similar integral over the boundary; however the latter contains at least one extra integration with respect to $\tau$ so we get
\begin{equation}
(\mu h)^{2n} h^{-2}\int \bigl(\tau -V -(2j+1)F \mu h\bigr)_+^r  \phi_{n,r}(x)\, ds_g
\label{15-2-61}
\end{equation}
where $r\ge 1$. If we replace in the latter term summation with respect to $j$ by integration the error will not exceed $C(\mu h)^{r+1}h^{-2}$ i.e. $C\mu^2$ which is less than (\ref{15-2-54}). 

On the other hand, under condition (\ref{15-2-41}) the error in question will not exceed $C(\mu h)^{r+1\frac{3}{2}}h^{-2}$ which is less than (\ref{15-2-55}). So, in (iiv) and (v) we also can replace terms with integration over $dx$ into $\cN^\MW$, may be changing $\kappa'_{n0}$. 

The same arguments albeit without integration with respect to $x_1$ work under condition (\ref{15-2-41}) alone; however we gain only factor $(\mu h)^{\frac{1}{2}}$.

The similar arguments work in frames of (ii) as well.

\medskip\noindent
(c) Now case (i) becomes the most complicated as many terms should be taken into account. We apply the following trick: consider the same operator albeit we replace $D_2$ everywhere by $D_3$:
\begin{multline}
\bar{A}\Def \sum_{j,k} \bar{P}_j g^{jk}(x_1,x_2)\bar{P}_k +V(x_1,x_2),\\
\bar{P}_1= hD_1,\ \bar{P}_2=hD_3-\mu V_2(x_1,x_2)
\label{15-2-62}
\end{multline}
Then it will affect in $\R^\W_\infty$ only terms with $m\ge 1$ we do not care about. But the problem remains microhyperbolic in the variable $x_2$ and therefore everything works as it should. We need to consider then shifts with respect to $\xi_2$ only and therefore only averaging with respect to $x_2$ is needed. 

Note then that what we get instead of $\int e(x,x,\tau)\psi \,dx$ is
\begin{equation}
(2\pi h)^{-1}\int \mathsf{e}(x_1,x_1;x_2,\xi_3,\tau)\psi \,dx_1\,dx_2 d\xi_3
\label{15-2-63}
\end{equation}
where $\mathsf{e}(x_1,y_1;x_2,\xi_2,\tau)$ is a Schwartz kernel for spectral projector for $1$-dimensional operator. Here integration with respect to $x_3$ is not needed. Neither is needed integration with respect to $\xi_2$ as we pass from Tauberian expression with $T=\epsilon \mu^{-1}$ for $\mathsf{e}(\ldots)$ to $\mathsf{e}(\ldots)$ itself. Finally we change in (\ref{15-2-63}) $\xi_3$ by $\xi_2$. 

Also we can replace everywhere (save $V_2$)  $x_1$ by $0$ while $V_2$ we replace by 
\begin{equation}
\bar{V}_2 = V(0,x_2) + (\partial_{x_1}V_2)(0,x_2)x_1;
\label{15-2-64}
\end{equation}
it will not affect essential boundary terms. But then boundary terms in $\R^\W_\infty$ together must match to a boundary term in $\R^\MW$ . This proves (i) completely.

\medskip\noindent
(d) Exactly the same arguments work for (ii)--(v); we do not use condition (\ref{15-2-7}) at all and in $\cX_\bound$ defined in terms of $(\xi_3,x_2)$ we consider shifts after hops with respect to $x_3$, so we do not need to integrate over $x_3$. Without condition (\ref{15-2-21}) we use the trivial 
$O(\bar{\rho} h^{-1})$ estimate for contribution of $\cX_\trans$.

Under condition (\ref{15-2-21}) in $\cX_\trans$ we consider shifts with respect $\xi_2$ and again integration over $x_3$ and $\xi_2$ is not needed.
\end{proof}

\chapter{Strong magnetic field}
\label{sect-15-3}

In this section we consider a case of the strong magnetic field when the results of the previous section are not as sharp as we want (so the remainder is not $O(\mu^{-1}h^{-1})$). As under different assumption it happens under different restrictions to $\mu$, we consider separately different cases. 

\section{Most non-degenerate case}
\label{sect-15-3-1}
If condition (\ref{15-2-8}) is fulfilled, we need to consider  only the case of \emph{very strong magnetic field\/}
\begin{equation}
h^{\delta-1}\le \mu \le h^{-1}
\label{15-3-1}
\end{equation}
and then operator is $x_2$-microhyperbolic.  Therefore as $\psi=\psi'\psi''_\mu$ with $\psi'=\psi'(x_2)\in \sC^\infty_0$ and $\psi''_\gamma=\psi''(x_1\gamma^{-1})$\,\footnote{\label{foot-15-5} Averaging with respect to $x_1$ is not needed at all.} with $\psi''_\gamma$ supported in $|x_1|\le C_0\gamma$, $\gamma= \mu^{-1}$,
\begin{gather}
|F_{t\to h^{-1}\tau}\chi_T(t)\Gamma \psi u|\le C\mu^{-1}h^{-2} T \bigl(\frac{h}{T}\bigr)^s, \label{15-3-2}\\
\shortintertext{and then}
|F_{t\to h^{-1}\tau}\bar{\chi}_T(t)\Gamma \psi u|\le C\mu^{-1}h^{-1}
\label{15-3-3}
\end{gather}
as $h\le T\le T^*=\epsilon_0$ where $\mu^{-1}$ comes as a measure of $X_\bound$ and therefore 
\begin{equation}
\R^\T \le C\mu^{-1}h^{-1}.
\label{15-3-4}
\end{equation}
Further advancing method of successive approximations with unperturbed operator\footnote{\label{foot-15-6} In comparison with (\ref{15-2-63}), (\ref{15-2-64}) we freeze at $y$ not $(0,y_2)$ at this moment.}
\begin{multline}
\bar{A}\Def \sum_{j,k}\bar{P}_j g^{jk}(y)\bar{P}_j +V(y),\\
\bar{P}_j= hD_j -\mu V_j(y) -\mu \sum_k (x_j-y_j)(\partial_k V_j)(y)
\label{15-3-5}
\end{multline}
we see that the first term results in expression (\ref{15-2-63}) of the magnitude $C\mu^{-1}h^{-2}$ while any next term acquires factor $h$ in the corresponding power and thus does not exceed the remainder estimate.

So, under conditions (\ref{15-2-8}) and (\ref{15-3-1}) and indicated $\psi$ the remainder does not exceed $C\mu^{-1}h^{-1}$ while the principal part is given by the Tauberian expression for the first term in the successive approximation method. 

On the other hand, if we take $\psi''\in \sC^\infty_0(\bR^+)$ (supported in $(\frac{1}{2},1)$) and $\gamma \ge \mu^{-1}$ we can take 
\begin{equation}
T^*\asymp \min\bigl( (h^{-1}\gamma)^{1+\sigma}, \mu \gamma^{1-\sigma}\bigr).
\label{15-3-6}
\end{equation}
Really, we could take $T^*\asymp \gamma^{1-\sigma}$ as 
\begin{equation}
\gamma\ge \bar{\gamma}\Def C_0\mu^{-1}+ C_0\mu^{-\frac{1}{2}}h^{\frac{1}{2}-\delta}
\label{15-3-7}
\end{equation}
which completely covers the case $\mu \le h^{\delta-1}$. Otherwise we can arrive to this case scaling $x\mapsto x\zeta$, $\mu \mapsto \mu \zeta$, $h\mapsto h\zeta^{-1}$ with 
$\zeta = \min\bigl( 1,(\mu^{-1}h)^{\frac{1}{2}}, (\mu^{-1}h\gamma^{-2})^s\bigr)$ with large $s$.

Therefore the contribution of the strip $\{x,\ x_1\asymp \gamma\}$ to the remainder does not exceed 
\begin{equation*}
C\gamma h^{-1}\Bigl( (h^{-1}\gamma)^{-1-\sigma}+ 
\mu^{-1} \gamma^{-1+\sigma}\Bigr) 
\end{equation*}
and hence contribution of $\{x_1\le \epsilon\}$ does not exceed this expression integrated over $\gamma^{-1}d\gamma$ resulting in
\begin{equation*}
C h^{-1}\Bigl( (h^{-1}\gamma)^{-1-\sigma}\gamma\bigr|_{\gamma=\mu^{-1}}+ 
\mu^{-1} \gamma^{\sigma}|_{\gamma=1}\Bigr) \asymp C\mu^{-1}h^{-1}.
\end{equation*}
So, we can take $\psi$ fixed function rather than scaled with respect to $x_1$.

However let us partition it into functions supported in $\{x_1\le 2\bar{\gamma}\}$ and in $\{x_1\ge \bar{\gamma}\}$. Then such expression in the latter case is not affected by the presence of the boundary resulting in the same expression but with approximation term calculated for the whole space.

However in the former case the presence of the boundary should be taken into account. Let us use again the method of successive approximation but use as unperturbed operator one with $x_1$ set to $0$ everywhere save  in the linear part of magnetic field; thus unperturbed operator is
\begin{equation}
\bar{A}\Def h^2D_1^2+(\mu x_1- hD_2)^2-W(0,x_2).
\label{15-3-8}
\end{equation}
Again as the main part of asymptotics is of magnitude $\bar{\gamma}h^{-2}$ and each next term acquires factor $\bar{\gamma}$, so only first two terms need to be considered.

So, let us consider the second term; we claim that calculating this term one does not need to take a boundary into account. Really, as perturbation vanishes at the boundary one needs to kill $x_1$ before restricting to the boundary, but it can be done only by commutator and then factor $h$ rather than $\bar{\gamma}$ appears. It is not enough but if we plug instead of $\psi''_{\bar{\gamma}}$ function $\psi''_\gamma$ with $C_0\mu ^{-1}\le \gamma\le \bar{\gamma}$ and $\supp \psi''$ disjoint from $0$ we acquire factor $(\mu \gamma)^{-s}$ and then contribution to the error is $C\gamma h^{-1} (\mu \gamma)^{-s}$ and it boils down to $C\mu^{-1}h^{-1}$ after summation with respect to $\gamma$. 

Then we get the final answer as the sum of two terms: one is for operator (\ref{15-3-5}) albeit with calculation (before taking $\Gamma$) in the whole plane i.e.
\begin{equation}
h^{-2} \int_X \cN^\MW (x,\mu h)\psi(x)\,dx
\label{15-3-9}
\end{equation}
and the second one for operator (\ref{15-3-8}) but in half-plane and subtracting the same expression for the same operator albeit in the whole plane we arrive to 
\begin{equation}
\int_X \Bigl( \int \mathsf{e}_1(x_1,x_1;x_2,\xi_2,0)\,d\xi_2 - h^{-2} \cN^\MW_0 (x,\mu h) \Bigr)\psi(x)\,dx
\label{15-3-10}
\end{equation}
where $\mathsf{e}_1(x_1,y_1;x_2,\xi_2,0)$ is the Schwartz kernel of one-dimensional operator 
\begin{equation}
h^2D_1^2+(\mu x_1- \xi_2)^2-W(0,x_2)
\label{15-3-11}
\end{equation}
and $\cN^\MW_0$ refers to $\cN^\MW$ calculated for the same operator (\ref{15-3-8}).

We will transform operators (\ref{15-3-8}), (\ref{15-3-11}) and silently (without changing notations) transform $\mathsf{e}(\ldots)$ and $\cN^\MW_0$.

Changing $x_1\mapsto x_1\zeta$ and $\xi_2\mapsto \xi_2\zeta \mu$ we acquire factor $\mu \zeta$ and get instead of the first term in (\ref{15-3-10})
\begin{equation}
(2\pi h)^{-1}\zeta \mu 
\int \mathsf{e}_1(x_1,x_1;x_2,\xi_2,0)\psi'(x_2)\psi''(x_1  \zeta) \,dx_1\,dx_2 d\xi_2
\label{15-3-12}
\end{equation}
meanwhile transforming operator into
\begin{equation}
h^2\zeta^{-2} \Bigl(  D_1^2+h^{-2} \mu^2\zeta^4( x_1- \xi_2)^2-W(0,x_2)h^{-2}\zeta^{-2}\Bigr).
\label{15-3-13}
\end{equation}
One can drop a factor in the front of operator and select 
$\zeta =\mu^{-\frac{1}{2}}h^{\frac{1}{2}}$ thus resulting in the answer
\begin{equation}
(2\pi )^{-1} \mu ^{\frac{1}{2}}h^{-\frac{1}{2}}
\int \mathsf{e}_1(x_1,x_1;x_2,\xi_2,0)\psi'(x_2)\psi''(x_1\mu ^{\frac{1}{2}}h^{-\frac{1}{2}}) \,dx_1\,dx_2 d\xi_2
\label{15-3-14}
\end{equation}
and in operator
\begin{equation}
D_1^2+( x_1- \xi_2)^2-\mu ^{-1} h^{-1}W(0,x_2).
\label{15-3-15}
\end{equation}
However we need to subtract from (\ref{15-3-14}) also transformed the second term in (\ref{15-3-10}). We can then tend $\zeta\to +\infty$ (the error will be negligible) and the total difference will tend to 
\begin{equation}
h^{-1}\int_{\partial X}  \cN^\MW_{*,\bound} (x_2,\mu h)\psi(x)\,ds_g.
\label{15-3-16}
\end{equation}
So, the final answer is
\begin{equation}
h^{-2}\int_X \cN^\MW(x,\mu h)\psi(x)\,dx 
+h^{-1}\int_{\partial X}  \cN^\MW_{*,\bound} (x_2,\mu h)\psi(x)\,ds_g
\label{15-3-17}
\end{equation}
with $*=\D,\N$ and we arrive to 

\begin{theorem}\label{thm-15-3-1}
Let $\psi\in \sC^\infty (\bar{X})$ be a fixed function with a compact support contained in the small vicinity of $\partial X$ and let conditions  \textup{(\ref{13-1-1})}--\textup{(\ref{13-1-4})}, \textup{(\ref{13-2-1})}  and  non-degeneracy condition  \textup{(\ref{15-2-8})} be fulfilled on $\supp\psi$. 

Then formula \textup{(\ref{15-3-17})} gives $\N$ with an error $\R^\MW = O(\mu^{-1}h^{-1})$ as $\mu \le h^{-1}$.
\end{theorem}

\section{Generic case. Analysis in inner zone}
\label{sect-15-3-2}

Now we are interested to improve results of the previous section in the generic case i.e. when both conditions (\ref{15-2-7}) and (\ref{15-2-41}) are fulfilled. Then in virtue of theorem~\ref{thm-15-2-26}(v) we can assume that
\begin{equation}
\mu \ge h^{\delta -\frac{3}{7}}.
\label{15-3-18}
\end{equation}

We start from the simpler analysis in $\cX_\inn$. In this case as $\rho\ge \bar{\rho}=(\mu h)^{\frac{2}{3}}h^{-\delta}$ we need to consider operator without boundary condition; however presence of the boundary as we remember manifests itself through uncertainty principle; we should take 
$T_*= h^{1-\delta'}\ell^{-2}$ or $T_*=\mu h^{1-\delta'}\rho^{-1}$ whatever is smaller and \emph{at this moment\/} we are interested only in the zone where 
$T_*\ge \epsilon\mu^{-1}$.

Actually we can take here effectively even $\delta'=0$ in the following sense: note that
\begin{equation}
|F_{t\to h^{-1}\tau} \phi_T(t)\Gamma\bigl( U\,^t\!Q_y\bigr)|\le C\mu^{-1}h^{-1}\rho \ell \times
\bigl(\mu T  +1\bigr) 
\label{15-3-19}
\end{equation}
as $\phi \in \sC^\infty_0([-1,1])$, $T\ge h$ and 
\begin{equation}
|F_{t\to h^{-1}\tau} \chi_T(t)\Gamma\bigl( U\,^t\!Q_y\bigr)|\le C\mu^{-1}h^{-1}\rho \ell \times
\bigl(\mu T  +1) \bigl(\frac{T}{T_*}\bigr)^{-s}
\label{15-3-20}
\end{equation}
as $\chi \in \sC^\infty_0([-1,-\frac{1}{2}]\cap[\frac{1}{2},1] )$ and 
$T_*\le T \le T^*$ with unspecified at this moment $T^*$ and 
\begin{equation}
T_*= \min\bigl(\mu h \rho^{-1}, \ell^{-2}\bigr)h=h  \times
\left\{ \begin{aligned}
&\mu \rho^{-1} \qquad &&\text{as\ \ } \rho > \mu \ell^2,\\
&\ell^{-2} &&\text{as\ \ } \rho < \mu \ell^2.
\end{aligned}
\right.
\label{15-3-21}
\end{equation}
Recall that $Q=Q(x_2,hD_2)$ is an $(\ell,\rho)$ admissible element with $\rho\ell\ge h^{1-\delta}$.

Really for one winding (i.e. $T=\epsilon\mu^{-1}$) estimate (\ref{15-3-19}) is obvious, and we need to take in account $N\asymp \mu T +1$ windings.

Estimates (\ref{15-3-19}) and (\ref{15-3-20}) imply that
\begin{equation}
|F_{t\to h^{-1}\tau} \phi_T(t)\Gamma\bigl( U\,^t\!Q_y\bigr)|\le C\mu^{-1}h^{-1}\rho \ell \times
\bigl(\mu T_*  +1\bigr) 
\label{15-3-22}
\end{equation}
as $\phi \in \sC^\infty_0([-1,1])$ and  $T_*\le T \le T^*$.

\begin{remark}\label{rem-15-3-2}
Therefore from the point of view of the remainder estimate, rather than the final formula we need to take in account $N^*\asymp (\mu T_* +1)$ windings but in the main part of asymptotics still $(\mu T_* h^{-\delta'}+1)$ windings should be taken in account.
\end{remark}

Let us discuss $T^*$. As average propagation speed with respect to $x$   does not exceed $C_0\mu^{-1}$ and propagation speed with respect to $\xi_2$ does not exceed $C_0\ell$, our dynamics remains in the same $(\ell, \rho)$-element for time 
\begin{equation}
T^*=\epsilon \min \bigl(\mu \ell, \rho\ell^{-1}\bigr)
\label{15-3-23}
\end{equation}
and $T_*:T^*\asymp h\rho^{-1}\ell^{-1}$. 

Therefore contribution  of $(\ell,\rho)$-element to the Tauberian remainder with 
$T\ge (T_*h^{-\delta'}+ \epsilon \mu^{-1})$ does not exceed
\begin{equation}
C\mu^{-1}h^{-1}\rho \ell\bigl(\mu T_*  +1\bigr)\times T^{*\,-1}.
\label{15-3-24}
\end{equation}
Note that in zone $\ell \le \mu^{-\frac{1}{2}}\rho^{\frac{1}{2}}$ we can reset 
$\ell=\ell(\rho)\Def \mu^{-\frac{1}{2}}\rho^{\frac{1}{2}}$ thus covering the whole zone with a fixed magnitude of $\rho$ by a single element. Also note that 
$\ell(\rho)\rho = \mu^{-\frac{1}{2}}\rho^{\frac{3}{2}}\ge 
\mu ^{\frac{1}{2}}h\ge h^{1-\delta}$ as $\rho\ge \bar{\rho}$. 

Therefore 
\begin{claim}\label{15-3-25}
Contribution of $(\ell(\rho),\rho)$-element to the Tauberian remainder with $T=\mu h^{1-\delta'}\rho^{-1}$ does not exceed 
\begin{equation*}
C\mu^{-1}h^{-1}\ell(\rho)^2 \bigl(\frac{\mu^2h}{\rho}+1\bigr)\asymp 
C+ C\mu^{-2}h^{-1}\rho
\end{equation*}\vglue-10pt
\end{claim}
and summation with respect to $\rho \in [\bar{\rho}, \epsilon \mu]$ results in 
\begin{equation}
C|\log h|+ C\mu^{-1}h^{-1}.
\label{15-3-26}
\end{equation}
So, 
\begin{claim}\label{15-3-27}
Contribution of zone $\{\rho\ge\bar{\rho},\ \mu \ell^2 \le \rho\}$ to the Tauberian remainder with $T=\mu h^{1-\delta'}\rho^{-1}$ does not exceed (\ref{15-3-26}).
\end{claim}

\begin{remark}\label{rem-15-3-3}
We need to keep $\ell(\rho)\ge C\mu^{-1}$ but this is definitely case as $\mu \ge h^{-\frac{2}{5}}$ and we are ensured in this by condition (\ref{15-3-18}).
\end{remark}

Similarly, consider case $\mu\ell^2\ge \rho\ge \bar{\rho}$. Then automatically 
$\ell \ge \ell(\bar{\rho})$ and $\ell\rho \ge h^{1-\delta}$ and 
\begin{claim}\label{15-3-28}
In this zone we can reset $\rho =\rho (\ell)= \mu \ell^2$. 
\end{claim}
Really, to avoid collision with $\cX_\trans$ we select time direction in which $\rho$ (and thus $\xi_2$) increases. It is possible because as long as $|t|\le \epsilon \mu \ell$ we remain in the same $\epsilon\ell$-vicinity. Again $\ell \ge C \mu^{-1}$ in virtue of condition (\ref{15-3-18}).

Then 
\begin{claim}\label{15-3-29}
Contribution of $(\ell ,\rho(\ell))$-element to the Tauberian remainder with $T= h^{1-\delta'}\ell^{-2}$ does not exceed 
\begin{equation*}
C\mu^{-1}h^{-1}\rho(\ell) \bigl(\frac{\mu h}{\ell^2}+1\bigr)\asymp 
C+ C\mu^{-1}h^{-1}\ell^2
\end{equation*}
\end{claim}
and summation with respect to $\ell \in [\bar{\ell}, \epsilon ]$ results in 
(\ref{15-3-26}) where $\bar{\ell}= \mu^{-\frac{1}{2}}\bar{\rho}^{\frac{1}{2}}$.
So, 
\begin{claim}\label{15-3-30}
Contribution of zone $\{\rho\ge\bar{\rho},\ \mu \ell^2 \ge \rho\}$ to the Tauberian remainder with $T=\mu h^{1-\delta'}\rho^{-1}$ does not exceed (\ref{15-3-26}).
\end{claim}

\begin{figure}[h!]
\centering
\subfloat[$\mu \ell^2<\rho$]{\includegraphics[width=0.4\linewidth]{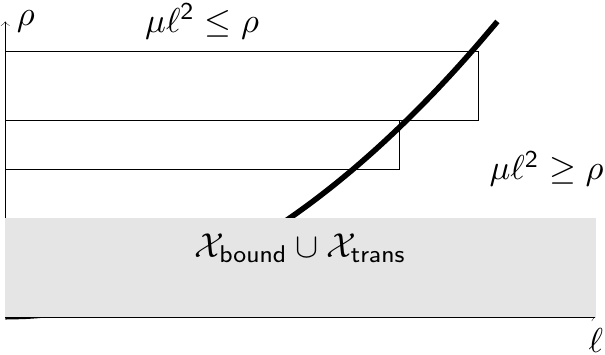}} \qquad
\subfloat[$\mu\ell^2>\rho$]{\includegraphics[width=0.4\linewidth]{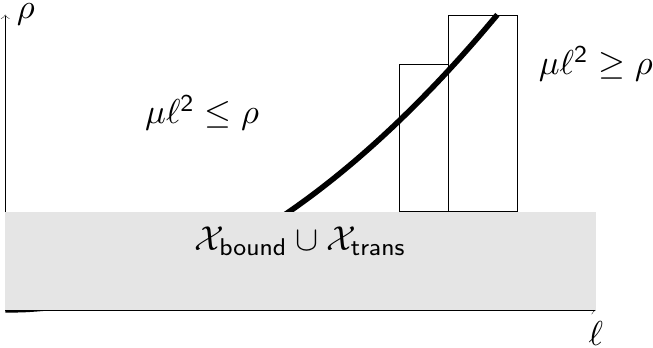}}
\caption{\label{fig-zones} Partition of two zones in $\cX_\inn$}
\end{figure}

So we arrive to

\begin{proposition}\label{prop-15-3-4}
Under assumptions \textup{(\ref{15-2-41})} and \textup{(\ref{15-2-7})}
contribution of $\cX_\inn$  to the Tauberian remainder with $T=\mu h^{1-\delta'}\rho^{-1}$ does not exceed \textup{(\ref{15-3-26})}.
\end{proposition}

We are completely happy with this estimate unless $\mu \ge h^{-1}|\log h|^{-1}$ and we need to derive contribution of much more troublesome zone $\cX_\trans$ before thinking if we should improve it.

\section{Generic case. Analysis in transitional zone}
\label{sect-15-3-3}

Let us consider zone $\cX_\trans$. Recall that its contribution to the Tauberian remainder with $T=\epsilon \mu^{-1}$ (and thus with $T=h^{1-\delta}$) does not exceed $C\bar{\rho}h^{-1}=C(\mu h)^{\frac{2}{3}}h^{-1-\delta}$ in the general case and $C\bar{\rho}\bar{\ell}h^{-1}=C(\mu h)^{\frac{7}{6}}h^{-1-\delta}$ under condition (\ref{15-2-41}) and therefore (as we assume (\ref{15-2-41})) we should consider only case 
\begin{equation}
\mu \ge h^{\delta-\frac{7}{13}}.
\label{15-3-31}
\end{equation}

Definitely $\cX_\trans$ is leaner than $\cX_\inn$ (albeit as 
$\mu \ge h^{\delta-1}$ it is thick enough to eliminate $\cX_\bound$). However the main problem there is that in the current settings we are not aware of any lower bound of the propagation speed and the only bound we know is an upper bound $C_0\bar{\rho}^{\frac{1}{2}}$ in both directions.

Therefore there is no mechanism except drift with respect to $\xi_2$ to break periodicity and we must take
\begin{gather}
T_*=\epsilon h \ell^{-2}
\label{15-3-32}
\shortintertext{and}
T^*= \epsilon \bar{\rho}^{-\frac{1}{2}}\ell.
\label{15-3-33}
\end{gather}
Sure we need to have $T_*\le T^*$, i.e. $\ell\ge \bar{\ell}_0$ with
\begin{equation}
\bar{\ell}_0\Def \bar{\rho}^{\frac{1}{6}}h^{\frac{1}{3}}= 
(\mu h)^{\frac{1}{9}} h^{\frac{1}{3}-\delta}.
\label{15-3-34}
\end{equation}
Obviously $\bar{\ell}_0\le \bar{\ell}=(\mu h)^{\frac{1}{2}}h^{-\delta}$ unless $\mu \le h^{-\frac{1}{7}}$ the case we are not interested in.

\begin{remark}\label{rem-15-3-5}
Note that it it requires time $t^*\asymp \bar{\rho}\ell^{-1}$ to pierce through $\cX_\trans$ and $t^* \le T^*$ as $\ell\le \bar{\ell}$ which is the case we are interested in. Otherwise we would be able to increase $T^*$ further.
\end{remark}

So, in the same manner as before contribution of $(\ell, \bar{\rho})$-element with $\ell\ge \bar{\ell}_0$ to the Tauberian remainder with any 
$T\ge T_*h^{-\delta}$~\footnote{\label{foot-15-7} As $T_*h^{-\delta}\ge T^*$ we reset $T=T^*$.} by
\begin{equation}
C\mu^{-1}\bar{\rho}\ell h^{-1}\bigl(\mu T_*+1) \times T^{*\,-1}=
C\bar{\rho}^{\frac{3}{2}}\ell^{-2} + C\mu^{-1}h^{-1}\bar{\rho}^{\frac{3}{2}}
\label{15-3-35}
\end{equation}
and summation over $\ell\ge \bar{\ell}_0$ results in
\begin{equation}
C\bar{\rho}^{\frac{3}{2}}\bar{\ell}_0^{-2} + C\mu^{-1}h^{-1}\bar{\rho}^{\frac{3}{2}}|\log h|\asymp
C (\mu h)^{\frac{7}{9}} h^{-\frac{2}{3}-\delta}.
\label{15-3-36}
\end{equation}
On the other hand, contribution of $(\bar{\ell}_0,\bar{\rho})$-element to the Tauberian remainder with $T\ge \epsilon\mu^{-1}$ does not exceed
\begin{equation}
C\bar{\rho}\bar{\ell}_0h^{-1}
\label{15-3-37}
\end{equation}
which coincides with the (\ref{15-2-37}).

So we arrive to
\begin{proposition}\label{prop-15-3-6}
Under condition \textup{(\ref{15-2-41})} contribution of $\cX_\trans$ with any $T\ge  T_*h^{-\delta}=h^{1-\delta}\ell^{-2}$ (reset to $T\ge h^{1-\delta}$ as either $\ell\le \bar{\ell}_0$ or $\ell \ge \bar{\ell}$; see also \footref{foot-15-7}) does not exceed \textup{(\ref{15-3-36})}.

In particular, it does not exceed $C\mu^{-1}h^{-1}$ as 
$\mu \le h^{\delta-\frac{5}{8}}$.
\end{proposition}

\begin{remark}\label{rem-15-3-7}
(i) Surely we need to keep $\ell \ge C_0\mu^{-1}$ thus requiring 
$\bar{\ell}_0\ge C_0\mu^{-1}$ i.e. $\mu \ge h^{-\frac{2}{5}}$ but this is a case. 

\medskip\noindent
(ii) We also need to keep an upper bound to speed greater than $C_0\mu^{-1}$, i.e. $\bar{\rho}^{\frac{1}{2}}\ge C_0\mu^{-1}$ i.e. $\mu \ge h^{-\frac{1}{4}}$ but this is also the case.

\medskip\noindent
(iii) Further, we need to keep $\bar{\ell}_0 \bar{\rho}\ge h^{1-\delta}$ i.e.
$\mu\ge h^{-\frac{1}{7}}$, but this is again the case.
\end{remark}

\begin{figure}[h!]
\centering
\includegraphics[scale=1]{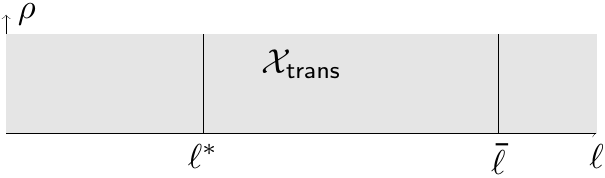}
\caption{\label{fig-zone-trans} Zone $\cX_\trans$} 
\end{figure}

Can we increase $T^*$ in these arguments? We need to do it only under condition (\ref{15-3-18}).  Then as we will prove later

\begin{proposition}\label{prop-15-3-8}
Let 
\begin{claim}\label{15-3-38}
Dirichlet boundary condition be given on $\partial X$.
\end{claim}
Then in the transitional zone $\{|\rho|\le\bar{\rho}\}$ as
\begin{equation}
\mu \ge h^{-\frac{1}{4}-\delta}
\label{15-3-39}
\end{equation} 
average propagation speed in one direction (direction of hops when we have chosen time) is bounded by $C\bar{\rho}^{\frac{1}{2}}$  and in the opposite direction it is bounded by $C\mu^{-1}$.
\end{proposition}
Note, that condition (\ref{15-3-39}) ensures that $\bar{\rho}^{\frac{1}{2}}\gg \mu^{-1}$.

So, assume that (\ref{15-3-38}) holds. Then selecting the time direction for given partition element so that \emph{$\ell$ increases along hops} we can chose
\begin{equation}
T^* \asymp 1.
\label{15-3-40}
\end{equation}
Really, we need to leave $\ell$-element and come back. This comeback includes \underline{either} going in the direction opposite of hops with the speed not exceeding $C_0\mu^{-1}$ and requires time $\asymp \mu \ell$ which is larger than (\ref{15-3-40}) as $\ell\ge \mu^{-1}$ and $\mu \ge h^{-\frac{1}{4}}$ \underline{or} leaving zone $\ell\le \epsilon$  which requires at least time $\asymp 1$\,\footnote{\label{foot-15-8} As $X$ is a bounded domain and $F\asymp 1$ in $X$ this would mean run-around $X$ along boundary which as we see later is impossible under certain assumptions.}.

Actually $T^*\asymp 1$ is better (larger) than given by (\ref{15-3-33})  we used before only as $\ell\le \bar{\rho}^{\frac{1}{2}}$, but this is only zone we need to care about.

So, contribution of the given element to the remainder does not exceed the left-hand expression of (\ref{15-3-35}) which now becomes
\begin{equation}
C\mu^{-1}\bar{\rho} \ell h^{-1} \times \bigl(\mu h \ell^{-2} +1\bigr)\times \min\bigl( \bar{\rho}^{\frac{1}{2}}\ell^{-1}, 1\bigr)\asymp
C \bar{\rho}  \ell^{-1} +C \mu^{-1}h^{-1}\bar{\rho} \ell. 
\label{15-3-41}
\end{equation}
Summation of the second  term in the right-hand expression over partition results in $O(\mu^{-1}h^{-1})$. Summation of the first term in the right-hand expression over partition with $\ell\ge \bar{\ell}_1$ results in the same term calculated as $\ell=\bar{\ell}_1$ and coincides with 
$C\bar{\rho}\bar{\ell}_1 h^{-1}$ which also estimates the contribution of zone $\{\ell\le \bar{\ell}_1\}$. Here again we chose $\bar{\ell}_j$ from condition $T_*\le T^*$ i.e. now we replace $\bar{\ell}_0$ defined by (\ref{15-3-34}) by 
\begin{equation}
\bar{\ell}_1 \Def h^{\frac{1}{2}} \ll \bar{\ell}_0
\label{15-3-42}
\end{equation}
resulting in 
\begin{equation}
C\bar{\rho}h^{-\frac{1}{2}}= C\mu^{\frac{2}{3}}h^{\frac{1}{6}-\delta}.
\label{15-3-43}
\end{equation}
Therefore (with pending proposition~\ref{prop-15-3-8}) we arrive 

\begin{proposition}\label{prop-15-3-9}
Assume that  Dirichlet boundary condition is given on $\partial X$. Then under assumption \textup{(\ref{15-2-40})} the total contribution of $\cX_\trans$ to the Tauberian remainder  does not exceed \textup{(\ref{15-3-43})}.

In particular, it does not exceed $C\mu^{-1}h^{-1}$ as $\mu \le h^{\delta-\frac{7}{10}}$ and it is always has an extra factor  $\mu^{\frac{2}{3}} h^{\frac{7}{6}-\delta}$ in comparison with $h^{-1}$.
\end{proposition}

Can we do better than this? Yes, as we can take 
\begin{equation}
T^* \asymp  \mu \ell.
\label{15-3-44}
\end{equation}
provided $\xi_2$ and $\ell$ increase in the same time direction (on given $\ell$-element) as $F_{12}<0$.

Really, then selecting such time direction we remain in 
$\cX_\trans \cap\cX_\inn$ and as we reach $\ell\asymp 1$ dynamics is already in $\cX_\inn$ where speed in both directions is $O(\mu^{-1})$.

Assuming that the critical point of $W\Def VF^{-1}|_{\partial X}$ is $x_2=0$ we note that hops are to the left (as $F_{12}<0$). So, we select time direction of the sign of $(-x_2)$, so $\ell \asymp |x_2|$ increases.

Meanwhile those expressions have the same sign: $d\xi_2/dt$, $-\partial _{x_2}W$ and  $(-\partial^2 _{x_2}W)x_2$ and our condition means that these expressions are positive, i.e. condition $\textup{(\ref{15-2-41})}^-$ is fulfilled which matches to cases displayed at Figure~\ref{fig-traj3}(b),(d).
So, under this condition we can select $T^*$ according to (\ref{15-3-44}), therefore replacing $\bar{\ell}_1$ defined by (\ref{15-3-42}) by 
\begin{equation}
\bar{\ell}_2 \Def \mu^{-\frac{1}{3}}h^{\frac{1}{3}-\delta}\ll \bar{\ell}_1
\label{15-3-45}
\end{equation}
(as $\mu \ge h^{-\frac{1}{2}-\delta}$) from condition $T_*\le T^*$ and arrive to remainder estimate 
\begin{equation}
C\bar{\rho}\bar{\ell}_2 h^{-1}+ C\mu^{-1}h^{-1}= 
C\mu^{\frac{1}{3}}h^{-\delta}+C\mu^{-1}h^{-1}.
\label{15-3-46}
\end{equation}
So we have proven 

\begin{proposition}\label{prop-15-3-10}
Assume that  Dirichlet boundary condition is given on $\partial{X}$. Then under assumption $\textup{(\ref{15-2-41})}^-$ the total contribution of $\cX_\trans$ to the Tauberian remainder  does not exceed \textup{(\ref{15-3-46})}.

In particular, it does not exceed $C\mu^{-1}h^{-1}$ as 
$\mu \le h^{\delta-\frac{3}{4}}$ and it has an extra factor  
$\mu^{\frac{1}{3}} h^{1-\delta}$ in comparison with $h^{-1}$ (as
$h^{\delta-\frac{3}{4}}\le \mu \le h^{-1} $).
\end{proposition}

Finally, assume that hops and magnetic drift have the same direction near point in question which is the case under condition 
\begin{equation}
|\partial_{x_2} VF^{-1}|\le \epsilon \implies 
\partial _{x_1}VF^{-1}\ge\epsilon_0;
\tag*{$\textup{(\ref*{15-2-7})}^+$}\label{15-2-7+}
\end{equation}
compare with original (\ref{15-2-7}).

Then according to proposition~\ref{prop-15-3-8} there can be no roll-back, only run-around with the path of the length $\asymp 1$. This corresponds to Figure~\ref{fig-traj3}(a),(b).

However if condition $\textup{(\ref{15-2-41})}^-$ is also fulfilled (so we are in frames of Figure~\ref{fig-traj3}(b)) fast run-around along boundary is not possible either, so the run around must contain the segment of the length $\asymp 1$ inside $\cX_\inn$ (even inside $\{x, x_1\ge \epsilon\}$) but the speed here is $O(\mu^{-1})$ and therefore $T^*\asymp \mu $. Then $T_*\le T^*$ defines 
\begin{equation}
\bar{\ell}_3\Def \mu^{-\frac{1}{2}}h^{\frac{1}{2}-\delta}\ll \bar{\ell}_2;
\label{15-3-47}
\end{equation}
this leads to remainder estimate $C\bar{\rho}\bar{\ell}_3h^{-1}=\mu^{\frac{1}{6}}h^{\frac{1}{6}-\delta}$ which is $O(\mu^{-1}h^{-1})$ as $\mu \le h^{\delta-1}$ and $O(h^{-\delta})$ otherwise. Thus we have proven

\begin{proposition}\label{prop-15-3-11}
Assume that  Dirichlet boundary condition is given on $\partial X$. Then under assumptions $\textup{(\ref{15-2-41})}^-$ and \ref{15-2-7+}  the total contribution of $\cX_\trans$ to the Tauberian remainder  does not exceed $C\mu^{-1}h^{-1}$ as $\mu \le h^{\delta-1}$ and $Ch^{-\delta}$ otherwise.
\end{proposition}

\begin{conjecture}\label{conj-15-3-12}
The above Tauberian estimates hold with $\delta=0$.
\end{conjecture}

To prove this conjecture we must reload $\cX_\bound$ and $\cX_\inn$ and to derive there the same remainder estimates as we proved in $\cX_\inn$ but with $\bar{\rho}= (\mu h)^{\frac{2}{3}}$. It does not look difficult as we need just to rescale $x_\mapsto x\gamma$ with $\gamma =h(\rho/\bar{\rho})^{-k}$ with arbitrarily large $k$. However then $h\mapsto \hbar = (\rho/\bar{\rho})^{-k}$ and it is not small enough to keep number of reflections below $\hbar^{-N}$. Really, we need to consider at least $\mu h\ell^{-2}$ number of reflections which means $\mu h\ell^{-2}\le \hbar^{-s}$ and as $\hbar$ goes to $1$ we need to allow $\ell\ge (\mu h)^{\frac{1}{2}}$ only!

Another approach would be based on logarithmic uncertainty principle (albeit instead of $h^{-\delta}$ logarithmic factor $|\log h|^l$ would appear) but we do not have here theory similar to one of section~\ref{book_new-sect-2-3} \cite{futurebook}.

\section{Propagation of singularities in transitional zone}
\label{sect-15-3-4}

Let us study propagation of singularities in $\cX_\trans$; later we assume that the Dirichlet boundary condition is given. We will use the technique developed section~\ref{book_new-sect-3-4} \cite{futurebook}. However there is a large difference: in section~\ref{book_new-sect-3-4} \cite{futurebook} we treated basically trajectories with a single tangent point while here we need to treat $\mu T$ such points in one shot.

Consider real function $\psi' (x,\xi,t,\tau)$ and by Weierstrass theorem replace it modulo  $\bigl(\tau -a(x,\xi)\bigr)$ by a linear with respect to $\xi_1$ function
\begin{multline}
\psi '(x,\xi,t,\tau) = \psi_0(x,\xi_2,t,\tau)+\psi_1 (x,\xi_2,t,\tau)\xi_1=\\
\psi (x,\xi,t,\tau) + \alpha (x,\xi,t,\tau) \bigl(\tau -a (x,\xi)\bigr).
\label{15-3-48}
\end{multline}
Then as coefficient at $\xi_1^2$ in $a(x,\xi)$ is $1$
\begin{equation}
\{\tau -a,\psi\} = \{\tau -a,\psi\} + \alpha_1 \bigl(\tau -a (x,\xi)\bigr)
\label{15-3-49}
\end{equation}
is at most quadratic with respect to $\xi_1^2$ and therefore the following identities hold: 
\begin{multline*}
( (hD_t-A)u,\psi^\w u)= (u, (hD_t-A)\psi^\w u)+
i h(hD_1u,\psi^\w u)_{\partial X}-\\
\shoveright{i h(u,h D_1\psi^\w u)_{\partial X},}\\[3pt]
\shoveleft{( \psi^\w u,  (hD_t-A)u)= (u, \psi^\w(hD_t-A) u) +
ih (\psi_1^\w u, (hD_t-A)u)_{\partial X}}
\end{multline*}
implying identity
\begin{multline}
2h^{-1}\Re i( (hD_t-A)u,\psi^\w u) = -( i h^{-1}[(hD_t-A),\psi^\w ] u,u)- \\[3pt]
(hD_1u,\psi^\w u)_{\partial X}+ (u,h D_1\psi^\w u)_{\partial X}+
(\psi_1^\w u, (hD_t-A)u)_{\partial X}.
\label{15-3-50}
\end{multline}
In this identity terms containing $(hD_t-A)u$ will be either negligible or under our control anyway; 
$[(hD_t-A),\psi^\w ]\equiv -ih \{\tau-a,\psi\}^\w $ modulo smaller terms and only boundary terms
\begin{equation}
-(hD_1u,\psi^\w u)_{\partial X}+ (u,h D_1\psi^\w u)_{\partial X}
\label{15-3-51}
\end{equation}
need special analysis. Under Dirichlet boundary condition we rewrite them as
\begin{gather}
-(hD_1u,\psi^\w_1 hD_1 u)_{\partial X}
\label{15-3-52}\\
\intertext{and under Neumann boundary condition we rewrite them as}
(u,h^2 D_1^2\psi^\w_1 u)_{\partial X}=
-(u,(hD_t-A)\psi^\w_1 u)_{\partial X}+(u,(hD_t-A')\psi^\w_1 u)
\label{15-3-53}
\end{gather}
where $A'=A-h^2D_1^2$ and the difference is that in the former case definiteness  of this quadratic form is ensured by $\psi_1$ having definite sign but in the latter case not. 

\begin{remark}\label{rem-15-3-13}
(i) The same difference would manifest itself as we would try to prove results of section~\ref{book_new-sect-3-4} \cite{futurebook} for wave equation under Dirichlet and Neumann boundary condition. The truth is that the former satisfies Lopatinski condition while the latter does not: uniformity breaks in the tangent zone. 

\medskip\noindent
(ii) The same difference manifests itself through different behavior of eigenvalues $\lambda_{\D,j}(\eta)$ and $\lambda_{\N,j}(\eta)$ as 
$\eta\to +\infty$: while they both tend to $(2j+1)\mu h$ the former are tending to it from above and the latter from below. This implies that under Neumann boundary condition at least for model operator some singularities propagate in the direction opposite to hops and this propagation is faster than magnetic drift. Such difference will be the most transparent  in the case of very strong and superstrong magnetic field in section~\ref{sect-15-4}.
\end{remark}

So, we assume that Dirichlet boundary condition is given on $\partial X$. We select 
\begin{gather}
\psi_1\ge 0
\label{15-3-54}
\intertext{and thus we should request}
\{\tau -a, \psi'\} +\beta (\tau -a) \le 0 \qquad \text{with \ \ }\beta = \beta(x,t,\xi_2,\tau).
\label{15-3-55}
\end{gather}
We take originally 
\begin{equation}
\psi'=\upchi (\phi ),\qquad \phi =\phi(x,\xi)
\label{15-3-56}
\end{equation}
where $\sC^\infty(\bR)\ni \upchi$ is a function of the same type as in section~\ref{book_new-sect-3-4} \cite{futurebook}: namely supported in $(-\infty,0]$ and with $\upchi'<0$ on $(-\infty,0)$.

Then as $\tau-a(x,\xi)=0$ has two real roots $\pm \eta\ne 0$ we conclude that 
\begin{equation}
\psi_1=\frac{1}{2\eta}\Bigl(\upchi (\phi (\xi_1=\eta))-
\upchi (\phi(\xi_1=-\eta))\Bigr)\ge 0
\label{15-3-57}
\end{equation}
provided 
\begin{equation}
-\partial _{\xi_1} \phi \ge \epsilon_0.
\label{15-3-58}
\end{equation}
It does not satisfy (\ref{15-3-54}) as $\tau -a'<0$ but we will handle this in the same way as in section~\ref{book_new-sect-3-4} \cite{futurebook}). 

Now, modulo terms of the type $\beta \cdot (\tau-a)$
\begin{multline}
\{\tau-a, \psi\} \equiv \{\tau-a,\psi'\} \equiv  
-\upchi'  (\phi) \cdot 
\{\tau-a, \phi\}\equiv \\
-\upchi_1^2  (\phi)  \{\tau-a, \phi\}\equiv -\hat{\psi_1}\,^2 \{\tau-a, \phi\}
\label{15-3-59}
\end{multline}
as $\upchi_1 =\sqrt{-\upchi'}$ is a smooth function (under correct choice of $\upchi$) and $\hat{\psi}$ is again linear with respect to $\xi_1$ function, 
\begin{equation}
\hat{\psi}\equiv \upchi_1(\phi) (\{\tau-a, \phi\})^{\frac{1}{2}}
\label{15-3-60}
\end{equation}
and we used Weierstrass theorem again.  We want 
\begin{equation}
\{\tau-a, \phi\}\ge \epsilon_0.
\label{15-3-61}
\end{equation}

\begin{remark}\label{rem-15-3-14}
Conditions (\ref{15-3-61}) and (\ref{15-3-58}) mean respectively that $\phi$ increases along trajectories as $x_1>0$ and at reflections.
\end{remark}

Let us recall that instead of (\ref{15-3-54}) we actually have a weaker inequality
Namely instead of (\ref{15-3-54}) we have 
\begin{equation}
\psi_1 \ge -\upchi_2 (\phi )^2\upchi_3 (\tau-a')^2
\tag*{$\textup{(\ref*{15-3-54})}'$}\label{15-3-54-'}
\end{equation}
with both $\upchi_2$ and $\upchi_3$ smooth functions supported in $(-\infty,0]$ and notice that both operator $(hD_t-A)$ and Dirichlet boundary problem for it are elliptic as $\tau-a'<0$ and we can apply elliptic arguments there.

So, repeating arguments of section~\ref{book_new-sect-3-4} \cite{futurebook}) we conclude that if 
\begin{gather}
\WF^{s+1}\bigl((hD_t-A)u\bigr) \cap \bigl\{\phi^+ <\varepsilon\bigr\}
\cap\{0\le t\le T\}
=\emptyset,\label{15-3-62}\\[2pt]
\WF^{s}\bigl(u|_{\partial X}\bigr) \cap \bigl\{\phi^+ <\varepsilon\bigr\} \cap\{0\le t\le T\} =\emptyset, \label{15-3-63}\\[2pt]
\bigl(\WF^{s}(u)\cup \WF^{s}(hD_1u)\bigr) \cap \bigl\{\phi^+ <\varepsilon\bigr\}
\cap\{t=0 \}=\emptyset\label{15-3-64}
\shortintertext{and}
\bigl(\WF^{s-\sigma}(u)\cup \WF^{s-\sigma}(hD_1 u)\bigr) \cap 
\{\phi^+ <\varepsilon\} \cap\{0\le t\le T\}=\emptyset,\label{15-3-65}\\[2pt]
\WF^{s-\sigma} (hD_1 u|_{\partial X}) \cap \{\phi^+ <\varepsilon\}\cap
\{0\le t\le T\}=\emptyset \label{15-3-66}
\shortintertext{then}
\bigl(\WF^{s}(u)\cup \WF^{s}(hD_1 u)\bigr) \cap \{\phi^+ <0\} \cap
\{0\le t\le T\}=\emptyset,\label{15-3-67}\\[2pt]
\WF^{s} (hD_1 u|_{\partial X}) \cap \{\phi^+ <0\}\cap\{0\le t\le T\}
=\emptyset .\label{15-3-68}
\end{gather}
Here $\sigma>0$ is a sufficiently small exponents, $\epsilon >0$ is an arbitrarily small constant and $T>0$ is an arbitrary constant,
\begin{equation}
\phi^\pm (x_1,x_2,\xi_2,t,\tau)= \phi|_{\xi_1=\pm \eta}, \qquad 
\eta =(\tau -a')^{\frac{1}{2}}
\label{15-3-69}
\end{equation}
and $\WF^{s}$ are defined in terms of pseudo-differential operators $b(x,t,hD_2,hD_t)$.

As we can plug $(\phi-\varepsilon)$ instead of $\phi$ we by induction can get rid off assumptions (\ref{15-3-65}), (\ref{15-3-66}) (assuming that $u$ is temperate). 

We also can rescale $x\mapsto T x$ (and $h\mapsto hT^{-1}$, $\mu\mapsto \mu T$) thus replacing $\epsilon$ by $\epsilon T$ when we come back and we can also consider large $T$. 

As $a=\xi_1^2+(\xi_2-\mu x_1^2)+V$ we select
\begin{equation}
\phi = \kappa t + (\mu x_2-\xi_1).
\label{15-3-70}
\end{equation}
In more general case we select
\begin{equation}
\phi = \kappa t + \mu x_2 + F_{12}p_1.
\label{15-3-71}
\end{equation}
Note that (\ref{15-3-49}) is fulfilled and also 
\begin{equation}
\{\tau-a ,\phi\}=\kappa -\partial_{x_1}V; 
\label{15-3-72}
\end{equation}
and (\ref{15-3-61}) becomes $\kappa - \partial_{x_1}V\ge \epsilon$;
so is fulfilled in two cases: 
\begin{gather}
\kappa=C_0\label{15-3-73}\\
\shortintertext{and}
\partial_{x_1}V\ge \epsilon_0, \qquad |\kappa|\le \epsilon_1\label{15-3-74}
\end{gather}
which allows us to prove respectively that ``$(\mu x_2-\xi_1)$'' propagates with a speed bounded from below by $-C_0$ and, as $\partial_{x_1}V\ge \epsilon_0$ by
$\epsilon_1$. To get rid off $\xi_1$ we must to pass from $\phi^+$ to $\phi$ but we need to notice that $\phi^+\le \phi + c_0$.

Then we arrive to

\begin{theorem}\label{thm-15-3-15}
Let Dirichlet boundary condition be given on $\partial X$ and let $F_{12}<0$ i.e. hops go to the left. Let $\chi \in \sC_0^\infty ([\frac{1}{2},1])$. Then 

\medskip\noindent 
(i) For $C_0\le T\le T^*$  
\begin{gather}
|F_{t\to h^{-1}\tau} \chi_T(t) \psi _2(x_2) U \psi_1(y_2)| \le Ch^s
\label{15-3-75}\\
\shortintertext{as}
x_2\le -C_0\mu^{-1}T +y_2\quad 
\forall x_2\in \supp\psi_2, \forall y_2 \in \supp \psi_1;\label{15-3-76}
\end{gather}

\medskip\noindent 
(ii) Under condition \textup{(\ref{15-3-47})} for $C_0\le T\le T^*$  \textup{(\ref{15-3-75})} holds as 
\begin{equation}
x_2\le \epsilon_0\mu^{-1}T +y_2\quad
\forall x_2\in \supp\psi_2, \forall y_2\in \supp\psi_1.
\label{15-3-77}
\end{equation}
\end{theorem}

This theorem implies proposition~\ref{prop-15-3-8} and thus justifies the improved results of the previous subsection.

\section{Calculations}
\label{sect-15-3-5}

Now we need to move from Tauberian to more explicit expressions for the principal part of asymptotics.

First of all recall that the Tauberian expression for 
$\Gamma \bigl(e(.,.,0)\,^t\!Q_y\bigr)$ is
\begin{equation}
h^{-1}\int^0_{-\infty} F_{t\to h^{-1}\tau}\bar{\chi}_T(t) \Gamma \bigl( U\,^t\!Q_y\bigr)\,d\tau
\label{15-3-78}
\end{equation}
and if we replace $\bar{\chi}_T(t)$ by $\chi_T(t)$~\footnote{\label{foot-15-9} Where as usual $\bar{\chi},\chi\in \sC^\infty_0([-1,1])$ and $\bar{\chi}=1$, $\chi=0$ on $[-\frac{1}{2},\frac{1}{2}]$ and later we will decompose 
$\bar{\chi}_T$ in the sum of $\bar{\chi}_{\bar{T}}(t)$ and $\chi_{2^{-n}T}(t)$}
we will get
\begin{equation}
T^{-1}\int^0_{-\infty} F_{t\to h^{-1}\tau}\tilde{\chi}_T(t) \Gamma \bigl( U\,^t\!Q_y\bigr)\,d\tau
\label{15-3-79}
\end{equation}
where $\tilde{\chi}(t)=it^{-1}\chi(t)$.

Let us list different cases depending on our assumptions

\subsection{General discussion}
\label{sect-15-3-5-1}

As $Q$ is supported in $\cX_\trans$ an absolute value of this expression does not exceed
\begin{equation}
C\mu^{-1}\bar{\rho}\ell h^{-1}T^{-1}\times \mu T = C\bar{\rho}h^{-1}
\label{15-3-80}
\end{equation}
as  $T\ge \epsilon_0\mu^{-1}$ and therefore an approximation error should not exceed 
\begin{equation}
C\bar{\rho}\ell h^{-1} \times Th^{-1} \Delta
\label{15-3-81}
\end{equation}
where $\Delta$ is the size of perturbation and we know that \emph{as we replace $V(x_1,x_2)$ by $V(x_1,y_2)$\/}
\begin{equation}
\Delta = C_0\ell (\updelta x_2) + C_0(\updelta x_2)^2=
C_0\ell (\mu^{-1}+ \mathsf{v} T) + 
C_0(\mu^{-1}+ \mathsf{v}T)^2 
\label{15-3-82}
\end{equation}
as $\updelta x_2 = (\mu^{-1}+ \mathsf{v} T) $ and in the first term factor $\ell$ comes as bound for $|\partial_{x_2}V|$. Here $\mathsf{v}\asymp\bar{\rho}^{\frac{1}{2}}$ is an upper bound for propagation speed in $\cX_\trans$. However  we will justify that under Dirichlet boundary condition we can take \emph{in these calculations\/} $\mathsf{v}\asymp\mu^{-1}$. The difference between $\bar{\rho}^{\frac{1}{2}}=(\mu h)^{\frac{1}{3}-\delta}$ increases as $\mu $ increases.

Neglecting the second term in (\ref{15-3-82}) (to be justified later), we acquire factor 
\begin{equation}
Th^{-1}\Delta =  C\ell (\mu^{-1}+ \mathsf{v} T) Th^{-1}
\label{15-3-83}
\end{equation}
which after plugging $T=h^{1-\sigma}\ell^{-2}$ becomes
\begin{equation}
 C\ell \mu^{-1} Th^{-1} + C\mathsf{v} T^2h^{-1}=
 C\Bigl( \mu^{-1} \ell^{-1} + C\mathsf{v}h\ell^{-3}\Bigr)h^{-\sigma}
 \label{15-3-84}
\end{equation}
and it is $O(h^\sigma)$ as 
\begin{equation}
\ell \ge \hat{\ell}\Def \max\bigl(\mu^{-1}, \mathsf{v}^{\frac{1}{3}}h^{\frac{1}{3}}\bigr)h^{-\sigma}
\label{15-3-85}
\end{equation}
and the contribution of zone $\ell\le \hat{\ell}$ does not exceed
\begin{equation}
C\bar{\rho}\hat{\ell}h^{-1}=  C\bar{\rho}
\max\bigl(\mu^{-1}, \mathsf{v} ^{\frac{1}{3}}h^{\frac{1}{3}}\bigr)h^{-1-\sigma}
\label{15-3-86}
\end{equation}
and as $\mu \le h^{\delta-1}$ modulo $O(\mu^{-1}h^{-1})$ it is 
\begin{equation}
\bar{\rho}\mathsf{v} ^{\frac{1}{3}}h^{\frac{1}{3}} h^{-1-\sigma}\asymp
 \left\{ \begin{aligned}
&\mu^{\frac{7}{9}}h^{\frac{1}{9}-\sigma}&&\text{as\ \ }
\mathsf{v}\asymp \bar{\rho}^{\frac{1}{2}}\\
&\mu^{\frac{1}{3}}h^{-\sigma} &&\text{as\ \ }
\mathsf{v}\asymp \mu^{-1}
\end{aligned}\right.
 \label{15-3-87}
\end{equation}
which is $O(\mu^{-1}h^{-1})$ as \underline{either}
$\mathsf{v}\asymp \bar{\rho}^{\frac{1}{2}}$ and 
$\mu \le h^{\delta-\frac{9}{13}}$  \underline{or}
$\mathsf{v}\asymp \mu^{-1}$ and 
$\mu \le h^{\delta-\frac{3}{4}}$.

Meanwhile contribution of an element with $\ell\ge\hat{\ell}$ does not exceed
(\ref{15-3-81}) which is
\begin{equation}
C\bar{\rho}\ell h^{-1} \times \Bigl( \mu^{-1} \ell^{-1} + C\mathsf{v} h\ell^{-3}\Bigr)h^{-\sigma}
\label{15-3-88}
\end{equation}
and summation over $\ell\ge \hat{\ell}$ results in the same expression with $\ell=\hat{\ell}$ and it is the same (\ref{15-3-87}) as before.

On the other hand, (the ratio of) second term in (\ref{15-3-82}) to the first one is $(\mu^{-1}+ \mathsf{v} T)$ to $\ell$ and $\ell\ge \mu^{-1}$ so only comparison is $ \mathsf{v} h\ell^{-2}$ vs $\ell$ and $\ell$ is larger by the choice.

So, 
\begin{claim}\label{15-3-89}
Under condition (\ref{15-2-41}) contribution of $\cX_\trans$ to approximation error is $O\bigl(\mu^{\frac{7}{9}} h^{\frac{1}{9}-\delta}\bigr)$ for Neumann boundary condition and $O\bigl(\mu^{\frac{1}{3}}h^{-\delta}\bigr)$ for Dirichlet boundary condition.
\end{claim}

\subsection{\texorpdfstring{Case of condition \ref{15-2-7+}}{Case of condition (\ref{15-2-7})\textplussuperior}}
\label{sect-15-3-5-2}

Assume now that Dirichlet boundary condition is given and condition \ref{15-2-7+} is fulfilled. Then 
\begin{equation}
T_*= h^{-\sigma} \min \bigl( \frac{(\mu h)}{\bar{\rho}}, \frac{h}{\ell^2}\bigr)
\label{15-3-90}
\end{equation}
and therefore the above arguments should be applied only as $\ell\ge\hat{\ell}_1$ with 
\begin{equation}
\hat{\ell}_1= \mu^{-\frac{1}{6}}h^{\frac{1}{3}};
\label{15-3-91}
\end{equation}
otherwise we replace $\ell$ by $\hat{\ell}_1$ (single element) and $T=\mu^{\frac{1}{3}}h^{\frac{1}{3}-\sigma}$ and $\updelta x_2=\updelta x_1 = O(\mu^{-1})$; so $\Delta \asymp \hat{\ell}_1\mu^{-1}$. Here we assume that $\hat{\ell}_1\ge C\mu^{-1}$ as otherwise we reset it to this value which would lead to the error less than $C\mu^{-1}h^{-1-\sigma}$.

Then obviously the total error under the same replacement $V(x_1,x_2)$ by 
$V(0,y_2)+ (\partial_{x_1}V)(0,y_2)x_1$ does not exceed
\begin{equation}
C\bar{\rho}\hat{\ell}_1h^{-1} \times (\mu h)^{\frac{1}{3}} h^{-1-\sigma}\hat{\ell}_1\mu^{-1}\asymp
C\hat{\ell}_1^2 h^{-1-\sigma}
= h^{-\sigma} \min \bigl( \frac{(\mu h)}{\bar{\rho}}, \frac{h}{\hat{\ell}_1^2}\bigr)
\label{15-3-92}
\end{equation}
which in turn is $O\bigl(\mu^{-1}h^{-1} + h^{-\sigma}\bigr)$.

\subsection{Non-degenerate case}
\label{sect-15-3-5-3}

Finally, under condition (\ref{15-2-7}) we take 
$T_*=h^{1-\sigma}$ and $\Delta = \mu^{-1} h^{-\sigma}$ so each next term acquires factor $C\mu ^{-1}h^{-\sigma}$ i.e. the second term is 
$C\bar{\rho} h^{-1} \times \mu^{-1}h^{-\sigma}$ which is $O(\mu^{-1}h^{-1})$ unless $\mu \ge h^{\delta-1}$, in which case it is $O(h^{-\delta})$. However in the latter case the the third term is $O(1)$ so we will need to consider the second term in approximations as well. 

\subsection{Boundary zone} 
\label{sect-15-3-5-4}
Here we take  $T_*=h^{1-\sigma}$ and $\Delta = \mu^{-1} h^{-\sigma}$ so each next term acquires factor $C\mu ^{-1}h^{-\sigma}$ i.e. the second term is 
$C h^{-1} \times \mu^{-1}h^{-\sigma}$ which is $O(\mu^{-1}h^{-1-\sigma})$. However  the third term is $O(1)$ so we will need to consider the second term in approximations as well. 

\subsection{Inner zone}
\label{sect-15-3-5-5}
In inner zone we make usual transformation and apply the standard method of successive approximations at point $y$ rather than $(0,y_2)$.

\subsection{Calculations. I}
\label{sect-15-3-5-6}

As a result in the zone $\{x_1\ge C_0\mu^{-1}\}$ we get 
\begin{equation*}
h^{-2}\int \cN^\MW (x,\tau, \mu h)\psi(x)\,dx
\end{equation*}
and in the zone $\{x_1\le C_0\mu^{-1}\}$ we get something like this but with $\cN^\MW$ replaced by $\cN^\MW_*$ temporarily denoting eigenvalue counting function for operator with the boundary conditions. So this main part would come from parametrix $\bar{G}^\pm$ for operator $hD_t-A$ in the direction of $\pm t>0$ (see previous and similar chapters) and this parametrix is equal to $\bar{G}_0^\pm + \bar{G}_1^\pm$ where ``bar'' refers to freezing coefficients in $y$.

Consider first term with $\bar{G}^\pm_0$. While $\bar{G}^\pm_0$ is what it was on $\bR^2$, $\Gamma$ is not as now old $\Gamma$ is replaced by 
$\Gamma \uptheta(x_1)= \Gamma -\Gamma \uptheta(-x_1)$ and only the first term results in $h^{-2}\cN^\MW( y,\tau, \mu h)$. 

Note that contribution of terms with $\Gamma \uptheta(-x_1)\bar{G}_0^\pm $ and $\Gamma \uptheta(x_1)\bar{G}_1^\pm $ into main part of asymptotics are of magnitudes $\mu^{-1}h^{-2} \times \mu h = h^{-1}$ where $\mu^{-1}$ is the width of $X_\bound$ and $\mu h$ is a semiclassical parameter after rescaling. Therefore replacing $y_1$ by $0$ will bring really the estimate we had referred to and these terms together will result in $h^{-1}\cN^\MW_\corr$.

\subsection{Justification}
\label{sect-15-3-5-7}

To justify setting of the upper bound of propagation speed $\mu^{-1}$ rather than $\bar{\rho}^{\frac{1}{2}}$ under Dirichlet boundary conditions we just note that if $A'=A-A'$ then considering sandwich $\bar{G}^\pm A'\bar{G}^\pm \dots A' \bar{G}^\pm\updelta$ (one of $\bar{G}^\pm$ may be replaced by $G^\pm$) we can decompose $A'= A'_0+A'_+ + A'_-$ where $A'_0$, $A'_+$ and $A'_-$  are copies of operator $A'$ localized as  $|x_2-y_2|\le 2c_0\mu^{-1}(T+1)$,  $(x_2-y_2)\ge  \frac{3}{2}c_0\mu^{-1}(T+1)$ and $(x_2-y_2)\le  -\frac{3}{2}c_0\mu^{-1}(T+1)$ respectively.

Note that due to the fact that in direction of $\pm t$ propagation speed does not exceed $c_0\mu^{-1}$ we conclude that as we consider  time direction $\pm t>0$ sandwiches containing at least one factor  $A'_\pm$ are negligible, and sandwiches containing at least one factor  $A'_\mp$ become negligible after we apply $x_2=y_2$.

So, the first term in approximation results in what we claimed as an answer.

\subsection{Calculations. II}
\label{sect-15-3-5-8}

Meanwhile \emph{as we replace $V(x_1,y_2)$ by $V(0,y_2)+(\partial_{x_1}V)(0,y_2)x_1$\/}  we notice that
$\Delta = C_0(\updelta x_1)^2\asymp C_0\mu^{-2}$ as $\updelta x_1= O(\mu^{-1})$ which is smaller than (\ref{15-3-82}). 

The same arguments show that removing $(\partial_{x_1}V)(0,y_2)x_1$ leads to the error not exceeding
\begin{equation}
C\mu ^{-\frac{1}{2}} \bar{\rho}h^{-1} = C\mu^{\frac{1}{6}}h^{-\frac{1}{3}-\sigma}
\label{15-3-93}
\end{equation}
as we need to use $\hat{\ell}_1= \mu^{-\frac{1}{2}}$ instead of $\hat{\ell}$ defined by (\ref{15-3-86}).

Unfortunately expression (\ref{15-3-93}) may be larger than anything we got before. Fortunately there are alternatives: first, replacing $x_1$ by $W^{\frac{1}{2}}_0(x_2)$ leads to an error in operator $O(\mu^{-1}\bar{\rho})$ and to the approximation error
\begin{equation}
C(\mu ^{-1}\bar{\rho})^{\frac{1}{2}} \bar{\rho}h^{-1} = C\mu^{\frac{1}{2}}h^{-\sigma};
\label{15-3-94}
\end{equation}
this would mean replacing $\cN^\MW_{*,\bound} (W_0)$ by 
$\cN^\MW_{*,\bound} (W_1)$ with $W_1=W_0+\mu^{-1}(\partial_{x_1}W)W_0^{\frac{1}{2}}$.

Second, we can replace $x_1$ by $\mu^{-1}\xi_2$ i.e. use 
$W(\mu^{-1}\xi_2 ,x_2)$ instead of $W _0(x_2)$ and we do not need to remove linear term $(\partial_{x_1}V)(0,y_2)(x_1-\mu^{-1}\xi_2)$ but just simply shift with respect to $\xi_2$. This would lead to $O(\mu^{-2})$ shift in $W(\mu^{-1}\xi_2,x_2)$ and this shift could be ignored.

\subsection{Calculations. III}
\label{sect-15-3-5-9}

Consider the second term in approximations. There are two cases when we need to do this: in the boundary zone as $\mu \le h^{\delta-1}$ and under condition (\ref{15-2-8}) as $\mu \ge h^{\delta-1}$. 

In the former case however we need to consider only $T\le \epsilon \mu$ and then only one winding should be considered and perturbation does not exceed  $C_0\mu^{-1}$, and we can take $T=h^{1-\sigma}$ so in fact we get an extra factor $\mu^{-1}h^{-\sigma}$. However scaling shows that in fact as we replace $\bar{\chi}_T$ by $\chi_T$ the error will be $C\mu^{-1}Th^{-2}(h/T)^s$ and summation results in $C\mu^{-1}h^{-1}$ error. Alternatively we could employ semiclassical approximation to show the same.

In the latter case there is no boundary zone anymore. However scaling shows that the same arguments work (without going to semiclassics).

\section{Final results}
\label{sect-15-3-6}

Let us compare Tauberian remainder and approximation error.

\medskip\noindent
(i) Under assumptions (\ref{15-2-8}) both Tauberian remainder and an approximation error are $O(\mu^{-1}h^{-1})$.

\medskip\noindent
(ii) Under assumptions (\ref{15-2-7}) and (\ref{15-2-41})  Tauberian remainder is estimated by $\textup{(\ref{15-3-36})}+O(\mu^{-1}h^{-1})$ and an approximation error is estimated by (\ref{15-3-87})+(\ref{15-3-93}) and estimates (\ref{15-3-36}) and (\ref{15-3-87}) coincide, so the total  error is 
$\textup{(\ref{15-3-87})}+\textup{(\ref{15-3-93})}+O(\mu^{-1}h^{-1})$.

\medskip\noindent
(iii) Under assumptions (\ref{15-2-7}) and (\ref{15-2-41}) and Dirichlet boundary condition Tauberian remainder is estimated by $\textup{(\ref{15-3-43})}+O(\mu^{-1}h^{-1})$ 
and approximation error is estimated by (\ref{15-3-87})+(\ref{15-3-93}) and (\ref{15-3-43})  is larger than (\ref{15-3-87}) so the total error is 
$\textup{(\ref{15-3-43})}+\textup{(\ref{15-3-93})}+O(\mu^{-1}h^{-1})$.

\medskip\noindent
(iv) Under assumptions(\ref{15-2-7}) and $\textup{(\ref{15-2-41})}^-$ and Dirichlet boundary condition Tauberian remainder is estimated by (\ref{15-3-46}) 
and approximation error is $\textup{(\ref{15-3-93})}+O(\mu^{-1}h^{-1}+h^{-\delta})$ and the total error is (\ref{15-3-46})+(\ref{15-3-93}).

\medskip\noindent
(v) Under assumptions \ref{15-2-7+} and $\textup{(\ref{15-2-41})}^-$ and Dirichlet boundary condition Tauberian remainder is $O(\mu^{-1}h^{-1}+h^{-\delta})$ and an approximation error is
$\textup{(\ref{15-3-93})}+O(\mu^{-1}h^{-1}+h^{-\delta})$ and this is a total error.

Thus we arrive to

\begin{theorem}\label{thm-15-3-16}
Let $\psi\in \sC^\infty (\bar{X})$ be a fixed function with a compact support contained in the small vicinity of $\partial X$ and let conditions  \textup{(\ref{13-2-1})}--\textup{(\ref{13-1-4})}, \textup{(\ref{13-2-1})} and and \textup{(\ref{13-3-45})} be fulfilled on $\supp\psi$. 

Further, let condition \textup{(\ref{15-3-1})} be fulfilled.  Then

\medskip\noindent
(i) Under non-degeneracy condition \textup{(\ref{15-2-8})}  $\R^\MW$  is  $O(\mu^{-1}h^{-1})$;

\medskip\noindent
(ii) Under non-degeneracy conditions \textup{(\ref{15-2-7})} and \textup{(\ref{15-2-41})} $\R^\MW$ does not exceed $\textup{(\ref{15-3-87})}+\textup{(\ref{15-3-93})}+O(\mu^{-1}h^{-1})$; 

\medskip\noindent
(iii) Under non-degeneracy conditions \textup{(\ref{15-2-7})} and \textup{(\ref{15-2-41})}  $\R^\MW_\D$ does not exceed  $\textup{(\ref{15-3-43})}+\textup{(\ref{15-3-93})}+O(\mu^{-1}h^{-1})$; 

\medskip\noindent
(iv) Under assumptions \textup{(\ref{15-2-7})} and $\textup{(\ref{15-2-41})}^-$   $\R^\MW_\D$ does not exceed $\textup{(\ref{15-3-46})}+\textup{(\ref{15-3-93})}$; 

\medskip\noindent
(v) Under non-degeneracy conditions \ref{15-2-7+} and $\textup{(\ref{15-2-41})}^-$  $\R^\MW_\D$ does not exceed  $\textup{(\ref{15-3-93})}+O(\mu^{-1}h^{-1}+h^{-\delta})$.
\end{theorem}

We leave to the reader

\begin{problem}\label{problem-15-3-17}
Write down correction terms $\cN^\MW_{*,\bound,\corr}$ corresponding to the procedures described subsubsection~``\nameref{sect-15-3-5-8}''. Then  $\R^\MW_\corr$ will be the same as in theorem~\ref{thm-15-3-16} albeit without  \textup{(\ref{15-3-93})}.

One of the possible modifications of $\cN^\MW_{*,\bound}$ is given by (\ref{15-4-81}).
\end{problem}

\begin{conjecture}\label{conj-15-3-18}
All these estimates hold with $\delta=0$.
\end{conjecture}

\chapter{Superstrong magnetic field}
\label{sect-15-4}

\section{Preliminary analysis}
\label{sect-15-4-1}

In this section we consider cases  of \emph{very strong magnetic field\/}\footnote{\label{foot-15-10} Note that it is more narrow definition than (\ref{15-3-1}) in the previous section.}\index{very strong magnetic field}\index{magnetic field!very strong}
\begin{gather}
\epsilon_0 h^{-1} \le \mu \le C_0h^{-1}\label{15-4-1}\\
\intertext{and superstrong magnetic field}
C_0 h^{-1} \le \mu \label{15-4-2}
\end{gather}
In the latter case operator needs to be modified: we will replace $V$ by 
$V-\mu h \fz F$ where constant $\fz$ will be specified later; so operator (\ref{15-2-1}) is replaced by
\begin{equation}
A=\sum_{1\le j,k\le 2} P_jg^{jk}P_k-\mu h \fz F+V,\qquad \text{with\ \ } P_j=hD_j-\mu V_j
\label{15-4-3}
\end{equation}

Then 
\begin{claim}\label{15-4-4}
We need only to consider eigenvalues $\lambda_n(\eta)$ of model operator $L(\eta)$ defined by (\ref{15-1-26}) with $n\le N\Def N(\epsilon_0)$.
\end{claim}
Really, as $n=0,1,2,\dots$
\begin{phantomequation}\label{15-4-5}\end{phantomequation}
\begin{align}
&\inf_\eta \lambda_{\D,n}(\eta)\ge (2n+1) \quad (n\ge 0)
\tag*{$\textup{(\ref*{15-4-5})}_\D$}\label{15-4-5-D}\\
&\inf_\eta \lambda_{\N,n}(\eta)\ge (2n-1) ,\quad (n\ge 1).
\tag*{$\textup{(\ref*{15-4-5})}_\N$}\label{15-4-5-N}
\end{align}
due to proposition~\ref{prop-15-A-1} and our operator is modelled by 
$(\mu h)^{-1}L(\eta)$.

However then 
\begin{equation} 
|\lambda_{*,m}(\eta)-\lambda_{*,n}(\eta)|\ge \epsilon (N)
\qquad\forall m\ne n\le N\;\forall \eta
\label{15-4-6}
\end{equation}
with subscript $*$ denoting either $\D$ or $\N$ and we can decompose into eigenfunctions $\upsilon_{*,n}(\mu x_1,\xi_2)$ of the model operator. So basically we consider one-dimensional diagonal operator perturbed and we can diagonalize it to $\mu hF(0,x_2)\cA_n (x_2,hD_2)$ with 
\begin{gather}
\cA_n (x_2,hD_2)\Def \bigl(\lambda_{*,n}(\hbar D_2) -\fz - (\mu h)^{-1} W(0,x_2)\bigr) + \dots\label{15-4-7}\\
\intertext{where in this section}
\hbar = \mu^{-\frac{1}{2}}h^{\frac{1}{2}};
\label{15-4-8}
\end{gather}
 preliminary we scaled 
$x_1\mapsto \hbar ^{-1}x_1$. Let us ignore perturbation in (\ref{15-4-7}); later in subsection (\ref{sect-15-4-5}) we will show that it could be made very small, albeit $W$ will be modified.

\begin{remark}\label{rem-15-4-1}
In (\ref{15-4-7}) we replaced $x_1$ by $0$. Actually the better approximation would come by replacing $x_1$ by $\mu^{-1}hD_2$ but need in it will not come out instantly.
\end{remark}

Then we need to consider only scalar operators (\ref{15-4-7}) and we are looking at fixed vicinity of some point $\bar{x}_2$. What we can say about $\xi_2$ apart of $\xi_2\ge -C_0$ (as $\lambda_{*,n}\to +\infty$ as $\xi_2\to -\infty$)?

There are two possibilities: either we would be inside of the spectral gap if not a boundary, or not. 

\section{``Almost spectral gaps''}
\label{sect-15-4-2}

Assume first that we would be inside of the spectral gap if not a boundary:
\begin{equation}
|(2m+1-\fz)\mu h F +V-\tau|\ge \epsilon_0\mu h \qquad \forall m=0,1,2,\dots
\label{15-4-9}
\end{equation}
where $|\tau| \le C$ and usually $\tau=0$.

We know (see subsection~\ref{book_new-sect-13-5-1} \cite{futurebook}) that inside of domain (\ref{15-4-9}) means a spectral gap. However it is not the case near boundary: looking at behavior of $\lambda_{*,n}(\xi_2)\mu h $ interpreted as \emph{Landau levels\/}
\index{Landau level} we can conjecture that  only lower spectral gap $(-\infty, .)$ survives and even it shrinks for Neumann boundary condition. 

As $\lambda_{*,n}(\xi_2)\to (2n+1)\mu h$ as $\xi_2\to +\infty$ due to proposition~\ref{prop-15-A-1} again and we conclude that 
\begin{claim}\label{15-4-10}
Under condition (\ref{15-4-9}) operators in question are elliptic as 
$\xi_2\ge C_0$ and we need to consider only a compact interval $|\xi_2|\le C_0$. 
\end{claim}

Note that as\begin{phantomequation}\label{15-4-11}\end{phantomequation} 
\begin{equation}
\lambda_{*,n}(\xi_2)-\fz  -(\mu h)^{-1}W  \ge \epsilon \mu h\qquad \forall \xi_2
\tag*{$\textup{(\ref*{15-4-11})}_{n}$}\label{15-4-11-*}
\end{equation}
then operator $\cA_n$ is elliptic for all $\xi_2$ and the contribution of it and interval $(-\infty,\tau)$ to the asymptotics is negligible. In frames of (\ref{15-4-9}) one can rewrite above condition as\begin{phantomequation}\label{15-4-12}\end{phantomequation}
\begin{gather}
(2n+1-\fz-\epsilon)\mu h F + V-\tau \ge 0,
\tag*{$\textup{(\ref*{15-4-12})}_{\D,n}$}\label{15-4-12-D}\\
(\lambda_{\N,n}^*-\fz-\epsilon)\mu h F + V-\tau \ge 0,
\tag*{$\textup{(\ref*{15-4-12})}_{\N,n}$}\label{15-4-12-N}
\end{gather}
for Dirichlet and Neumann boundary conditions respectively where
\begin{equation}
\lambda_{\N,n}^*=\min_\eta \lambda_{\N,n}(\eta) \in (2m-1,2m+1).
\label{15-4-13}
\end{equation}

One can then concludes easily that

\begin{theorem}\label{thm-15-4-2}
Let one of assumptions \textup{(\ref{15-4-1})}, \textup{(\ref{15-4-2})} be fulfilled. Further, let  $\textup{(\ref*{15-4-12})}_{\D,0}$ or $\textup{(\ref*{15-4-12})}_{\N,0}$ be fulfilled with $\tau=0$ (matching to the boundary condition).

Then 
\begin{equation}
|e(x,x,0)|\le C_s\mu^{-s}
\label{15-4-14}
\end{equation}
with arbitrarily large exponent $s$.
\end{theorem}

Assume now that condition (\ref{15-4-9}) is fulfilled for all $m$ but 
$\textup{(\ref*{15-4-12})}_{*,0}$ fails. Then there is a large difference between Dirichlet and Neumann cases because only in the former case in virtue of proposition~\ref{prop-15-A-2}
\begin{equation}
 \lambda'_{*,n} \le -\epsilon (N,C_0)
\qquad \forall n\le N\qquad \forall \xi_2:|\xi_2|\le C_0
\label{15-4-15}
\end{equation}
where here and below $ \lambda'_{*,n}\Def \partial_{\xi_2}\lambda_{*,n}$ etc. For Neumann boundary condition we assume that
\begin{equation}
| \lambda'_{\N,n}(\xi_2)| \le \epsilon_0  \implies 
|\bigl(\lambda _{\N,n}(\xi_2)-\fz\bigr)\mu h F- W|\ge \epsilon_0\mu h.
\label{15-4-16}
\end{equation}

So, we arrive to the remainder estimate $O(1)$; transition from Tauberian expression to magnetic Weyl one is trivial:

\begin{theorem}\label{prop-15-4-3}
Let one of assumptions \textup{(\ref{15-4-1})}, \textup{(\ref{15-4-2})} be fulfilled. Further, let  \textup{(\ref{15-4-9})} and in case of Neumann boundary  \textup{(\ref{15-4-16})} be fulfilled. Then $\R^\MW_*=O(1)$.
\end{theorem}

However there are other mechanisms to break degeneracy for Neumann problem: first of all  we know (see proposition~\ref{prop-15-A-3}) that 
\begin{equation}
|\lambda'_{\N,n} (\xi_2)|\le \epsilon_0 \implies 
 \lambda''_{\N,n} (\xi_2) \ge \epsilon_0
\qquad
\forall n=0,1,2,\dots
\label{15-4-17}
\end{equation}
The first non-degeneracy assumption  linked to derivatives with respect to $x_2$ is
\begin{multline}
|\bigl(\lambda_{\N,n} (\xi_2)-\fz\bigr)\mu h -W | \le \epsilon_0 \mu h,\quad |\lambda'_{\N,n} (\xi_2) |\le \epsilon_0 \implies\\[2pt]
|\partial_{x_2}W| \ge \epsilon_0\qquad
\forall m=0,1,2,\dots,
\label{15-4-18}
\end{multline}
where as usual $W=-VF^{-1}$;  the next one is\begin{phantomequation}\label{15-4-19}\end{phantomequation}
\begin{multline}
|\bigl(\lambda_{\N,n} (\xi_2)-\fz\bigr) \mu h -W| \le \epsilon_0 \mu h,\\[2pt]   |\lambda'_{\N,n} (\xi_2)|+ |\partial_{x_2}W|\le \epsilon_0 \implies 
\pm \partial_{x_2}^2 W \ge \epsilon_0
\qquad
\forall n=0,1,2,\dots
\tag*{$\textup{(\ref*{15-4-19})}^\pm$}\label{15-4-19-*}
\end{multline}

\begin{theorem}\label{thm-15-4-4}
Let one of  condition \textup{(\ref{15-4-1})}, \textup{(\ref{15-4-2})} be fulfilled.  Further, let conditions \textup{(\ref{13-2-1})} and \textup{(\ref{15-4-9})} be fulfilled in $B(\bar{x},1)$ and let 
$\psi\in \sC_0^\infty$ be supported in 
$B(\bar{x},\frac{1}{2})\cap \{x:x_1\le \epsilon '\}$. 

Then $\R^\MW_\N=O(1)$ under one of the assumptions \textup{(\ref{15-4-18})},  $\textup{(\ref{15-4-19})}^+$  \underline{and} $\R^\MW_\N=O\bigl(|\log h|)\bigr)$ under assumption $\textup{(\ref{15-4-19})}^-$.
\end{theorem}

\begin{proof}
As $\mu h\asymp 1$ we are in frames of subsection~\ref{book_new-sect-5-1-3} \cite{futurebook} and arrive immediately to conclusion of theorem. 

Consider case (\ref{15-4-2}) of superstrong magnetic field and adapt arguments of subsection~\ref{book_new-sect-5-1-3} \cite{futurebook}. Then first we consider $\epsilon(1,\ell)$-admissible (with respect to $(x_2,\xi_2)$) with 
$\rho =\epsilon (|\lambda' _{\N,n}|$ partition of zone 
$\{|\lambda' _{\N,n}| \ge C(\mu h)^{-\frac{1}{2}}\}$. Then for operator (\ref{15-4-7}) 
\begin{equation}
T_*\asymp \hbar\rho^{-2},\qquad T^* \asymp \rho^{-1},
\label{15-4-20}
\end{equation}
as propagation velocities with respect to $x_2,\xi_2$  are $\asymp \rho$ and 
$O\bigl((\mu h)^{-1})$ respectively. Then contribution of an element to the Tauberian remainder does not exceed 
$C\hbar^{-1}\rho \times T_* \times T^{*\,-1}\asymp C$ and the total contribution of zone in question does not exceed $C\int \rho^{-1}\,d\rho \asymp C$ as operator is elliptic unless $\rho$ has a specific magnitude.

In the remaining zone 
$\{|\lambda' _{\N,n}|\le C(\mu h)^{-\frac{1}{2}}$ let us introduce $\epsilon(\rho,\ell)$-admissible partition with
\begin{equation}
\rho = \bigl(|\lambda' _{\N,n}|^2+
(\mu h)^{-1}|\partial_{x_2}W|^2\bigr)^{\frac{1}{2}}+\bar{\rho},\quad 
\ell = (\mu h)^{\frac{1}{2}}\rho, \quad \bar{\rho}=\mu^{-\frac{1}{2}}.
\label{15-4-21}
\end{equation}
Then for operator (\ref{15-4-7}) divided by $\mu h$
\begin{equation}
T_*\asymp \hbar\rho^{-2},\qquad T^* \asymp (\mu h)^{\frac{1}{2}},
\label{15-4-22}
\end{equation}
as propagation velocity with respect to $(x_2,(\mu h)^{\frac{1}{2}}\xi_2)$  is $\asymp \rho$. Therefore contribution of $(\rho,\ell)$ element to the Tauberian remainder does not exceed 
$C\hbar^{-1}\rho\ell \times T_* \times T^{*\,-1}\asymp C$ and the total contribution does not exceed $C\int \rho^{-1}\,d\rho$ which is $O(1)$ under one of the assumption (\ref{15-4-18}), $\textup{(\ref{15-4-19})}^+$ as in the former case $\rho\asymp \bar{\rho}_1=(\mu h)^{-\frac{1}{2}}$ and in the latter case operator is elliptic unless $\rho$ has a fixed magnitude. Further, we get $O(|\log h|)$ under assumption $\textup{(\ref{15-4-19})}^-$ as integral is taken from $\bar{\rho}$ to $\bar{\rho}_1$.

Transition from Tauberian expression to magnetic Weyl expression is trivial.\ 
\end{proof}

\begin{remark}\label{rem-15-4-5}
Following section~\ref{book_new-sect-5-1-3} \cite{futurebook} we would be able to establish remainder estimate $O(1)$ even under assumption $(\textup{(\ref{15-4-19})}^-$. 
\end{remark}

\begin{remark}\label{rem-15-4-6}
(i) The principal part of asymptotics in question is $O(\mu h^{-1})$ and as 
\begin{equation}
 (1+\fz)+F^{-1}(V -\tau)\le -\epsilon
\label{15-4-23}
\end{equation}
it is $\asymp \mu h^{-1}$.

\medskip\noindent
(ii) As Neumann boundary condition is considered and
\begin{equation}
\mu h (1+\fz) F+V -\tau \ge \epsilon\mu h
\label{15-4-24}
\end{equation}
the principal part of asymptotics in question is $O(h^{-1})$ and as
\begin{equation}
\mu h F (\lambda^*_{\N,0}+\fz) + V -\tau \le -\epsilon\mu h
\label{15-4-25}
\end{equation}
it is $\asymp \hbar^{-1}=\mu^{\frac{1}{2}}h^{-\frac{1}{2}}$. 

\medskip\noindent
(iii) As 
\begin{align}
&\mu h (2m-1+k) F+V -\tau_1 \le -\epsilon\mu h, \label{15-4-26}\\[2pt] 
&\mu h (2m+1+k) F+V -\tau_2 \ge -\epsilon\mu h, \qquad
\tau_1<\tau_2, \notag
\end{align}
then the increment of principal part of asymptotics calculated as $\tau$ increases from $\tau_1$ to $\tau_2$ is $O(\hbar^{-1})$ and is 
$\asymp \hbar^{-1}$ as 
$|\tau_2-\tau_1|\asymp \mu h$. Recall that without boundary it would be 
$O(\hbar^s)$. 
\end{remark}

\begin{problem}\label{problem-15-4-7}
Investigate under weaker non-degeneracy conditions than \ref{15-4-19-*}.
\end{problem}

\section{Landau level}
\label{sect-15-4-3}

\subsection{Preliminary analysis}
\label{sect-15-4-3-1}

Now consider case when condition (\ref{15-4-9}) is broken. Without any loss of the generality we can assume that 
\begin{claim}\label{15-4-27}
Condition (\ref{15-4-9}) breaks only for one value of $m=\bar{m}$.
\end{claim}
Really, if $\mu h \ge C$ this assumption would be fulfilled automatically; otherwise we achieve it by considering $\epsilon$-admissible partition of unity.

Further, without any loss of the generality we can assume that 
\begin{equation}
\fz=2\bar{m}+1.
\label{15-4-28}
\end{equation}
Really, let $\nu =|\fz-(2\bar{m}+1)|\le 1$. If \ $\nu \mu h \le C$ then term 
$\pm \nu \mu h F$ could be absorbed into $V$; if $\nu \mu h\ge C$ then dividing operator by $\nu \mu h$ we find ourselves in the frames of the previous subsection with 
$\mu_\new = \mu ^{\frac{1}{2}}h^{-\frac{1}{2}}\nu^{-\frac{1}{2}}$ and
$\mu_\new = \mu ^{-\frac{1}{2}}h^{\frac{1}{2}}\nu^{-\frac{1}{2}}\le 1$.

Then for all $n\ne \bar{m}$ the previous arguments work and we need to consider  $n=\bar{m}$ only. Further, as $n=\bar{m}$ these arguments work in zone $\{\xi_2\le C_0\}$ and we need to consider only zone $\{C_0\le \xi_2\le \epsilon_0\mu \}$.

As symbol operator $\mu h\cA_n$ with $\cA_n$ defined by (\ref{15-4-7}) has a either first or second derivative with respect to $x_2$ disjoint from $0$, but derivatives with respect to $\xi_2$ could be pretty small, we will see that  we cannot skip without reservation $O(\hbar)$ terms  like those with an extra factor $x_1$ even as we derive Tauberian remainder estimate but we in our assumptions will be able in Tauberian arguments only to skip $O(\hbar^2)$  terms like those with an extra factor $x_1^2$ so we need to modify (\ref{15-4-7}) a bit\footnote{\label{foot-15-11} Actually effectively $x_1$ could be marginally more than $\hbar$.}. To do this we first drag out $F(x)$ out of $(A-\tau)$ leaving inside 
\begin{multline}
(hD_1)^2 + (hD_2-\mu x_1)^2 - W(x_1,x_2)+\dots=\\[3pt]
(hD_1)^2 + (hD_2-\mu x_1)^2 - W(\mu^{-1}hD_2,x_2)-\\[3pt]
\shoveright{(\partial_{x_1}W)(\mu ^{-1}hD_2,x_2)(x_1-\mu^{-1}hD_2)+\dots=} \\[3pt]
(hD_1)^2 + \bigl(hD_2-\mu x_1+\frac{1}{2}\mu^{-1} 
(\partial_{x_1}W)(\mu ^{-1}hD_2,x_2)\bigr)^2 - W(\mu ^{-1}hD_2,x_2) +\dots
\label{15-4-29}
\end{multline}
with $W=-(V-\tau)F^{-1}$ resulting after shift 
\begin{equation}
hD_2\mapsto  
hD_2-\mu x_1-\frac{1}{2}\mu^{-1} (\partial_{x_1}W)(\mu ^{-1}hD_2,x_2)
\label{15-4-30}
\end{equation}
in
\begin{equation}
\mu h \cA_n (x_2,hD_2)=  
\mu h\lambda_{*,n}(\hbar D_2 ) - W(\hbar^2 D_2,x_2) + \dots .
\label{15-4-31}
\end{equation}
 There could also be $O(x_1)$ terms\footnote{\label{foot-15-12} After division by $F$ intensity of magnetic field is $1$ so there could be no terms (\ref{15-4-32}) with ``$-$'' replaced by ``$+$''.} 
\begin{equation}
\alpha (x_2)x_1  \Bigl(h^2 D_1^2 -  (hD_2-\mu x_1)^2\Bigr);
\label{15-4-32}
\end{equation}
then replacing $x_1\mapsto x_1+\beta (x_2) x_1^2$ we also need to replace 
\begin{equation*}
hD_1\mapsto \bigl(1-2\beta x_1 +O(x_1^2)\bigr)hD_1 \quad\text{and} \quad hD_2\mapsto \bigl(hD_2 +O(x_1^2)hD_1\bigr)
\end{equation*}
so modulo $O(x_1^2)$
\begin{multline*}
(hD_1)^2 + (hD_2-\mu x_1)^2\mapsto (1-4\beta x_1)(hD_1)^2 +
(1+4\beta x_1) (hD_2-\mu x_1)^2 + \\2(hD_2-\mu x_1)\beta \mu^{-1}h^2D_2^2 
\end{multline*}
and choosing $\beta=-\frac{1}{4}\alpha$ we arrive to 
\begin{equation*}
(hD_1)^2+ (hD_2-\mu x_1+2\beta \mu^{-1}h^2D_2^2)^2 
\end{equation*}
and thus we need to replace (\ref{15-4-31}) by 
\begin{multline}
\mu h \bigl(\cA_n (x_2,hD_2)-2n-1\bigr)= \\ 
\mu h\bigl(\lambda_{*,n}(\hbar D_2 + 2\beta \hbar^3 D_2^2) -2n-1\bigr)- 
W(\hbar^2 D_2,x_2(1+) + \dots .
\label{15-4-33}
\end{multline}
\begin{remark}\label{rem-15-4-8}
In the case of superstrong magnetic field we need to consider only zone
\begin{equation}
\bigl\{\xi_2,\ |\lambda_{*,n}-2n-1|\le C_0(\mu h)^{-1}\bigr\}
\label{15-4-34}
\end{equation}
otherwise operator is elliptic.
\end{remark}

\subsection{Non-degenerate case}
\label{sect-15-4-3-2}
Consider first case 
\begin{multline}
|(2n+1-\fz)\mu hF + V|\le \epsilon_0 \implies |\nabla_{\partial X} VF^{-1}|\ge \epsilon_0
\\
\forall n=0,1,2,\dots
\label{15-4-35}
\end{multline}
Then obviously  one can take $T_*\asymp\hbar$. Really, speed with respect to $\xi_2$ is $\asymp |\partial _{x_2}W|\asymp 1$.

Meanwhile consider time direction in which $\xi_2$ increases; we can select it due to assumption (\ref{15-4-35}). To reach $\xi_2=\epsilon\hbar^{-1}$ we need time $\asymp \hbar^{-1}\mu h$ as for $\xi_2\asymp \hbar^{-1}$ speed with respect to $\xi_2$ is $(\mu h)^{-1}$. Consider evolution with respect to $x_2$; speed with respect to $x_2$ does not exceed $C\mu h |\lambda'_{*,n}|+C\hbar$. Then we can take for sure take 
\begin{equation}
T^*=T^*(\xi_2)\asymp  \min\bigl( (\mu h|\lambda'_{*,n}|)^{-1}, \hbar^{-1}\bigr)
\label{15-4-36}
\end{equation}

Then contribution of $B(\bar{x},1)\cap \{x:x_1\le \epsilon '\}$ to the remainder does not exceed  
\begin{equation}
C\int \hbar^{-1}\frac{T_*(\xi_2)}{T^*(\xi_2)}\,d\xi_2\le
\int \bigl(\hbar + C\mu h |\lambda'_{*,n}|\bigr)\,d\xi_2
\label{15-4-37}
\end{equation}
as integrand in the left-hand expression does not exceed $C\hbar$ as 
$|\lambda'_{*,n}|\le \hbar (\mu h)^{-1}$ and it does not exceed
$C\mu h |\lambda'_{*,n}|$
otherwise. As $\lambda_{*,n}(\xi_2)$ is monotone (at least for $\xi_2\ge C$)  and integral is taken over $\xi_2\le \epsilon \hbar^{-1}$ also satisfying (\ref{15-4-34}) the right-hand expression does not exceed $C$.

Thus 

\begin{claim}\label{15-4-38}
Under conditions (\ref{15-4-1}) or  (\ref{15-4-2}), (\ref{15-4-28}) and (\ref{15-4-35}) fulfilled in $B(\bar{x},1)$ contribution of 
$B(\bar{x},\frac{1}{2})\cap \{x:x_1\le \epsilon '\}\cap\{\xi_2\ge C\}$ to the Tauberian remainder does not exceed $C$.
\end{claim}

\subsection{Generic case}
\label{sect-15-4-3-3}
Now look what happens as \textup{(\ref{15-4-35})} is also broken and replaced by conditions\begin{phantomequation}\label{15-4-39}\end{phantomequation}
\begin{multline}
|(2n+1-\fz)\mu h F +V|+ |\nabla_{\partial X} VF^{-1}|\le \epsilon_0\implies\\[3pt]
\mp \nabla^2_{\partial X} VF^{-1}\ge \epsilon_0
\qquad
\forall n=0,1,2,\dots
\tag*{$\textup{(\ref*{15-4-39})}^{\pm}$}
\end{multline}
and\begin{phantomequation}\label{15-4-40}\end{phantomequation}
\begin{multline}
|(2n+1-\fz)\mu h F +V|+ |\nabla_{\partial X} VF^{-1}|\le \epsilon_0\implies\\[3pt]
\mp \nabla VF^{-1} \ge \epsilon_0
\qquad
\forall n=0,1,2,\dots
\tag*{$\textup{(\ref*{15-4-40})}^{\pm}$}
\end{multline}

\paragraph{Inner zone.}
\label{sect-15-4-3-3-1}
We define this zone preliminary by 
\begin{equation*}
\cX_\inn =\{\xi_2 \ge c_0(\log \mu)^{\frac{1}{2}}\}
\end{equation*}
which means not only that $\mu h (\lambda_{*,n}-2n-1)$ and  
$\mu h \lambda'_{*,n}$ are negligible but that they remain so even if we replace $\xi_2$ by $(1-\epsilon)\xi_2$. Then under assumption (\ref{15-4-40}) we can take 
\begin{gather}
T_*\asymp \min \bigl(\xi_2^{-1},\hbar\ell^{-2}\bigr), \label{15-4-41}\\
\ell\asymp |W_{x_2}|\label{15-4-42}
\end{gather}
Really, the propagation speed with respect to $x_2$ is $\asymp \hbar$ but as  scale with respect to $\xi_2$ is $\asymp \xi_2$ there, uncertainty principle means that $\hbar T \cdot \xi_2 \ge \hbar$. Similarly, propagation speed with respect to $\xi_2$ is $\asymp \ell$, scale with respect to $x_2$ is $\asymp \ell$  and uncertainty principle means that $\ell T \cdot \ell \ge \hbar$.

Consider propagation in the time direction $\sign (W_{x_2}) t>0$ in which $\xi_2$ increases. Then 
\begin{equation*}
\sign (W_{x_2}) \frac{d\ }{dt}\log |W_{x_2}|=  
-\hbar W_{x_1} \sign |W_{x_2}|^{-1} W_{x_2x_2} 
\end{equation*}
so $\log |W_{x_2}|$ also increases provided $\W_{x_1}$ and $W_{x_2x_2}$ have opposite signs. Then we can take $T^*=\hbar^{-1}$. This would lead to contribution of $\cX_\inn \cap\{\xi_2\ell \ge \hbar\}$ to the Tauberian remainder not exceeding 
\begin{equation*}
C\iint \hbar^{-1}\frac{T_*(\xi_2)}{T^*(\xi_2)}\,d\xi_2 dx_2= 
C\iint \min (\xi_2^{-1}, \hbar\ell^{-2})\,d\ell d\xi_2 \le C. 
\end{equation*}
On the other hand, in $\cX_\inn \cap\{\xi_2\ell \le \hbar\}$ we can take 
$T_*\asymp\xi_2^{-1}$ and $T^*\asymp \xi_2$ and contribution of this subzone 
to the Tauberian remainder not exceeding 
\begin{equation*}
C\iint \hbar^{-1}\frac{T_*(\xi_2)}{T^*(\xi_2)}\,d\xi_2 dx_2= 
C\int \xi_2^{-2}\,d\xi_2\le C.
\end{equation*}

On the other hand, as $\textup{(\ref{15-4-39})}^\pm$, $\textup{(\ref{15-4-40})}^\pm$ with the same signs are fulfilled we should chose between $T^*\asymp \xi_2$ and $T^*\asymp \hbar^{-1}\ell$. So, as $\ell \ge (\hbar\xi_2)^{\frac{1}{2}}$ we can take $T_*=\hbar\ell^{-2}$, $T^*=\hbar\ell^{-1}$ and contribution of this subzone to the Tauberian remainder does not exceed 
\begin{equation*}
C\iint \hbar^{-1}\frac{T_*(\xi_2)}{T^*(\xi_2)}\,d\xi_2 dx_2= 
C\int \ell^{-1}\,d\ell\le C
\end{equation*}
as in the case of the same signs $\ell^2+\xi_2\hbar$ must have the same magnitude or operator is elliptic; so in this subzone $\ell$ must have the same magnitude. 

As $\ell\le (\hbar\xi_2)^{\frac{1}{2}}$ we can take $T_*=\xi_2^{-1}$, but the same ellipticity argument means that actually we can upgrade $T^*\asymp \hbar$ to $T^*\asymp \hbar^{-\frac{1}{2}}\xi_2^{\frac{1}{2}}$ contribution of this subzone to the Tauberian remainder does not exceed 
\begin{equation*}
C\iint \hbar^{-1}\frac{T_*(\xi_2)}{T^*(\xi_2)}\,d\xi_2 dx_2= 
C\int \xi_2^{-1} \,d\xi_2\le C
\end{equation*}
as in this subzone $\ell$ must have the same magnitude.

We need to extend $\cX_\inn$ up to 
\begin{gather}
\cX_\inn\Def \{ \xi_2\ge  \bar{\xi}^+_2 \}\label{15-4-43}\\
\shortintertext{with} 
\bar{\xi}_2^\pm=
\Bigl( \log (\mu^{\frac{3}{2}}h^{\frac{1}{2}})- 
(2n+2) \log \log (\mu^{\frac{3}{2}}h^{\frac{1}{2}})\pm C\Bigr)^{\frac{1}{2}}
\label{15-4-44}
\end{gather}
describing a zone where $\mu h |\lambda'_{*,n}(\xi_2)|\le \epsilon \hbar$ is dominated by $|\partial_{\xi_2}W(\xi_1 \hbar, x_2)|$. Then the same arguments work there albeit scale with respect to $\xi_2$ is now $\xi_2-\bar{\xi}_2^{+\prime}$ $\bar{\xi}_2^{\pm\prime}$ are defined by the same formula with $C$ replaced by $C/2$.

Further, these arguments remain valid for $\xi_2\ge \xi_2^-$ provided $\lambda'_{*,n}$ and $-W_{x_1}$ have the same sign.

Thus we arrive to 

\begin{proposition}\label{prop-15-4-9}
(i) Under conditions \textup{(\ref{15-4-1})} or  \textup{(\ref{15-4-2})}, \textup{(\ref{15-4-28})} and $\textup{(\ref{15-4-39})}^\pm$, $\textup{(\ref{15-4-40})}^\pm$ (with the same or opposite signs) fulfilled in $B(\bar{x},1)$ contribution of 
$B(\bar{x},\frac{1}{2})\cap \cX_\inn \cap\{\xi_2\ge C\}$ to the Tauberian remainder does not exceed $C$.

\medskip\noindent
(ii) Further as $\lambda'_{*,n}$ and $-W_{x_1}$  have the same signs\footnote{\label{foot-15-13} Which means that under Dirichlet or Neumann boundary  condition $\textup{(\ref{15-4-40})}^+$ and $\textup{(\ref{15-4-40})}^-$ respectively must be fulfilled.} (i) remains true for $\cX_\inn=\{\xi_2\ge \bar{\xi}^-_2\}$.
\end{proposition}

\paragraph{Boundary zone.}
\label{sect-15-4-3-3-2}
Consider now boundary zone temporarily introduced as 
\begin{equation*}
\cX_\bound \Def \{\xi_2\le \frac{1}{2}\bar{\xi}_2\}
\end{equation*}
which actually should be intersected with zone (\ref{15-4-34}).

Then we can take
\begin{equation}
T_*\asymp\min 
\hbar \Bigl(  \xi_2
\bigl(\mu h |\lambda'_{*,n}|\bigr)^{-1}, \ell^{-2}\Bigr).
\label{15-4-45}
\end{equation}
Really, in this zone propagation speed with respect to $x_2$ is 
$\asymp \mu h |\lambda'_{*,n}|$ and scale with respect to $\xi_2$ is 
$\asymp \xi_2^{-1}$ to keep $\lambda'_{*,n}$ of the fixed magnitude.

As $\ell \ge (\xi_2^{-1}\mu h |\lambda'_{*,n}|\bigr)^{\frac{1}{2}}$ consider propagation in the time direction of $\xi_2$ increasing. Then 
\begin{equation}
T^* \asymp   \min \bigl( (\mu h |\lambda'_{*,n}| )^{-1}, \hbar^{-1}\ell\bigr)
\label{15-4-46}
\end{equation}
provided $|W_{x_2}|$ also increases in the same direction because speed with respect to $x_2$ is $\asymp (\mu h |\lambda'_{*,n}| )$ as long as we remain in $\cX_\bound$ but then drops to $O(\hbar)$ outside of $\cX_\bound$  but $x_2$ can go in the opposite direction there. Our extra assumption means exactly that 
\begin{claim}\label{15-4-47}
$W_{x_2x_2}$ and $\lambda'_{*,n}$ have the same signs\footnote{\label{foot-15-14} Or equivalently that under Dirichlet or Neumann boundary condition $\textup{(\ref{15-2-39})}^-$ and $\textup{(\ref{15-2-39})}^+$ respectively must be fulfilled.}.
\end{claim}
Then contribution of the subzone in question to the Tauberian remainder does not exceed 
\begin{multline*}
C\iint \hbar^{-1}\frac{T_*(\xi_2)}{T^*(\xi_2)}\,d\xi_2 dx_2\asymp 
C\iint \ell^{-2} \bigl( \mu h |\lambda'_{*,n}|  + \hbar\ell^{-1}\bigr) \,d\ell d\xi_2\asymp\\
 \asymp 
C\int  \bigl( (\mu h |\lambda'_{*,n}|\xi_2)^{\frac{1}{2}}   + \hbar \xi_2 (\mu h |\lambda'_{*,n}|)^{-1}\bigr) \, d\xi_2 \asymp \\
C (\mu h |\lambda'_{*,n}|\xi_2)^{\frac{1}{2}}\xi_2^{-1}\bigr|_{\xi_2=\hat{\xi}_2}   + 
C\hbar  (\mu h |\lambda'_{*,n}|)^{-1}\bigr|_{\xi_2=\bar{\xi}_2} \asymp C
\end{multline*}
where $\hat{\xi}_2$ is defined from (\ref{15-4-34}) as solution to 
$\mu h |\lambda_{*,n}-(2n+1)|= C$.

Consider now subzone 
\begin{equation*}
\bigl\{\ell \le (\xi_2^{-1}\mu h |\lambda'_{*,n}|\bigr)^{\frac{1}{2}}, 
\ \ell \ge \xi_2 \hbar
\bigr\}
\end{equation*}
where the second condition is due to uncertainty principle and scaling with respect to $\xi_2$.

Then $T^*$ is defined by (\ref{15-4-45}) with $\ell$ replaced by
$(\xi_2^{-1}\mu h |\lambda'_{*,n}|\bigr)^{\frac{1}{2}}$ 
as we can go into direction of increasing $\xi_2$. Then contribution of this subzone to the Tauberian remainder does not exceed 
\begin{equation*}
C\iint \hbar^{-1}\frac{T_*(\xi_2)}{T^*(\xi_2)}\,d\xi_2 dx_2\asymp 
C\hbar \int  d\xi_2 \ll C
\end{equation*}
again.
Further, in subzone $\{\ell \le \hbar \xi_2\}$ one can take 
$T^*\asymp \xi_2^{-1}$ and its contribution is $\ll 1$ as well.

Consider now case when 

\begin{claim}\label{15-4-48}
$W_{x_2x_2}$ and $\lambda'_{*,n}$ have the opposite signs\footnote{\label{foot-15-15} Or equivalently that under Dirichlet or Neumann boundary condition $\textup{(\ref{15-2-39})}^+$ and $\textup{(\ref{15-2-39})}^-$ respectively must be fulfilled.}.
\end{claim}

Then as $\ell\ge \ell^*$ we can do nothing better than 
\begin{equation}
T^*= \ell \bigl(\mu h |\lambda'_{*,n}|\bigr)^{-1}.
\label{15-4-49}
\end{equation}
Then contribution of the subzone in question to the Tauberian remainder does not exceed 
\begin{multline*}
C\iint \hbar^{-1}\frac{T_*(\xi_2)}{T^*(\xi_2)}\,d\xi_2 dx_2\asymp 
C\iint \ell^{-3} ( \mu h |\lambda'_{*,n}|) \,d\ell d\xi_2\asymp\\
 \asymp 
C\int     d\xi_2 \asymp C\log \mu.
\end{multline*}
On the other hand, for $\ell\le \ell^*$ we can improve (\ref{15-4-49}) to
\begin{equation}
T^*= \ell^{1-\delta} \bigl(\mu h |\lambda'_{*,n}|\bigr)^{-1-\frac{1}{2}\delta}.
\label{15-4-50}
\end{equation}
Then contribution of the subzone in question to the Tauberian remainder does not exceed
\begin{multline*}
C\iint \hbar^{-1}\frac{T_*(\xi_2)}{T^*(\xi_2)}\,d\xi_2 dx_2\asymp 
C\iint \ell^{-1+\delta} \ell^{*\,-\delta}  \,d\ell d\xi_2\asymp\\
 \asymp 
C\int  d\xi_2 \asymp C\log \mu.
\end{multline*}

Finally, extending $\cX_\bound$ to 
\begin{equation}
\cX_\bound =\{\xi_2\le \bar{\xi}_2^-\}
\label{15-4-51}
\end{equation}
comes with no cost at all as scale with respect to $\xi_2$ is 
$\xi_2^{-1}\asymp (\bar{\xi}_2^{-\prime} -\xi_2)$ anyway.

So we arrive to

\begin{proposition}\label{prop-15-4-10}
(i) Under conditions \textup{(\ref{15-4-1})} or  \textup{(\ref{15-4-2})}, \textup{(\ref{15-4-28})} and $\textup{(\ref{15-4-39})}^\pm$, $\textup{(\ref{15-4-40})}^\pm$ (with the same or opposite signs) fulfilled in $B(\bar{x},1)$ contribution of 
$B(\bar{x},\frac{1}{2})\cap \cX_\bound \cap\{\xi_2\ge C\}$ to the Tauberian remainder does not exceed $C\log \mu$.

\medskip\noindent
(ii) Further as $\lambda'_{*,n}$ and $W_{x_2x_2}$  have the same signs\footref{foot-15-14} this contribution does not exceed $C$.
\end{proposition}

As $\lambda'_{*,n}$ and $-W_{x_1}$ have the same signs\footref{foot-15-13} we are done. Otherwise we need to consider 

\paragraph{Transitional zone.} 
\label{sect-15-4-3-3-3}

$\cX_\trans$ is defined by
\begin{equation}
\bar{\xi}_2^-\le \xi_2\le \bar{\xi}_2^+
\label{15-4-52}
\end{equation}
with $\bar{\xi}_2^\pm$ defined by (\ref{15-4-44}) and in this zone 
$\mu h\lambda'_{*,n}$ has magnitude $\hbar$ and $\mu h\lambda''_{*,n}$ has magnitude $\hbar \xi_2\gg \hbar^2$ and we conclude that

\begin{proposition}\label{prop-15-4-11}
Under conditions \textup{(\ref{15-4-1})} or  \textup{(\ref{15-4-2})}, \textup{(\ref{15-4-28})} and $\textup{(\ref{15-4-39})}^\pm$, $\textup{(\ref{15-4-40})}^\pm$ (with the same or opposite signs) fulfilled in $B(\bar{x},1)$ contribution of 
$B(\bar{x},\frac{1}{2})\cap \cX_\trans \cap\{\xi_2\ge C\}$ to the Tauberian remainder does not exceed $C\log \mu$.

\medskip\noindent
(ii) Further as $\lambda''_{*,n}$ and $W_{x_2x_2}$  have the same signs\footref{foot-15-14} this contribution does not exceed $C$.
\end{proposition}

\paragraph{Synthesis}
\label{sect-15-4-3-3-4}

So, we proved that the total Tauberian remainder does not exceed $C\log \mu$. To upgrade it to $C$ we must assume that $W_{x_2x_2}$, $\lambda'_{*,n}$ and $-W_{x_1}$ have the same signs as $\lambda'_{*,n}$ and $\lambda''_{*,n}$ have opposite signs. So we proved 

\begin{proposition}\label{prop-15-4-12}
Under conditions \textup{(\ref{15-4-1})} or  \textup{(\ref{15-4-2})}, \textup{(\ref{15-4-28})} and $\textup{(\ref{15-4-39})}^\pm$, $\textup{(\ref{15-4-40})}^\pm$ (with the same or opposite signs) fulfilled in $B(\bar{x},1)$ contribution of 
$B(\bar{x},\frac{1}{2})\cap\{\xi_2\ge C\}$ to the Tauberian remainder does not exceed $C\log \mu$.

\medskip\noindent
(ii) Further as  $W_{x_2x_2}$, $\lambda'_{*,n}$ and $-W_{x_1}$ have the same signs\footref{foot-15-14} this contribution does not exceed $C$.
\end{proposition}

\section{From Tauberian to magnetic Weyl}
\label{sect-15-4-4}

In this subsection $T=T_*\mu^\delta$ but we will calculate with $\delta=0$ because if we consider time interval $[T,2T]$ with $T\ge T_*$ in the estimate of its contribution there will be an extra factor $T^{-s}T_*^s$ and the summation with respect to $T$ will result in the same answer albeit with $T=T_*$.

We apply the successive approximation method to operator  (\ref{15-4-7})

\subsection{Inner zone}
\label{sect-15-4-4-1}
Replacing $x_2$ by $y_2$ leads to a contribution of $(\ell,\rho)$ element to an error does not exceeding
\begin{equation}
C \hbar^{-1} \rho \ell\times \frac{\hbar T_*^2}{\hbar}= 
C\hbar^{-1}\rho \ell T_*^2
\label{15-4-53}
\end{equation}
as propagation speed with respect to $x_2$ is $O(\hbar)$ in the time scale used here and we use time scale compatible with the choice of $T_*$ in the previous subsection.

Under non-degeneracy condition (\ref{15-2-8}) we can take $T_*= \hbar $ resulting in the contribution of $\cX_\inn$ equal to $O(1)$ as $\rho \asymp \hbar^{-1}$.

In more general generic case we estimate a contribution of $(\ell,\rho)$ element to an error  by
\begin{equation}
C \hbar^{-1} \rho \ell \times \frac{\hbar \ell T^2}{\hbar}= 
C\xi_2 \ell^2 T_*^2\asymp
C\hbar ^{-1}\xi_2\ell^2 \min \bigl(  \xi_2^{-2}, \hbar^2\ell^{-4}\bigr)
\label{15-4-54}
\end{equation}
where we plugged $T_*$ defined by (\ref{15-4-41}); an extra factor $\ell$ appears as we replace $x_2$ by $y_2$ because $\partial_{x_2}W=O(\ell)$ and at this moment we take original ``narrow'' $\cX_\inn$ with scale $\rho\asymp \xi_2$ in $\xi_2$.

Therefore the contribution of $\cX_\inn$ to an error does not exceed
\begin{equation}
C \hbar^{-1}\iint  \min \bigl(   \xi_2^{-2}, \hbar^2\ell^{-4}\bigr)\ell \,d\ell d\xi_2\asymp
C \int \xi_2 ^{-1}\,d\xi_2 \asymp C\log \mu.
\label{15-4-55}
\end{equation}

In the remaining part of $\cX_\inn$ defined by (\ref{15-4-44}) we must replace $\xi_2$ by $\rho \asymp \xi_2-\bar{\xi}_2$ as the scaling is concerned and it redefines correspondingly $T_*$ and instead of right-hand expression in (\ref{15-4-55}) we get  $C \int (\xi_2 -\bar{\xi}_2)^{-1}\,d\xi_2$ which is also $\asymp C\log \mu$.

\subsection{Boundary zone} 
\label{sect-15-4-4-2}
Due to the standard arguments contribution of subzone $\{\xi_2\le C_0\}$ to the error is either $O(\log \mu)$ or $O(1)$ depending if there is a saddle point or no; one can see easily that this saddle points cannot come from $\lambda_{\N,m}$ but only from $\lambda_{\N,n}$ with $n\ne m$ and it is possible only as $\lambda h\asymp 1$.

Consider subzone $\{\xi_2\ge C_0\}$. Here we need to replace (\ref{15-4-53}) by
\begin{equation}
C \hbar^{-1} \rho \ell\times \frac{\mu h| \lambda'_{*,m}| T_*^2}{\hbar}= C\hbar^{-2}\mu h|\lambda'_{*,m}| \rho \ell T_*^2
\label{15-4-56}
\end{equation}
as propagation speed with respect to $x_2$ is $\asymp \mu h| \lambda'_{*,m}$. Plugging under condition (\ref{15-2-8}) $\ell=1$, $\rho=\xi_2^{-1}$ 
\begin{equation*}
T_* =  \hbar \min \bigl( \xi_2(\mu h|\lambda'_{*,m}|)^{-1}, 1\bigr)
\end{equation*}
we arrive to
\begin{equation*}
C\mu h|\lambda'_{*,m}| \rho \min 
\bigl( \xi_2^2(\mu h|\lambda'_{*,m}|)^{-2}, 1\bigr)
\end{equation*}
and contribution of both subzones $\{\xi_2\ge C,\ \mu h|\lambda'_{*,m}|\ge \xi_2\}$ and
$\{C\hbar \le \mu h|\lambda'_{*,m}|\le \xi_2\hbar\}$  are $O(1)$.

In the same manner  (\ref{15-4-54}) is replaced by 
\begin{equation}
C h^{-1} \rho \ell^2 \times \frac{\mu h| \lambda'_{*,m}| T^2}{\hbar}\asymp
C  \hbar^{-2} \mu h | \lambda'_{*,m}| \rho \ell^2 T_*^2
\label{15-4-57}
\end{equation}
and plugging
\begin{equation*}
T_* =  \hbar \min \bigl( \xi_2(\mu h|\lambda'_{*,m}|)^{-1}, \ell^{-2}\bigr)
\end{equation*}
we arrive to
\begin{equation*}
C   \mu h | \lambda'_{*,m}| \rho \ell^2  \min \bigl( \xi_2^2(\mu h|\lambda'_{*,m}|)^{-2}, \ell^{-4}\bigr)
\end{equation*}
and summation with respect to $\ell$ results in 
\begin{equation*}
C   \mu h | \lambda'_{*,m}| \rho \ell^{*,-2}  = C\rho \xi_2
\end{equation*}
as $\ell^*= \xi_2^{-1}\mu h | \lambda'_{*,m}|$ and summation with respect to $\xi_2$ results in $C\bar{\xi}_2^{-\,2}\asymp C\log \mu$.

\subsection{Transitional zone} 
\label{sect-15-4-4-3}
In this zone we can return to (\ref{15-4-53}) and plug $T_*= \hbar$ under condition (\ref{15-2-8}) resulting in $O(1)$.

Further we can return to (\ref{15-4-54}) and plug in the generic case either 
$T_*= \min \bigl(\hbar\ell^{-2}, \xi_2^{-1}\bigr)$ or
$T_*= \min (\hbar\ell^{-2}, |\xi_2-\bar{\xi}_2|^{-1})$ 
resulting after summation with respect to $\ell$  in $C\xi_2^{-1}\rho$ or $C|\xi_2-\bar{\xi}_2|^{-1}\rho$ respectively. In the former case we instantly get $O(1)$ while in the latter  after summation with respect to subzone $\{|\xi_2-\bar{\xi}_2|\ge \hbar\}$ we arrive to $O(\log \mu)$ while contribution of subzone $\{|\xi_2-\bar{\xi}_2|\le \hbar, \ell \le \hbar\}$ to asymptotics is $O(1)$.

\section{Justification: reduction to model operator}
\label{sect-15-4-5}

\subsection{Reduction: step 1}
\label{sect-15-4-5-1}
The problem however is that the ``kinetic'' part of our operator is not exactly $h^2D_1^2 + (hD_2-\mu x_1)^2$ but is different. We need to improve construction of subsection~\ref{sect-15-2-1}). First we can assume that $V_1=0$ and that
$g^{jk}=\updelta_{jk}$, $V_2=0$, $\partial_1V_2=1$ as $x_1=0$. Further, assume that 
\begin{equation}
g^{12}=O(x_1^k),\qquad V_2= x_1+O(x_1^{k+1})
\label{15-4-58}
\end{equation}
with $k\ge 1$. Then changing $x_1\mapsto x_1$, $x_2\mapsto x_2+\alpha(x_2) x_1^{k+1}$ we can achieve $g^{12}=O(x_1^{k+1})$, simultaneously preserving the second relation. Then redefining $x_1\Def V_2(x)$ preserves the first relation (with $k+1$ instead of $k$). Continuing we can make $k$ as large as we wish and thus $k=\infty$.

One needs to remember that $g^{11}$ and $g^{22}$ differ from $1$ by $O(x_1)$. However as we consider operator $F^{-1}A$ with the intensity $1$ of magnetic field we conclude that $g^{11}g^{22}=1+O(x_1^\infty)$.

Our goal is to get rid off this perturbation which we rewrite as 
\begin{equation}
\cP^\w\Def \mu^2   \Bigl( \hslash^2 D_1\sigma D_1 - \sigma (\hslash D_2- x_1)\sigma  (\hslash D_2- x_1)\Bigr).
\label{15-4-59}
\end{equation}
Consider Poisson brackets
\begin{equation}
\frac{1}{2}
\bigr\{ \bigl( \xi_1^2 + (\xi_2- x_1)^2\bigr), 
\alpha +\beta  \xi_1\bigr\}
\label{15-4-60}
\end{equation}
and note that for coefficients at $\xi_1$ equal $0$ i.e.
\begin{equation}
\alpha_{x_1}-(x_1-\xi_2)\beta_{x_2}=0
\label{15-4-61}
\end{equation}
it is equal to 
\begin{equation}
\beta_{x_1}   \xi_1^2 -
(\beta+\alpha_{x_2})(x_1-\xi_2).
\label{15-4-62}
\end{equation}
To make it equal to (\ref{15-4-59}) modulo $\xi_1^2 + (x_1-\xi_2)^2$ one needs to satisfy 
\begin{equation}
\alpha_{x_2}=  2\sigma (x_1-\xi_2)-\beta_{x_1}(x_1-\xi_2)-\beta.
\label{15-4-63}
\end{equation}
Compatibility with (\ref{15-4-61}) requires 
\begin{equation}
(\partial_{x_1}^2+\partial_{x_2}^2)  \bigl((x_1-\xi_2)\beta \bigr) = 2\partial_{x_1}\bigl((x_1-\xi_2)\sigma\bigr).
\label{15-4-64}
\end{equation}
Consider instead it without $\partial_{x_2}$, with an extra condition ``$\beta=0$ as $x_1=0$''. We can solve this one-dimensional equation
\begin{equation}
\partial_{x_1}^2 \bigl( (x_1-\xi_2)\beta\bigr) = 2\partial_{x_1}\bigl((x_1-\xi_2)\sigma\bigr),
\label{15-4-65}
\end{equation}
then 
\begin{equation}
\beta (x_1,.)= 
2(x_1-\xi_2)^{-1}\int_{\xi_2}^{x_1} (x'_1-\xi_2)\sigma (x_1',.)dx'_1 + \rho 
\label{15-4-66}
\end{equation} 
with $\rho$ which does not depend on $x_1$; extra condition ``$\beta=0$ as $x_1=0$'' means that
\begin{equation}
\rho = -\xi_2^{-1}\int_0^{\xi_2} (x'_1-\xi_2)\sigma (x_1',.)dx'_1
\label{15-4-67}
\end{equation}
which gives us  smooth function $\beta$ and $\alpha$.

\begin{remark}\label{rem-15-4-13}
As $\alpha_{x_2\xi_2}=0$ as $x_1=\xi_2=0$ due to (\ref{15-4-63}) and properties of $\sigma,\beta$ we can select $\alpha$ such that $\alpha_{\xi_2}=0$ there as well. This would enable further calculations.
\end{remark}
Then 
\begin{multline}
\frac{1}{2}
\bigr\{ \bigl( \xi_1^2 + (\xi_2- x_1)^2\bigr), 
\alpha +\beta  \xi_1\bigr\}=\\ (\beta_{x_1}-\sigma)\bigl(\xi_1^2+(x_1-\xi_2)^2\bigr)
+ \sigma\bigl(\xi_1^2-(x_1-\xi_2)^2\bigr).
\label{15-4-68}
\end{multline}
Let us pass to operators; to do this we consider corresponding $\hslash$-quantizations, divide by $-i\hslash$ and multiply by $\mu^2$; then
\begin{gather}
-\frac{1}{i\hslash}[A,\cL^\w]= \cP^\w-\cP^\w_1,\label{15-4-69}\\
\shortintertext{with}
A\Def \hslash^2 D_1^2 + (\hslash D_2-  x_1)^2,\label{15-4-70}\\
\cL\Def\frac{1}{2}(\alpha +\beta\xi_1),\label{15-4-71}\\
\cP_1=(\sigma-\beta_{x_1}\sigma)\bigl(\xi_1^2 +(x_1-\xi_2)^2\bigr)
\label{15-4-72}
\end{gather}
where as we consider Weyl quantizations and symmetrized products after multiplication by $\mu^2$ the error is $O(\mu^2\hslash^2)=O(h^2)$.

One can see easily that in our settings $\sigma=\beta=0$ as $x_1=0$ and $\beta$, $\alpha_{x_2}$ has $0$ of the second order as $x_1=\xi_1=\xi_2=0$ while $\alpha_{x_1}$ has  $0$ of the third order as $x_1=\xi_1=\xi_2=0$. Also note that due to (\ref{15-4-64}) $\beta_{x_1}-\sigma=0$ as $x_1=\xi_2$ so 
\begin{claim}\label{15-4-73}
symbol  $(\sigma-\beta_{x_1})$ is divisible by $(x_1-\xi_2)$.
\end{claim}

Consider transformation of $A+\cP^\w$ by $\exp (i\hslash^{-1}\cL^\w)$:
\begin{multline} 
\exp (i\hslash^{-1}\cL^\w)\, (A+\cP^\w)\exp (-i\hslash^{-1}\cL^\w)=\\
A+\cP^\w + \frac{1}{i\hslash}[A+\cP^\w,\cL^\w]+
\frac{1}{2}\bigl(\frac{1}{i\hslash}\bigr)^2[[A+\cP^\w,\cL^\w],\cL^\w]+\ldots
\label{15-4-74}
\end{multline}
where as one can see easily we left out terms which are $o(h^2)$ and due to (\ref{15-4-69})
\begin{equation}
\equiv A+\cP_1^\w + 
\frac{1}{2}\bigl(\frac{1}{i\hslash}\bigr)[\cP^\w,\cL^\w]+
\frac{1}{2}\bigl(\frac{1}{i\hslash}\bigr)[\cP^\w_1,\cL^\w]
\label{15-4-75}
\end{equation}
Consider first the symbol of the third term 
\begin{multline*}
-\frac{1}{4}\{\cP , \alpha+\beta\xi_1\}= \\
-\frac{1}{4}\{\sigma, \alpha +\beta\xi_1\}\bigl(\xi_1^2-(x_1-\xi_2)^2\bigr)
-\frac{1}{4} \sigma \{\xi_1^2-(x_1-\xi_2)^2, \alpha+\beta\xi_1\};
\end{multline*}
as 
\begin{multline*}
\frac{1}{2}\{\xi_1^2-(x_1-\xi_2)^2, \alpha+\beta\xi_1\}=
\bigl(\alpha_{x_1}+\beta_{x_2}(x_1-\xi_2)\bigr)\xi_1 + \beta_{x_1}\xi_1^2 + (x_1-\xi_2)(\alpha_{x_2}+\beta)=\\[2pt]
2\beta_{x_2}(x_1-\xi_2)\xi_1 +\sigma \bigl(\xi_1^2+(x_1-\xi_2)^2\bigr)-
(\sigma-\beta_{x_1})\bigl(\xi_1^2-(x_1-\xi_2)^2\bigr)
\end{multline*}
we conclude that
\begin{multline*}
-\frac{1}{4}\{\cP , \alpha+\beta\xi_1\}= 
\Bigl(-\frac{1}{4}\{\sigma, \alpha +\beta\xi_1\}+
\frac{1}{2}\sigma(\sigma-\beta_{x_1})\Bigr)\bigl(\xi_1^2-(x_1-\xi_2)^2\bigr)-\\
\sigma\beta_{x_2}(x_1-\xi_2)\xi_1 
-\frac{1}{2}\sigma^2 \bigl(\xi_1^2+(x_1-\xi_2)^2\bigr).
\end{multline*}
Consider now the symbol of the fourth term in (\ref{15-4-75}):
\begin{multline*}
-\frac{1}{4}\{\cP_1 , \alpha+\beta\xi_1\}= \\
-\frac{1}{4}\{(\sigma-\beta_{x_1}) , \alpha+\beta\xi_1\}\bigl(\xi_1^2+(x_1-\xi_2)^2\bigr)
-\frac{1}{4} (\sigma-\beta_{x_1}) \{\xi_1^2+(x_1-\xi_2)^2, \alpha+\beta\xi_1\};
\end{multline*}
as we know the last Poisson bracket we conclude that
\begin{multline*}
-\frac{1}{4}\{\cP_1 , \alpha+\beta\xi_1\}= 
\Bigl(-\frac{1}{4}\{(\sigma-\beta_{x_1}) , \alpha+\beta\xi_1\}+\frac{1}{2}(\sigma-\beta_{x_1}) ^2\Bigr)
\bigl(\xi_1^2+(x_1-\xi_2)^2\bigr)\\
-\frac{1}{2} (\sigma-\beta_{x_1}) \sigma \bigl(\xi_1^2-(x_1-\xi_2)^2\bigr).
\end{multline*}
Then the sum of symbols of the third and the fourth terms in (\ref{15-4-75}) is equal to
\begin{multline}
-\frac{1}{4}\{\sigma, \alpha +\beta\xi_1\}\bigl(\xi_1^2-(x_1-\xi_2)^2\bigr)-
{\color{gray}\sigma\beta_{x_2}(x_1-\xi_2)\xi_1} +\\
\Bigl(-\frac{1}{4}\{(\sigma-\beta_{x_1}) , \alpha+\beta\xi_1\}+ \frac{1}{2}\bigl({\color{gray}-\sigma^2}+(\sigma-\beta_{x_1}) ^2\bigr)\Bigr)
\bigl(\xi_1^2+(x_1-\xi_2)^2\bigr).
\label{15-4-76}
\end{multline}

\begin{remark}\label{rem-15-4-14}
However correction (\ref{15-4-59}) is not a correct one as we need $g^{11}g^{22}=1$ and a more precisely is $\cP_1^\w +\frac{1}{2}\sigma^2A$ as
$g^{11}=1+\sigma+\frac{1}{2}\sigma^2+\dots$, $g^{22}=1-\sigma+\frac{1}{2}\sigma^2+\dots$ .
\end{remark}

We can skip terms here which has $0$ of degree 5 at $x_1=\xi_1=\xi_2=0$; this takes care of the middle term. 

By the same reason we can skip in
\begin{equation*}
-\frac{1}{4}\{\sigma,\alpha+\beta \xi_1\}= \frac{1}{4}\sigma_{x_2}\alpha_{\xi_2}+{\color{gray}\frac{1}{4}\sigma_{x_2}\beta_{\xi_2}\xi_1}+\frac{1}{4}\sigma_{x_1}\beta 
\end{equation*}
middle term where we used that $\sigma=\sigma(x)$. 

Consider two remaining terms (which have $0$ of the second order at $x_1=\xi_2=0$ and also equal $0$ as $x_1=0$ )
\begin{equation}
\sigma_1\Def 
\frac{1}{4}\sigma_{x_2}\alpha_{\xi_2}+\frac{1}{4}\sigma_{x_1}\beta ;
\label{15-4-77}
\end{equation}
then we can repeat the same procedure as before thus replacing the whole first term in (\ref{15-4-76}) by
\begin{equation*}
(\sigma_1 -\beta_{1\,x_1})\bigl(\xi_1^2+(x_1-\xi_2)^2\bigr)
\end{equation*}
with $\beta_1$ matching $\sigma_1$. One can see easily that then neither $\cP_1$ nor the last term in (\ref{15-4-76}) would  change modulo $O(h^2)$ and therefore we arrive to operator $A+\cP_2^\w$ with 
\begin{multline}
\cP_2=\\ 
\Bigl((\sigma-\beta_{x_1})+\sigma_1-\beta_{1\,x_1} -\frac{1}{4}\{(\sigma-\beta_{x_1}) , \alpha+\beta\xi_1\}+
\frac{1}{2}(\sigma-\beta_{x_1}) ^2\Bigr) \bigl(\xi_1^2+(x_1-\xi_2)^2\bigr).
\label{15-4-78}
\end{multline}
In this factor the first term has $0$ of the first order, other terms have $0$ of the second order or higher and we can skip those which are higher. We can also skip terms divisible by $(x_1-\xi_2)^2$ as those terms are $O(\hslash)$ on energy levels in question leading to $O(h^2)$ error in the final answer. Due to (\ref{15-4-73}) we arrive to
\begin{equation*}
(\sigma-\beta_{x_1})+(\sigma_1-\beta_{1\,x_1}) -
\frac{1}{4}\{(\sigma-\beta_{x_1}) , \alpha+\beta\xi_1\}
\end{equation*}
where the last term is 
\begin{equation*}
-\frac{1}{4}\{(\sigma-\beta_{x_1}) , \alpha\}-
{\color{gray}\frac{1}{4}\{(\sigma-\beta_{x_1}) , \beta\}\xi_1}
+(\sigma-\beta_{x_1})_{x_1}\beta.
\end{equation*}
One can prove easily that due to (\ref{15-4-63}), (\ref{15-4-73}) (which is valid for $\sigma_1,\beta_1$ as well) the remaining terms vanish as  $x_1=\xi_2$. 

Therefore as we skip terms containing factor $(x_1-\xi_2)^2$, what is left of perturbation is $\cP_3^\w$ with
\begin{equation}
\cP_3\Def \rho(x_2,\xi_2) (x_1-\xi_2) \bigl(\xi_1^2+(x_1-\xi_2)^2\bigr),
\label{15-4-79}
\end{equation}
so, $\cP_3^\w$ is a symmetric product of $\rho ^\w$, $(x_1-\mu^{-1}hD_2)$ and $A$ (as factor $\mu^2$ should be remembered).

\subsection{Reduction: step 2}
\label{sect-15-4-5-2}

This perturbation is of the same magnitude as the original perturbation $\cP^\w$ (albeit $\hbar^{-\delta}$ in estimates does not pop-up) but it is much easier to handle. If we had no boundary\footnote{\label{foot-15-16} Or were on the distance at least $\hbar^{\frac{1}{2}-\delta}$ from it.} we would use 
$\cL= \omega (x_2,\xi_2) \xi_1 \bigl(\xi_1^2+(x_1-\xi_2)^2\bigr)$ to eliminate it generating lower order terms but closer to  factor $\xi_1$ forbids it unless accompanied by factor $x_1$. Still replacing $\rho(x_2,\xi_2)$ in (\ref{15-4-79}) by  
\begin{equation*}
\rho(x_2,\xi_2)= \rho(x_2,0)- \rho_{\xi_2}x_1 +\rho_{\xi_2}(x_1-\xi_2) + +\dots
\end{equation*}
we can eliminate the second term (producing exactly such component in $\cP_3$) and skip the third and all terms as producing $O(h^2)$ contribution. So 

\begin{remark}\label{rem-15-4-15}
Without any loss of the generality one can assume that $\rho=\rho(x_2)$ in (\ref{15-4-79}).
\end{remark}

Still it does not help us much and we cannot use obvious transformation by
$\exp(i\hslash^{-1}\cL^\w)$ with 
$\cL=  \omega (x_2)\bigl(\xi_1^2+(x_1-\xi_2)^2\bigr)$ due to trouble in the next commutators. Instead we recall that basically our operator (\ref{15-4-70}) multiplied by $\mu^2$ should be equal $W +\mu h \fz$ with constant 
$\fz$. Actually one needs to insert there also $-hD_t$ but at energy level $0$ it is really small and we can actually by multiplication by $(I+Q)$ make our evolution equation linear with respect to $hD_t$. We leave absolutely standard detailed arguments to the reader.

So now perturbation becomes
\begin{equation}
\cP_4^\w = \mu h \rho(x_2)(x_1-\hslash D_2)\fz  + \rho(x_2)(x_1-\hslash D_2) \fz 
\label{15-4-80}
\end{equation}
and it is completely different game as we can get rid of the first term using 
transformation by $\exp(i\hslash^{-1}\cL^\w)$ with 
$\cL= \mu h  \omega (x_2)\fz $ as one can see easily that there will be no problems with the next commutators. Remaining perturbation is of the type
$\bigl(W_1(x,\xi)(x_1-\xi_2)\bigr)^\w$ and is of magnitude $\hbar$. So we gained factor $(\mu h)^{-1}$ which is important only in the case of superstrong magnetic field so we can skip this  step for a very strong magnetic field. 

\subsection{Reduction: step 3}
\label{sect-15-4-5-3}

Note that the only part of new perturbation which is not necessarily $O(h^2)$ comes from $\bigl(W_1(x_2,\xi_2)(x_1-\xi_2)\bigr)^\w$ and we can use the same approach as before to eliminate it.

Actually we can continue further. We can keep power of $\hslash D_1$ below $2$  replacing $\hslash^2 D_1^2$ by $-(x_1-\hslash^2 D_2)^2+ \hslash \fz -\mu^{-2}W$. We can eliminate $\hslash D_1$ transforming by  $\exp(i\hslash^{-1}\cL^\w)$ with $\cL = \rho (x_1-\xi_2)$. Similarly, we can eliminate $x_1^p (x_1-\xi_2)^q$ with $p\ge q$ transforming by similar operator with each $\xi_1$ having cofactor $x_1$. Therefore having $x_1^n$ with $n\ge 2$ we can decompose it into sum of terms of the type
$x_1^p (x_1-\hslash D_2)^q (\hslash D_2)^{n-p-q}$ where either $p\ge q$ or $p=0,q=1$ or $p=0,q=0$ and eliminate the first type by the method we just discussed and the second one transforming by $\exp(i\hslash^{-1}\cL^\w)$ with $\cL= \cL(x_2,\xi_2)$. It eliminates all powers -- including those we skipped before as lesser than $O(h^2)$.

\begin{remark}\label{rem-15-4-16}
(i) So in the end we arrive to the model operator. Sure $W$ is going to be perturbed by $O(h^2)$ albeit it may be because of smallness of $\xi_2$; so actually we get in the end $W_\eff(\xi_2,x_2)= -(V/F)(\xi_2,x_2)+O(\xi_2^3+h^2)$ albeit this does not change it basic properties.

\medskip\noindent
(ii) Recall that in this subsection $\xi_2$ in contrast to the rest of section corresponds to $\hslash D_2$ rather than $\hbar D_2$ and $\hslash=\hbar^2$.
\end{remark}

\begin{conjecture}\label{conj-15-4-17}
$W_\eff(\xi_2,x_2)= -(V/F)(\xi_2,x_2)+O(\xi_2^\infty+h^2)$
as $\xi_2<\hbar^{\delta}$.
\end{conjecture}

\section{Final results}
\label{sect-15-4-6}

Thus we arrive to the final contribution of zone $\{\xi_2\le \hbar^{-\delta}\}$ (or equivalently $\{x_1\le \hbar^{\frac{1}{2}-\delta}\}$) to the asymptotics:
\begin{equation*}
\sum_{n\ge 0} (2\pi \hbar)^{-1}\mu
\iint \uptheta \bigl(\mu h \bigl(\lambda _{*,n}(\xi_2)-\fz \bigr) -
W_\eff (\hbar \xi_2,x_2)\bigr)\psi(\xi_2\hbar^{-1},x_2) \zeta (\xi_2\hbar^{\delta})\,dx_2d\xi_2
\end{equation*}
where transformation of $\psi$ is obvious. After change of variables 
$x_1=\hbar \xi_2$ 
\begin{equation*}
\sum_{n\ge 0} (2\pi h)^{-1}\ 
\iint \uptheta \bigl(\mu h \bigl(\lambda _{*,n}(x_1\hbar^{-1})-\fz \bigr)- W_\eff (x_1,x_2)\bigr)\psi(x)\zeta (x_1\varepsilon^{-1})\,dx_1dx_2
\end{equation*}
with $\varepsilon =\hbar^{1-\delta}$ and $\hbar= (\mu^{-1}h)^{\frac{1}{2}}$. Here $\zeta \in \sC_0^\infty([-1,1])$ equals $1$ on $[-\frac{1}{2},\frac{1}{2}]$.

On the other hand,  the final contribution of zone $\{\xi_2\ge \hbar^{-\delta}\}$ (or equivalently $\{x_1\ge \hbar^{\frac{1}{2}-\delta}\}$) to the asymptotics is
\begin{equation*}
\sum_{n\ge 0} (2\pi h)^{-1}\mu  
\iint \uptheta \bigl(\mu h \bigl(2n+1)-\fz \bigr)-W_\eff(x_1,x_2)\bigr) \psi(x)\bigl(1-\zeta (x_1\varepsilon^{-1})\bigr)\,dx_1dx_2
\end{equation*}
and the final answer is 
\begin{multline*}
\sum_{n\ge 0} (2\pi h)^{-1}\mu\!  \iint \uptheta \bigl(\mu h \bigl(2n+1)-\fz \bigr)-W_\eff (x_1,x_2)\bigr)\psi(x)dx_1dx_2+\\
\sum_{n\ge 0} (2\pi h)^{-1}\mu \iint \Bigl(\uptheta \bigl(\mu h \bigl(\lambda _{*,n}(x_1\hbar^{-1})-\fz \bigr)-W_\eff (x_1,x_2)\bigr)-\\
\uptheta \bigl(\mu h \bigl(2n+1-\fz \bigr)-W_\eff (x_1,x_2)\bigr)\Bigr)
\psi(x)\zeta (x_1\varepsilon^{-1})\,dx_1dx_2.
\end{multline*}
where now obviously we can take $\varepsilon = C\hbar |\log \hbar|$ and even $\varepsilon=C\hbar$ as we are in the spectral gap situation. 

This matches (\ref{15-3-17}) with
\begin{multline}
\cN^\MW_{*,\bound}\Def \\
\sum_n (2\pi )^{-1}\mu \iint \Bigl(\uptheta \bigl(\mu h \bigl(\lambda _{*,n}(x_1\hbar^{-1})-\fz \bigr)-W_\eff (x_1,x_2)\bigr)-\\
\uptheta \bigl(\mu h \bigl(2n+1-\fz \bigr)-W_\eff (x_1,x_2)\bigr)\Bigr)
\psi(x)\zeta (x_1\varepsilon^{-1})\,dx_1dx_2.
\label{15-4-81}
\end{multline}

Therefore we arrive to our final theorem

\begin{theorem}\label{thm-15-4-18}
Let $\psi\in \sC^\infty (\bar{X})$ be a fixed function with a compact support contained in the small vicinity of $\partial X$ and let conditions  \textup{(\ref{13-2-1})}, \textup{(\ref{15-2-7})} and \textup{(\ref{15-2-41})} be fulfilled there. Then

\medskip\noindent
(i)  $\R^\MW$ with $\cN^\MW_\bound$ defined by \textup{(\ref{15-4-81})} does not exceed $C\log \mu$; 

\medskip\noindent
(ii)  This estimate could be improved to $O(1)$ in each of the following cases:

\begin{enumerate}[label=(\alph*), leftmargin=*]

\item Assumption \textup{(\ref{15-2-8})} is fulfilled;

\item Assumption \textup{(\ref{15-4-9})} and in case of Neumann boundary problem one of conditions \textup{(\ref{15-4-16})}, $\textup{(\ref{15-4-19})}^+$ is fulfilled;

\item Assumption \textup{(\ref{15-4-9})} fails but there is Dirichlet boundary problem and both conditions $\textup{(\ref{15-2-41})}^-$ and $\textup{(\ref{15-2-7})}^+$ are fulfilled;

\item Assumption \textup{(\ref{15-4-9})} fails but there is Neumann boundary problem and both conditions $\textup{(\ref{15-2-41})}^+$ and $\textup{(\ref{15-2-7})}^-$ are fulfilled.
\end{enumerate}
\end{theorem}

\begin{remark}\label{rem-15-4-19}
Here we redefined $\cN^\MW_{*,\bound}$ including in it in contrast to what we used before (see f.e. theorem~\ref{thm-15-2-6}) $W(x_1,x_2)$ rather than $W(0,x_2)$. 

We leave to the reader to prove that replacing $W(x_1,x_2)$ by $W(0,x_2)$ leads to an error $O(1)$ under condition (\ref{15-4-8}) and $O(\hbar^{-\frac{1}{2}}|\log \hbar|^L)$ under condition (\ref{15-2-41}).
\end{remark}

\chapter{Generalizations}
\label{sect-15-5}

\section{Robin boundary value problem}
\label{sect-15-5-1}

Consider boundary condition 
\begin{equation} 
\bigl(i \sum_{j,k} \nu_j g^{jk}P_k -\alpha \bigr)u\bigr|_{\partial X}=0
\label{15-5-1}
\end{equation}
with real-valued $\alpha=\alpha(x)\ge 0$ where $\nu=(\nu_1,\ldots,\nu_d)$ is an inner normal.

As $\alpha=0$ we get Neumann boundary condition and in a formal limit 
$\alpha\to +\infty$ we get Dirichlet boundary condition.

\paragraph{Weak magnetic field}
Obviously arguments of section~\ref{sect-15-2} remain valid. Recall that in that section we have not distinguished between Dirichlet and Neumann boundary conditions.

\paragraph{Strong magnetic field}
Obviously arguments of section~\ref{sect-15-3} remain valid.  Further, under assumption $\alpha \ge \epsilon_0$ propagation theorem~\ref{thm-15-3-15} remains valid; really then boundary problem satisfies Lopatinski condition. Therefore under this assumption one can derive the same remainder estimates as under Dirichlet boundary condition. 

\paragraph{Very strong and superstrong magnetic field}
On the contrary, in this case we are practically in frames of Neumann boundary condition: see remark~\ref{rem-15-A-8}. 

\section{Boundary meets degeneration of $V$}
\label{sect-15-5-2}

The following problem seems to be rather easy and straightforward:

\begin{problem}\label{problem-15-5-1}
Get rid off condition (\ref{15-2-12}) $V\le -\epsilon$ as $\mu h\le 1$ (we have not imposed this condition as $\mu h\ge 1$).
\end{problem}

One should apply  rescaling method and follow arguments of subsection~\ref{book_new-sect-13-7-1} \cite{futurebook}.

\section{Domains with corners}
\label{sect-15-5-3}

The following problem is far more interesting and challenging and I strongly believe that the detailed analysis merits a publication:

\begin{problem}\label{problem-15-5-2}
Consider domains with corners $\ne 0,\pi, 2\pi$ assuming that coefficients belong to $\sC^\infty(X)$, all other assumptions are fulfilled and in the corners non-degeneracy condition $|\nabla_{\partial X} W|\asymp 1$ holds for each of two directions along $\partial X$.
\end{problem}

Without any loss of the generality one can assume that locally either $X=\{x_1>0,x_2>0\}$ or $X=\bR^2\setminus\{x_1<0,x_2<0\}$. Then non-degeneracy condition means that
\begin{equation}
|W_{x_1}|\asymp 1,\qquad |W_{x_2}|\asymp 1.
\label{15-5-2}
\end{equation}
We also assume that (\ref{15-2-12})  is fulfilled.

\subsection{Weak magnetic field case}
\label{sect-15-5-3-1} 
One can use rescaling method with $\ell =\epsilon |x|$  as long as  $h_\eff=h\ell^{-1}\le 1$, $\mu _\eff=\mu \ell\ge 1$. In the inner zone $\cX_\inn$ (defined exactly as before) magnetic drift is the only way to break periodicity and one can see easily that (\ref{15-3-21}) should be replaced by 
\begin{equation}
T_*\Def h^{1-\delta}\ell^{-1}
\label{15-5-3}
\end{equation}
Then we are in frames of the weak magnetic field approach as long as 
\begin{equation}
\ell\ge \bar{\ell}\Def \max(C\mu^{-1},\mu h^{1-\delta}).
\label{15-5-4}
\end{equation}
Obviously we can take 
$T^*=\epsilon \mu \gamma$ with $\gamma \Def \min(|x_1|,|x_2|)$. To increase it consider dynamics on Figure~\ref{fig-traj4}.

\begin{figure}[h!]
\centering
\subfloat[${W_{x_1}>0,W_{x_2}>0}$]{\includegraphics[width=0.23\linewidth]{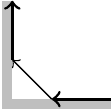}}\ 
\subfloat[$W_{x_1}>0,W_{x_2}<0$]{\includegraphics[width=0.23\linewidth]{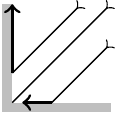}}\ 
\subfloat[$W_{x_1}<0,W_{x_2}>0$]{\includegraphics[width=0.23\linewidth]{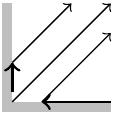}}\ 
\subfloat[$W_{x_1}<0,W_{x_2}<0$]{\includegraphics[width=0.23\linewidth]{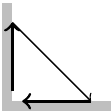}} 
\\
\subfloat[${W_{x_1}<0,W_{x_2}<0}$]{\includegraphics[width=0.23\linewidth]{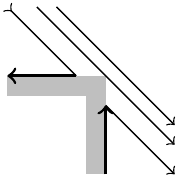}}\ 
\subfloat[$W_{x_1}>0,W_{x_2}>0$]{\includegraphics[width=0.23\linewidth]{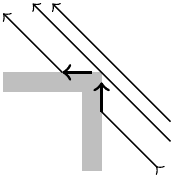}}\ 
\subfloat[$W_{x_1}<0,W_{x_2}>0$]{\includegraphics[width=0.23\linewidth]{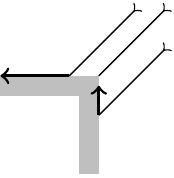}}\ 
\subfloat[$W_{x_1}>0,W_{x_2}<0$]{\includegraphics[width=0.23\linewidth]{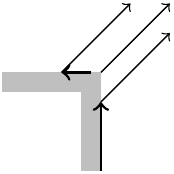}} 
\caption{\label{fig-traj4} Dynamics in the  corner with angle $<\pi$ (a)-(d) and $>\pi$ (e)-(h). Bold lines show hop movement and thin lines show drift movement which is along level lines of $W$. Consistent with Figure~\ref{fig-traj}.}
\end{figure}

Selecting a proper direction (away from the nearest boundary) we can take there $T^*=\epsilon \mu \ell$  and the contribution of  $\cX_\inn$ intersected with $\ell$-element to the remainder does not exceed 
$Ch^{-1}\ell^2 \times (\mu \ell)^{-1}\asymp C\mu^{-1}h^{-1}\ell$ and summation over all inner zone returns $O(\mu^{-1}h^{-1})$. 

Meanwhile contribution of $\ell$-element intersected with the boundary  zone $\cX_\bound$ and transitional zone $\cX_\trans$ does not exceed $C\mu^{-1}h^{-1}\ell \times \ell^{-1}$ as $T^*=\epsilon \ell$ here; then summation with respect to these zones returns $O(\mu^{-1}h^{-1}|\log h|)$. 

One can get rid off logarithm noting that one can upgrade $T^*$ to 
$T^*=\min ( v^{-1} \ell,  1)$ as we are trying to move away from the corner; recall that in the transitional zone $\cX_\trans$ speed in the opposite direction may be at most $v\Def C(\mu h)^{\frac{1}{3}}h^{-\delta}+C\mu^{-1}$ due to analysis of section~\ref{book_new-sect-5-3}\cite{futurebook}  and $\mu \ell \ge 1$. Further, in the transitional zone we have an extra factor  $C(\mu h)^{\frac{2}{3}}h^{-\delta}$ due to analysis of section~\ref{book_new-sect-5-2} \cite{futurebook}.

We are left with the \emph{corner zone\/}\index{corner zone} 
$\cX_{\mathsf{corner}}\Def \{x:\ell(x) \le \bar{\ell}\}$ and in virtue of simple rescaling $x\mapsto \mu x$, $\mu \mapsto 1$ and $h\mapsto \mu h\le 1$ we can estimate its contribution to the remainder by 
\begin{equation}
C\mu h^{-1}\bar{\ell}^2= C\mu^{-1}h^{-1}\bigl(1+ \mu^4 h^{2-\delta}\bigr);
\label{15-5-5}
\end{equation}
So, the total remainder estimate is also given by (\ref{15-5-5}) which is $O(\mu^{-1}h^{-1})$ for
\begin{equation}
\mu \le h^{\delta-\frac{1}{2}}.
\label{15-5-6}
\end{equation}

\subsection{Strong magnetic field case}
\label{sect-15-5-3-2} 

To improve  estimate (\ref{15-5-5}) (our ultimate goal is  $O(\mu^{-1}h^{-1})$ as usual) as (\ref{15-5-6}) breaks one needs to invoke arguments of sections~\ref{sect-15-2} and \ref{sect-15-3} and take into account dynamics near the corner (see Figure~\ref{fig-traj4}) in the same way we took into account dynamics near critical point of $W|_{\partial X}$ (see Figure~\ref{fig-traj3}) before.

\paragraph{Tauberian remainder estimate.}
\label{sect-15-5-3-2-1}

As we mentioned in $\cX_\inn$ one can take $T^*\asymp \mu \ell$ and $T_*=h^{1-\delta}\ell^{-1}$. Therefore contribution of $\ell$-element intersected with $\cX_\inn$ to the Tauberian remainder does not exceed
\begin{equation}
Ch^{-1}\ell^2\bigl( 1+ \mu T_*\bigr) T^{*\,-1}\asymp 
C\mu^{-1}h^{-1}\bigl(\ell+ \mu h\bigr)
\label{15-5-7}
\end{equation}
where due to arguments of section~\ref{sect-15-2} here and below in the estimates we can take $\delta=0$ in the definition of $T_*$. Then summation over $\cX_\inn\cap \{x, \ell(x)\ge \bar{\ell}\}$ results in $O\bigl(\mu^{-1}h^{-1} + |\log h|\bigr)$. Moreover, in cases (b),(c), (e)--(h) one can take $T^*\asymp \mu$ and the result will be $O(\mu^{-1}h^{-1})$.

The weak magnetic field arguments remain valid in $\cX_\bound$ and its contribution to the Tauberian remainder is $O(\mu^{-1}h^{-1})$. 

Meanwhile exactly as in section~\ref{sect-15-2} contribution of $\ell$-element intersected with $\cX_\trans$ to the Tauberian remainder does not exceed
\begin{equation*}
C\mu^{-1}(\mu h)^{\frac{2}{3}}h^{-1-\delta}\ell \times 
\bigl(1+ \frac{\mu h}{\ell}\bigr)\times \bigl( 1+ \frac{v}{\ell}\bigr)
\end{equation*}
as $T^*\asymp \min (1,v^{-1}\ell)$ where $v$ is the propagation speed in the direction opposite to the hops and $(\mu h)^{\frac{2}{3}}h^{-\delta}$ is the width of $\cX_\trans$; then the total contribution of $\cX_\trans\cap\{x, \ell(x) \ge\bar{\ell}\}$ does not exceed 
\begin{equation*}
C\mu^{-1}(\mu h)^{\frac{2}{3}}h^{-1-\delta}
\bigl(1+\mu h + v +\mu hv \bar{\ell}^{-1}\bigr);
\end{equation*}
meanwhile the contribution of $\cX_\trans\cap\{x, C\hbar\le \ell\le\bar{\ell}\}$ does not exceed $C(\mu h)^{\frac{2}{3}}h^{-1-\delta}\bar{\ell}$ as we take $T^*\asymp \mu^{-1}$ here where here and below
\begin{equation}
\hbar\Def \max(\mu^{-1}, h^{1-\delta}).
\label{15-5-8}
\end{equation}
Let us find $\bar{\ell}$ from $T^*=T_*$ but keep it greater than $\hbar$ i.e. 
\begin{equation}
\bar{\ell}= C\max\bigl( \hbar, (vh)^{\frac{1}{2}}\bigr);
\label{15-5-9}
\end{equation}
so the total contribution of $\cX_\trans\cap \{x,\ell(x)\ge \hbar\}$ does not exceed 
\begin{equation*}
C\mu^{-1}h^{-1}+Ch^{-\delta}+ 
Ch^{-1}(\mu h)^{\frac{2}{3}} (vh)^{\frac{1}{2}}h^{-\delta}.
\end{equation*}
Plugging $v=(\mu h)^{\frac{1}{3}}h^{-\delta}$ we arrive to 
\begin{equation}
C\mu^{-1}h^{-1}+Ch^{-\delta}+ Ch^{-\frac{1}{2}-\delta}(\mu h)^{\frac{5}{6}}
\label{15-5-10}
\end{equation}
where the last term is not greater than the first two as $\mu\le h^{-\frac{8}{11}}$.

On the other hand, under Dirichlet boundary condition\footnote{\label{foot-15-17} And Robin boundary condition with $\alpha\ge \epsilon$.} we can take $v= \mu^{-1}$ but then $\mu \ell\ge 1$ and we can take $T^*=1$ and we arrive to $O(\mu^{-1}h^{-1}+h^{-\delta})$. 

\begin{remark}\label{rem-15-5-3}
As it takes time $(\mu h)^{\frac{2}{3}}h^{-\delta}$ to punch through $\cX_\trans$ and shift with respect to $x_2$ will be 
$v(\mu h)^{\frac{2}{3}}h^{-\delta}$ we can improve $T^*$ as 
$v(\mu h)^{\frac{2}{3}}h^{-\delta} \le \ell$ which is not needed in the case of the Dirichlet\footref{foot-15-17} problem and does not help in the case of the Neumann problem.
\end{remark}

We are left with $\cX_{\mathsf{corner}}$ but its contribution does not exceed $C\mu h^{-1} \bar{\ell}^2= Ch^{-\delta}$.

Therefore remainder estimate is given by (\ref{15-5-10}); if on the both sides of the corner Dirichlet\footref{foot-15-17} \footnote{\label{foot-15-18} We can obviously mix conditions.} is given remainder estimate is $O(\mu^{-1}h^{-1}+h^{-\delta})$.

\paragraph{Magnetic Weyl remainder estimate.}
\label{sect-15-5-3-2-2}

Transition from Tauberian to magnetic Weyl formula is obvious in $\cX_\inn$. To run successive approximations one needs to assume that 
\begin{equation}
(\mu^{-1}+v T_*)T_*\le h^{1+\delta}.
\label{15-5-11}
\end{equation}
This is definitely the case in $\cX_\bound$ and we need to consider $\cX_\trans$ only. Plugging $T_*=h\ell^{-1}$ we arrive to $\ell\ge\bar{\ell}_1$ with 
$\bar{\ell}_1$ defined by 
\begin{equation}
\bar{\ell}_1= C\max\bigl( \mu^{-1}h^{-\delta}, (vh)^{\frac{1}{2}}\bigr)
\label{15-5-12}
\end{equation}
which brings no extra terms in comparison with the Tauberian remainder. 

So in contrast to section~\ref{sect-15-3} transition from Tauberian to magnetic Weyl formula does not increase an error.

\subsection{Very strong magnetic field case }
\label{sect-15-5-3-3}
We need also to consider case of very strong magnetic field case $h^{\delta-1}\le \mu \le \epsilon h^{-1}$  under Dirichlet and improve remainder estimate $O(h^{-\delta})$ we got here. In this case analysis follows basically from the same arguments as in subsection~\ref{sect-15-3-1}. Then in formula (\ref{15-5-7}) one should replace $(\mu h)^{\frac{2}{3}}h^{-\delta}$ by $1$ resulting in the remainder estimate $O(\mu^{-1}h^{-1}+|\log \mu|^K)$. In this approach actually corner zone becomes 
$\cX_{\mathsf{corner}}=\{x, \ell(x)\le \bar{\ell}\}$ with 
$\bar{\ell}= C\mu^{-1}+ Ch|\log h|^{K/2})$ or even 
$\bar{\ell}= C\mu^{-1}$.

\begin{problem}\label{problem-15-5-4}
(i) Investigate very strong magnetic field case 
under Dirichlet\footref{foot-15-17}  and recover remainder estimate $O(\mu^{-1}h^{-1}+|\log \mu|^K)$;

\medskip\noindent
(ii) Improve it under Dirichlet boundary condition to $O(\mu^{-1}h^{-1})$; is it possible?

\medskip\noindent
(iii) Could be it done under Robin boundary condition?
\end{problem}

\subsection{Superstrong magnetic field case}
\label{sect-15-5-3-4}

The similar arguments should work in the superstrong magnetic field case 
$\mu \ge \epsilon h^{-1}$ albeit we need to take $\bar{\ell}=\mu^{-\frac{1}{2}+\delta}h^{\frac{1}{2}}$, $\bar{\ell}=\mu^{-\frac{1}{2}}h^{\frac{1}{2}}|\log \mu|^{K/2}$ and
$\bar{\ell}=C\mu^{-\frac{1}{2}}h^{\frac{1}{2}}$ to recover remainder estimates $O(\mu^\delta)$, $O(|\log \mu|^K)$ and $O(1)$ respectively and we arrive to

\begin{problem}\label{problem-15-5-5}
(i) Investigate very strong magnetic field case 
under Dirichlet\footref{foot-15-17}  and recover remainder estimate 
$O(\mu^\delta)$;

\medskip\noindent
(ii) Improve it under Dirichlet boundary condition to 
$O(\mu^{-1}h^{-1}+|\log \mu|^K)$ and even to $O(\mu^{-1}h^{-1})$; is it possible?

\medskip\noindent
(iii) Could be it done under Robin boundary condition?
\end{problem}

\section{Boundary meets field degeneration of $F$}
\label{sect-15-5-4}

\subsection{Set-up}
\label{sect-15-5-4-1}

Even more challenging (and in my opinion worth both publication and degree) seems to be a problem

\begin{problem}\label{problem-15-5-6}
Get rid off condition $|F| \ge\epsilon$.
\end{problem}

Our goal (not necessary achievable in all cases) is  to derive asymptotics  as sharp as if there was no boundary. We will assume that 
$F\asymp \gamma(x)^{\nu -1}$ where $\gamma\in \sC^\infty$, $\gamma(x)=0\implies |\nabla \gamma (x)|\asymp 1$, $x\in\partial X, \gamma (x)=0\implies |\nabla_{\partial X} \gamma (x)|\asymp 1$, $V<0$.

Then we will have 
$X_\reg =\{x: |\gamma (x)|\ge \bar{\gamma}\Def C\mu^{-1/\nu}\}$ in the case of $\mu \le \epsilon h^{-\nu}$ and degeneration zone 
$X_\degen= \{|\gamma(x)|\le \bar{\gamma}\}$ and then we will have also four subzones:
\begin{gather}
X_{\reg,\inn}=\{x: \gamma(x)\ge \bar{\gamma}, \ 
\ell(x)\Def \dist (x,\partial X)\ge C\mu^{-1}\gamma(x)^{1-\nu}\},\label{15-5-13}\\[2pt]
X_{\reg,\bound}= \{x: \gamma(x)\ge \bar{\gamma}, \ 
\ell(x) \le C\mu^{-1}\gamma(x)^{1-\nu}\},\label{15-5-14}\\[2pt]
X_{\degen, \inn}= \{x: \gamma(x)\le \bar{\gamma}, \ \ell(x) \ge 
C\mu^{-1}\bar{\gamma}^{1-\nu}= C\bar{\gamma}\},\label{15-5-15}\\[2pt]
X_{\mathsf{corner}}= \{x: \gamma(x)\le \bar{\gamma}, \ \ell(x) \le 
 C\bar{\gamma}\}\label{15-5-16}
\end{gather}
as on Figure~\ref{fig-deg}(a) and each zone should be analyzed separately. 

\begin{remark}\label{rem-15-5-7}
(i) Actually we need slightly increase $\bar{\gamma}$ as $\mu $ is close to $h^{-\nu}$ and we need  increase ``slightly'' the threshold for $\ell(x)$ as $\gamma(x)$ is close to $(\mu h)^{-1/(\nu-1)}$ for $\mu \ge h^{\delta-1}$.

\medskip\noindent
(ii) When we consider modified Schr\"odinger operators with $\mu h\ge 1$ so that domain  $\{x,\ \gamma(x)\ge C(\mu h)^{-1/(\nu-1)}\}$ is no more forbidden, according to section~\ref{sect-15-4} we need to change threshold for $\ell(x)$ to ``slightly more'' than $\mu^{-1/2}h^{1/2} \gamma(x)^{(1-\nu)/2}$.

\medskip\noindent
(iii)  When we consider modified Schr\"odinger operators with $\mu h^\nu \ge 1$
according to section~\ref{book_new-sect-14-8} \cite{futurebook} we need to change threshold for $\bar{\gamma}$ to ``slightly more'' than $\mu^{-1/2\nu}h^{1/2}$. 
\end{remark}

Without any loss of the generality one can assume that $X=\{x_1>0\}$ and $\Sigma=\{x_2=0\}$ and $F_{12}<0$ as $x_2>0$. Then hops here are to the left ($x_2$ decays) and drift down ($x_1$ decays) because 
$\partial_{x_2} (-V)/F_{12}>0$. 

\begin{figure}[h!]
\centering
\subfloat[ ]{\includegraphics[width=0.44\linewidth]{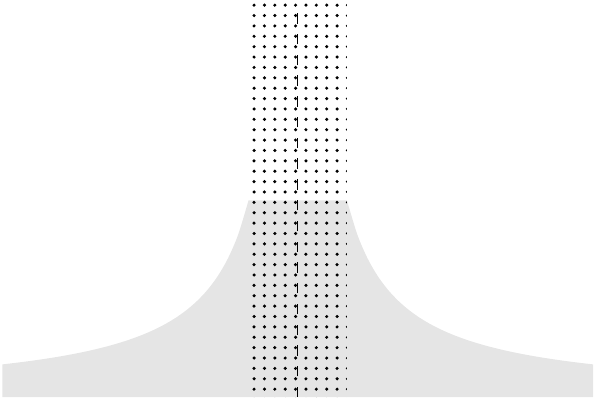}}\quad 
\subfloat[$\nu$ is even]
{\includegraphics[width=0.15\linewidth]{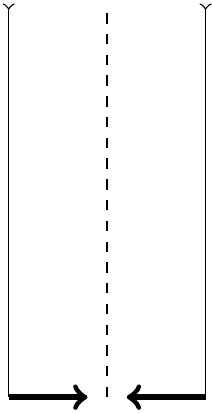}}\quad
\subfloat[$\nu$ is odd]
{\includegraphics[width=0.15\linewidth]{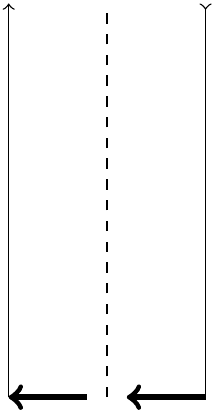}}
\caption{\label{fig-deg} Zones (boundary zone is gray, degeneration zone is dotted) and dynamics near degeneration line of $F$ (dashed). As $V<0$ we consider $F\asymp \gamma^{\nu-1}$ with smooth $\gamma$ such that $|\nabla \gamma|\asymp 1$.  Consistent with Figure~\ref{fig-traj}.}
\end{figure}

Consider $x_2<0$; note that $F_{12}0$ changes sign and 
$\partial_{x_2} (-V)/F_{12}$ does not if $\nu$ is even and then hops are to the right and drift is down there (Figure~\ref{fig-deg}(b)). On the other hand, $F_{12}0$ does not change sign and $\partial_{x_2} (-V)/F_{12}$ changes if $\nu$ is odd and then hops are to the left and drift is up there (Figure~\ref{fig-deg}(c)). 

\subsection{Weak magnetic field case}
\label{sect-15-5-4-2}

Using methods of the previous chapter and considering time direction in which drift up ne can prove easily that contribution of $X_{\reg,\inn}$ to the remainder is $O(\mu^{-1/\nu}h^{-1})$ at least as $\mu \le h^{\delta-\nu}$. 

Furthermore, using the methods of this chapter and and considering time direction in which hops away from $\Sigma$ (and drift up) one can prove easily that contribution of $X_{\reg,\bound}$ to the remainder is $O(\mu^{-1/\nu}h^{-1}$ at least as $\mu \le h^{\delta-\nu}$.

Consider zone $X_{\degen,\inn}$. Then analysis of the previous chapter gives us remainder estimate 
$O\bigl(\mu^{-1/\nu}h^{-1}+ (\mu^{1/\nu}h)^{1/2}h^{-1}\bigr)$. Recall that the last term is responsible for the correction which would come from near-periodic trajectories if $W=\const$. Now it is time to use the last non-degeneracy condition 
\begin{equation}
|\nabla_\Sigma W|\asymp 1.
\label{15-5-17}
\end{equation}
where $V^*$ and $W=-V^*|_\Sigma$ are introduced by formulae (\ref{book_new-14-2-28})  and (\ref{book_new-14-2-29}) \cite{futurebook}. Then using a canonical form of the previous chapter near $\Sigma$ (and thus ignoring presence of the boundary) and considering evolution of $\xi_2$ one can take $T_*= h^{1-\delta}\ell^{-1}$ with $\ell\asymp |x_1|$ (and this takes in account the presence of $\partial X$).

As $T_*\le \epsilon \mu^{-1/\nu}$ one can then ignore all periodicity, which proves that the contribution of zone $X_\degen \cap\{x_1\ge \bar{\ell}\}$ to the remainder is $O(\mu^{-1/\nu}h^{-1})$; meanwhile, the contribution of zone 
$X_{\degen,\inn} \cap\{x_1\le \bar{\ell}\}$ to the remainder is ${O\bigl(\mu^{-1/\nu}h^{-1}+ (\mu^{1/\nu}h)^{1/2}h^{-1} \bar{\ell}\bigr)}$ with
$\bar{\ell}= C\max( \mu^{1/\nu}h^{1-\delta}, \mu^{-1/\nu})$. By our standard technique we can actually make $\delta=0$ in the estimate and then we arrive to the remainder estimate 
$O\bigl(\mu^{-1/\nu}h^{-1}+ (\mu^{1/\nu}h)^{3/2}h^{-1} \bigr)$. In particular, remainder estimate is $O(\mu^{-1/\nu}h^{-1})$ as $\mu \le h^{-3\nu/5}$.

\subsection{Strong magnetic field case}
\label{sect-15-5-4-3}

Here comes the difficult part:

\begin{problem}\label{problem-15-5-8}
(i) In frames of condition (\ref{15-5-3}) prove that if 
\begin{equation}
|\xi_2-\eta^*_\hbar W^{1/2}|\le \rho \le \ell^{1/2},\qquad\text{with\ \ }\rho \ell \ge h
\label{15-5-18}
\end{equation}
then one can take 
$T_*\asymp  h \ell^{-1}$ and $T^*\asymp 1$ and then contribution of such element to the Tauberian remainder does not exceed
\begin{equation}
C\mu^{-1/\nu} \rho \ell h^{-1}\bigl(1+ \mu^{1/\nu} T_*\bigr) T^{*,-1} \asymp
C\mu^{-1/\nu}h^{-1} \rho \ell \bigl(1+ \mu^{1/\nu} h\ell^{-1}\bigr)
\label{15-5-19}
\end{equation}
and therefore the total contribution to the Tauberian remainder of elements of this type with  $\ell \ge \max(h^{2/3},\mu^{-1/\nu})$ is $O(\mu^{-1/\nu}h^{-1})$.;

\medskip\noindent
(ii) Prove that if 
\begin{equation}
|\xi_2-\eta^*_\hbar W^{1/2}|\asymp \rho \ge \ell^{1/2},\qquad\text{with\ \ }\rho \ell \ge h
\label{15-5-20}
\end{equation}
then one can take $T_*\asymp  h \rho^{-2}$ and $T^*\asymp \rho$ 
and then contribution of such element to the Tauberian remainder does not exceed
\begin{equation}
C\mu^{-1/\nu} \rho \ell h^{-1}\bigl(1+ \mu^{1/\nu} T_*\bigr) T^{*,-1} \asymp
C\mu^{-1/\nu}h^{-1} \rho \ell \bigl(1+ \mu^{1/\nu} h\rho^{-2}\bigr)\rho^{-1}
\label{15-5-21}
\end{equation}
and therefore the total contribution to the Tauberian remainder of elements of this type with  $\rho \ge \max(h^{1/3},\mu^{-1/2\nu})$ is $O(\mu^{-1/\nu}h^{-1}+|\log \mu|)$;

\medskip\noindent
(iii) Provide the same  estimate for contribution of the remaining small zone 
$\{\rho \le \max(h^{1/3},\mu^{-1/2\nu}), \ C\mu^{-1/\nu} \le \ell \le h^{2/3},  \rho \ell \le h \}$; combining with estimate for $X_{\mathsf{corner}}$ we arrive to the same estimate $O(\mu^{-1/\nu}h^{-1}+h^{-\delta})$;

\medskip\noindent
(iv) Improve this estimate to $O(\mu^{-1/\nu}h^{-1}+1)$;
\end{problem}

Also looks challenging the next step:

\begin{problem}\label{problem-15-5-9}
Pass from asymptotics with the Tauberian principal part to a more explicit ones. 
\end{problem}

There should be usual Magnetic Weyl term, a term associated with the boundary and a term associated with degeneration albeit we do not expect any mixed boundary-degeneration term.

\subsection{Superstrong and hyperstrong magnetic field cases}
\label{sect-15-5-4-4}

In the same way one needs to consider modified Schr\"odinger operator (with $\mathfrak{z}\mu h F^\#$ subtracted where $F^\#$  is ``$F$ but without absolute value'') and taking into account remark~\ref{rem-15-5-7} follow the same path:

\begin{problem} 
Derive spectral asymptotics for the modified Schr\"odinger operator:

\medskip\noindent
(i) Consider the case of superstrong magnetic field: $\epsilon h^{-1}\le \mu \le \epsilon h^{-\nu}$;

\medskip\noindent
(ii) Consider the case of hyperstrong magnetic field: $ \mu \ge \epsilon h^{-\nu}$.
\end{problem}

\begin{subappendices}
\chapter{Appendices: eigenvalues of \texorpdfstring{$L(\eta)$}{L(\texteta)}}
\label{sect-15-A}

\section{Basic properties}
\label{sect-15-A-1}

Recall that $\lambda_{\D,n}(\eta)$ and $\lambda_{\N,n}(\eta)$ $n=0,1,\dots$ are eigenvalues of\footnote{\label{foot-15-19} Obviously, transformation 
$x\mapsto -x$ reduces  (\ref{15-1-26}) to this one.}
\begin{equation}
L(\eta)= -\partial_x ^2 +(x+\eta)^2
\label{15-A-1}
\end{equation}
on $\bR^-$ with Dirichlet and Neumann boundary condition respectively at $0$, 
$u_n(x)=\upsilon_n(x+\eta ,\eta)$ are (real-valued) eigenfunctions such that
$\|\upsilon_n\|=1$.

\begin{proposition}\label{prop-15-A-1}
$\lambda_{\D, n}(\eta)$ and $\lambda_{\N, n}(\eta)$, $n=0,1,2,\dots$ are real analytic functions with the following properties:

\medskip\noindent
(i)  $\lambda_{\D, n}(\eta)$ are monotone decreasing for $\eta\in \bR$; 
$\lambda_{\D, n}(\eta)\nearrow +\infty$ as $\eta\to -\infty$;
$\lambda_{\D, n}(\eta)\searrow (2n+1)$ as $\eta\to +\infty$; 
$\lambda_{\D, n}(0)=(4n+3)$;

\medskip\noindent
(ii) $\lambda_{\N, n}(\eta)$ are monotone decreasing for $\eta\in \bR^-$; $\lambda_{\N, n}(\eta)\nearrow +\infty$ as $\eta\to -\infty$; 
$\lambda_{\N, n}(\eta)< (2n+1)$ as $\eta\ge (2n+1)^{\frac{1}{2}}$;
$\lambda_{\N, n}(0)=(4n+1)$;

\medskip\noindent
(iii) 
$\lambda_{\N, n}(\eta) < \lambda_{\D, n}(\eta) <\lambda_{\N, (n+1)}(\eta)$;
$\lambda_{\D, n}(\eta)> (2n+1)$, $\lambda_{\N, n}(\eta)> (2n-1)_+$.
\end{proposition}

\begin{proof}
Note that $\lambda_{\D,n}(\eta)$ and $\lambda_{\N,n}(\eta)$ are obtained by variational procedure from quadratic forms
\begin{equation}
Q^-(u)=\int_{x<0} \Bigl(|\partial_x u|^2 + (x +\eta)^2|u|^2\Bigr)\,dx, \quad
Q^-_0(u)=\int_{x<0} |u|^2\,dx
\label{15-A-2}
\end{equation}
respectively with Dirichlet or no boundary conditions at $x=0$. Alternatively we can consider 
\begin{equation}
Q^-(u)=\int_{x<\eta} \Bigl(|\partial_x u|^2 + x^2|u|^2\Bigr)\,dx, \quad
Q^-_0(u)=\int_{x<\eta} |u|^2\,dx
\label{15-A-3}
\end{equation}
respectively with Dirichlet or no boundary conditions at $x=\eta$ and we will shift notations between these two.

We know eigenvalues $\bar{\lambda}_n=(2n+1)$ and eigenfunctions of this operator on $\bR$ (i.e. as $\eta=-\infty$). Further, we know that eigenvalues of $L_\D(0)$ and $L_\N(0)$ are $(2m+1)$ respectively with odd and even $m$ we arrive to all assertions in (i), and to all assertions in (ii) except the third one.

Consider $\bar{\upsilon}_k (x)$ which are eigenfunctions of harmonic oscillator on $\bR$, $k=0,\dots,n$. As $u$ belongs to their span 
\begin{multline*}
Q^- (u) = Q (u)-Q^+(u)\le (2n+1) Q_0(u) - Q^+(u)=\\
(2n+1) Q^-_0(u) +(2n+1) Q^+_0(u) -Q^+(u)
\end{multline*}
where $Q$, $Q_0$ denote the same forms on $\bR$ and $Q^+$, $Q^+_0$ on $\{x>\eta\}$.

Let $\eta\ge (2n+1)^{\frac{1}{2}}$. Obviously $(2n+1) Q^+_0(u) < Q^+(u)$ and therefore $Q^-(u)< (2n+1) Q^-_0(u)$ for $u$ belonging to some $(n+1)$-dimensional space; then $\lambda_{\N,n}(\eta)<(2n+1)$.

In (iii) one needs only to prove that 
$\lambda_{\D,n} (\eta)<  \lambda_{\N, (n+1)}(\eta)$. However in $(n+2)$ dimensional span of $\upsilon_{\N, 0}(x),\dots, \upsilon_{\N, (n+1)}(x)$, functions vanishing at $x=\eta$ form $(n+1)$ dimensional subspace on which 
$Q^-(u) < \lambda_{\N, (n+1)}Q^-_0(u)$ (strict inequality is due to the fact that only on $v=\upsilon_{\N, (n+1)}(x)$ equality can be reached and this function does not belong to the subspace in question (as then $v(\eta)=\partial v(\eta)=0$ which is impossible).
\end{proof}

So far we have not proved that $\lambda_{\N, n}(\eta)\to (2n+1)$ as $\eta\to +\infty$; it will be done later.

\section{More properties}
\label{sect-15-A-2}

We follow M.~Dauge--B.~Helffer~\cite{dauge-helffer-1} here.

Consider equation on $x<0$
\begin{equation}
\bigl(L(\eta)-\lambda_n(\eta)\bigr)u_n=0.
\label{15-A-4}
\end{equation}
Differentiating by $\eta$ and multiplying by $u_n$ we arrive to
\begin{multline}
\partial_\eta \lambda_n (\eta)= ((\partial_\eta L(\eta)) u_n,u_n)= 2((x+\eta)u_n,u_n)= \\[3pt]
([\partial_x, (x+\eta)^2]u_n,u_n)=
2((x+\eta)^2 u_n,\partial_x u_n)+ \eta^2  |u_n(0)|^2= \\[3pt]
2(L (\eta)u_n,\partial_x u_n)+ 
(\partial_x^2 u_n,\partial_x u_n)+ \eta^2  |u_n(0)|^2 =\\[3pt]
\bigl(\eta^2-\lambda_{n}(\eta)\bigr) |u_n (0)|^2 -|\partial_x u_n (0)|^2. 
\label{15-A-5}
\end{multline}
In particular,
\begin{phantomequation}\label{15-A-6}\end{phantomequation} 
\begin{gather}
\partial_\eta \lambda_{\D,n} (\eta)= -|\partial_x \upsilon_{\D,n}(\eta)|^2 \tag*{$\textup{(\ref*{15-A-6})}_\D$}\label{15-A-6-D}\\
\shortintertext{and}
\partial_\eta \lambda_{\N,n} (\eta)= \bigl(\eta^2-\lambda _{\N,n}(\eta)\bigr) |\upsilon_{\N,n}(\eta)|^2. \tag*{$\textup{(\ref*{15-A-6})}_\N$}\label{15-A-6-N}
\end{gather}
We refer to them as \emph{Dauge-Helffer formulae\/}~\cite{dauge-helffer-1}.
\index{Dauge-Helffer formulae}

We conclude immediately that 

\begin{proposition}\label{prop-15-A-2}
(i) $\partial_\eta \lambda_{\D,n} (\eta)<0$;

\medskip\noindent
(ii)  $\partial_\eta \lambda_{\N,n}(\eta) \gtreqqless 0$ if and only if
\begin{equation}
\lambda_{\N,n} (\eta)\lesseqqgtr\eta^2.
\label{15-A-7}
\end{equation}
\end{proposition}

Differentiating \ref{15-A-6-N} we arrive to
\begin{multline}
\partial^2_\eta \lambda_{\N,n} (\eta)= \\\bigl(2\eta-\partial_\eta\lambda_{\N,n}(\eta)\bigr)|\upsilon_{\N,n}(\eta)|^2
+2\bigl(\eta^2-\lambda_{\N,n} (\eta)\bigr) \partial_\eta \upsilon_{\N,n}(\eta)\cdot \upsilon_{\N,n}(\eta).
\label{15-A-8}
\end{multline}
Therefore as $\partial_\eta \lambda_{\N,n}(\eta)=0$ in virtue of proposition~\ref{prop-15-A-1} 
\begin{equation}
\partial^2_\eta \lambda_{\N,n}(\eta) = 2\eta |\upsilon_{\N,n}(\eta)|^2
\label{15-A-9}
\end{equation}
and therefore it is non-degenerate minimum. So, 

\begin{proposition}\label{prop-15-A-3}
(i) $\lambda_{\N,n}(\eta)$ has a single stationary point\footnote{\label{foot-15-20}
And it must have one due to proposition~\ref{prop-15-A-1}.} $\eta_n$, it is non-degenerate minimum, and in this point \textup{(\ref{15-A-7})}, \textup{(\ref{15-A-9})} hold.

\medskip\noindent
(ii) In particular, $(\lambda_{\N,n}(\eta)-\eta^2)$ has the same sign as $(\eta_n-\eta)$.
\end{proposition}

 Therefore
$\lambda_{\N, n+1} (\eta_n)-\eta_n^2> \lambda_{\N, n} (\eta_n)-\eta_n^2=0$ implies that 
\begin{equation}
\eta_{n+1}>\eta_n.
\label{15-A-10}
\end{equation}

\section{Estimates of \texorpdfstring{$\lambda_{*,n}(\eta)$ as $\eta\to+\infty$}{\textlambda(\texteta)  as \texteta\textrightarrow+\textinfty}}
\label{sect-15-A-3}

So we must investigate properties of $\lambda_n(\eta)$ as 
$\eta\to +\infty$. We know that then $\lambda_{\D,n}(\eta)\ge (2n+1)\ge \lambda_{\N,n}(\eta)$ and $\lambda_{\D,n}(\eta)$  tends to $(2n+1)$.

Consider $\bar{\upsilon}_j(x)$. Note that 
\begin{equation}
\int_\eta^{+\infty} \bar{\upsilon}_j(x)\bar{\upsilon}_k(x)\,dx =O\bigl(\eta^{j+k-1}e^{-\eta^2}\bigr)
\label{15-A-11}
\end{equation}
(non-uniformly with respect to $j,k$) and as $j=k$ this is an exact magnitude.

Therefore 
\begin{equation}
(\bar{\upsilon}_j,\bar{\upsilon}_k)= \updelta_{jk} + O\bigl(\eta^{j+k-1}e^{-\eta^2}\bigr)
\label{15-A-12}
\end{equation}
as our space is $\sL^2((-\infty,\eta))$.

Consider bilinear form 
\begin{gather*}
Q^-(u,v)=\int_{-\infty}^\eta \bigl(u'v'+x^2 uv\bigr)\,dx\\
\intertext{(where $'$ means $\partial_x$). One can rewrite it as} 
Q^-(u,v)=\int_{-\infty}^\eta \bigl(-u''+x^2 u\bigr)v\,dx + u'(\eta)v(\eta).
\end{gather*}
Then
\begin{equation*}
Q^-(\bar{\upsilon}_j,\bar{\upsilon}_k)= (2j+1) (\bar{\upsilon}_j,\bar{\upsilon}_k) + \bar{\upsilon}'_j(\eta)\bar{\upsilon}_k(\eta).
\end{equation*}
Thus if we consider $u=\sum_{k\le n}\alpha_k\bar{\upsilon}_k$ with 
$\alpha_k\in \bR$ such that $\sum_{k\le n}\alpha_k^2\asymp 1$ then 
\begin{equation*}
Q^-(u)-(2n+1)Q_0^-(u)= \sum_{k\le n} 2(k-n)\alpha_k^2 +
O\bigl(e^{-\eta^2} \eta^{n+k+1}\bigr)
\end{equation*}
and this expression is negative provided $\alpha_k^2\ge e^{-\eta^2}\eta^{n+k+1}$ for some $k<n$;  otherwise  we can assume that $\alpha_n=1$ and 
\begin{multline*}
Q^-(u) -(2n+1)Q_0^-(u) \equiv 
\sum_{k\le n} \bigl(Q^-(\bar{\upsilon}_k)-(2n+1)Q_0^-(\bar{\upsilon}_k)\bigr)\alpha_k^2 \le \\ \bigl(Q^-(\bar{\upsilon}_n)-(2n+1)Q_0^-(\bar{\upsilon}_n)\bigr)
\equiv  \bar{\upsilon}'_n(\eta)\bar{\upsilon}_n(\eta)
\mod O\bigl(e^{-\frac{5}{4}\eta^2}\bigr)
\end{multline*}
and we know that $\bar{\upsilon}'_n(\eta)\bar{\upsilon}_n(\eta)$ is negative and of magnitude $\eta^{2n+1}e^{-\eta^2}$ as $\eta\ge c_n$ and therefore due to variational principle
\begin{equation*}
\lambda_{\N,n} (\eta)\le (2n+1) - \epsilon \eta^{2n+1}e^{-\eta^2}\qquad
\text{as\ \ } \eta \ge c_n. 
\end{equation*}

Meanwhile let us replace $\bar{\upsilon}_k(x)$ by 
$\tilde{\upsilon}_k(x)=\bar{\upsilon}_n(x)+\phi_k$, $\phi_k=-\bar{\upsilon}_k(\eta)e^{\eta x -\eta^2}$ satisfying Dirichlet conditions as $x=\eta$. All the above conclusions except calculations of $Q(\bar{\upsilon}_n)$ remain true for these functions as well and for them
\begin{multline*}
Q^-(\tilde{\upsilon}_n)- (2n+1)Q_0^-(\tilde{\upsilon}_n)=\\[2pt]
\bar{\upsilon}'_n (\eta)\bar{\upsilon}_n(\eta)+ 2\bar{\upsilon}'_n(\eta)\phi_n(\eta) +Q^-(\phi_n)-(2n+1)Q^-_0(\phi_n) =\\[2pt]
-\upsilon'_n(\eta)\upsilon_n(\eta) +  \upsilon^2 _n(\eta) 
\bigl(\eta +O(\eta ^{-1})\bigr)\asymp \eta^{2n+1}e^{-\eta^2}
\end{multline*}
and therefore
\begin{equation*}
\lambda_{\D,n} (\eta)\le (2n+1) + C_n \eta^{2n+1}e^{-\eta^2}\qquad
\text{as\ \ } \eta \ge c_n. 
\end{equation*}

On the other hand, consider $\sC\ni u$ supported in $(-\infty,\eta]$ and orthogonal to $\bar{\upsilon}_0,\dots,\bar{\upsilon}_{n-1}$. Then (modulo constant factor) $u=  \bar{\upsilon}_n(x) + v$ where $v$ is orthogonal to $\bar{\upsilon}_n$ in $\bR$ as well.  

Note that
\begin{equation*}
Q^- (u)=  Q(u) = Q(\bar{\upsilon}_n)+Q(v) = (2n+1) + Q(v),
\end{equation*}
and $Q_0^-(u) = 1 +\|v\|_\bR ^2$, and 
\begin{multline*}
Q^-(u)-(2n+1)Q_0^-(u) = Q(v)-(2n+1)Q_0(v) \ge 2Q_0(v)\ge 2Q_0^+(v) =\\
2Q_0^+(\bar{\upsilon}_n)\ge \epsilon \eta^{2n+1}e^{-\eta^2}.
\end{multline*}
So, any $(n+1)$-dimensional subspace of $\sL^2((-\infty,\eta))$ contains an element with 
$Q^-(u)\ge \bigl((2n+1)+ \epsilon \eta^{2n+1}e^{-\eta^2}\bigr)Q_0^-(u)$
and we arrive to the left inequality of  
\begin{equation}
\epsilon_n \eta^{2n+1}e^{-\eta^2} \le \lambda_{\D,n} (\eta)-(2n+1) \le  C_n \eta^{2k+1}e^{-\eta^2}\qquad
\text{as\ \ } \eta \ge c_n
\label{15-A-13}
\end{equation}
as the right one has been already proven.

Finally, estimate $\lambda_{\N,n}$ from below. Consider $\upsilon_n(x)$ and extend it as $\upsilon_n(\eta)e^{\frac{1}{2}(\eta^2- x^2)}$ to $(\eta,\infty)$. Let $u$ be a resulting function. Note
\begin{equation*}
Q^+(u)-(2n+1)Q_0^+(u) =
|\upsilon_n(\eta)|^2\bigl(\eta - 2n \Phi (\eta)\bigr),
\end{equation*}
with $\Phi (\eta) =\int_\eta^\infty e^{\eta^2-x^2}\,dx$ 
and then as $n=0$
\begin{multline*}
\lambda_{\N,n}(\eta) -(2n+1)=Q^-(\upsilon_n)-(2n+1)Q_0^-(\upsilon_n) =\\[3pt]
Q(u)-(2n+1)Q_0(u) - Q^+(u)+(2n+1) Q_0^+(u)\ge \\[2pt]
-|\upsilon_n(\eta)|^2\bigl(\eta - 2n \Phi(\eta)\bigr) = 
-\frac{1}{2}\bigl(\eta - 2n \Phi (\eta)\bigr) (\eta^2-\lambda_{\N,n})^{-1}\partial_\eta \lambda_{\N,n}(\eta)
\end{multline*}
due to \ref{15-A-6-N}.Therefore
\begin{equation*}
\bigl(2n+1-\lambda_{\N,n}(\eta)\bigr)^{-1}  \partial_\eta \lambda_{\N,n} (\eta)\ge 
2 \bigl(\eta^2 -\lambda_{\N,n}(\eta)\bigr)\bigl(\eta -\Phi(\eta)\bigr)^{-1}.
\end{equation*}
Note that $\Phi (\eta) = \frac{1}{2}\eta^{-1}+O(\eta^{-3})$ and $\lambda_n<2n+1$. Then the right-hand side of the last inequality is greater than 
$2(\eta -(n+1)\eta^{-1})-O(\eta^{-2})$ and integrating we conclude that 
\begin{equation*}
\log \bigl(2n+1-\lambda_{\N,n}(\eta)\bigr)^{-1}\ge  \eta^2 -2(n+1)\log \eta -  \log C_n
\end{equation*}
and 
\begin{equation*}
2n+1-\lambda_{\N,n}(\eta)\le \sigma_n(\eta)\Def C_ne^{-\eta^2}\eta^{2n+2}
\end{equation*}
which is just one factor $\eta$ greater than we got as an lower estimate for $(1-\lambda_n)$. 

Finally, as  $\lambda_{\N,n+1}(\eta)> \lambda_{\N,n+1}(\eta_{n+1})> \lambda_{\D,n}(\eta_{n+1})>(2n+1)$ we conclude that $\lambda_{\N,n+1}(\eta)$ and $\lambda_{\N,n}(\eta)$ are disjoint as $\eta>C_n$ and therefore as 
$Q^-(\bar{\upsilon}_n)\le Q^-(\upsilon_n)+C_n\sigma_n$ we conclude that
$Q^-(\upsilon_n -\bar{\upsilon}_n)\le C_n\sigma_n(\eta)$:
\begin{equation}
\|\partial_x(\upsilon_n -\bar{\upsilon}_n)\|+
\|x (\upsilon_n -\bar{\upsilon}_n)\|\le C_n\sigma_n(\eta)^{\frac{1}{2}}
\label{15-A-14}
\end{equation}
where norms are calculated in $\sL^2((-\infty,\eta))$ and then 
\begin{equation}
|(\upsilon_n -\bar{\upsilon}_n)|\le 
C_n\sigma_n(\eta)^{\frac{1}{2}};
\label{15-A-15}
\end{equation}
In particular $|\upsilon_n(\eta)|\le C_n\sigma_n(\eta)^{\frac{1}{2}}$.

So far it was proven only for $n=0$. However we apply induction: as it was proven for $k<n$ with fast decaying $\sigma_k(\eta)$. Then 
$(\upsilon_n,\bar{\upsilon}_k)=O\bigl(\sigma_{k}(\eta)^{\frac{1}{2}}\bigr)$ for $k<n$ and extending $\upsilon_n$ as before we arrive to $(u,\bar{\upsilon}_k)=O\bigl(\sigma_{n-1}(\eta)^{\frac{1}{2}}\bigr)$ for all $k<n$.

Then $Q(u)-(2n+1)Q_0(u)\ge -C \sigma_{n-1}$ and we arrive to
\begin{equation*}
\lambda_{\N,n} -(2n+1)\ge 
-\frac{1}{2}\bigl(\eta - 2n \Phi(\eta)\bigr) \bigl(\eta^2-\lambda_{\N,n}(\eta)\bigr)^{-1}\partial_\eta \lambda_{\N,n}(\eta) -
C\eta^{2n} e^{-\eta^2}
\end{equation*}
which implies the same conclusions.

We also conclude that (\ref{15-A-14}) and (\ref{15-A-15}) hold for $\bar{\upsilon}_n$ replaced by $\upsilon_{\D,n}$ and that 
$|\partial_x \upsilon_{\D,n}(\eta)|\le C_n\eta \sigma_n(\eta)^{\frac{1}{2}}$.

\section{Asymptotics of \texorpdfstring{$\lambda_{*,n}(\eta)$ as $\eta\to+\infty$}{\textlambda(\texteta)  as \texteta\textrightarrow+\textinfty}}
\label{sect-15-A-4}

Here will reproduce results of section~5 of C.~Bolley--B.~Helffer~\cite{bolley-helffer-1} and generalize them to arbitrary $j$ and to Dirichlet problem (definitely their methods could do the same).

In contrast to this paper we will not build quasi-modes but use Dauge-Helffer formulae. Let 
\begin{equation}
\varepsilon_{*,n}(\eta)= \pm \bigl(\lambda_{*,n}(\eta)-2n-1\bigr)
\label{15-A-16}
\end{equation}
with $*=\D,\N$ respectively.
Then from
\begin{equation*}
L(\eta)\upsilon_{*,n}(\eta)- \bigl(2n+1\pm \varepsilon_{*,n}(\eta)\bigr)\upsilon_{*,n}=0, 
\qquad
L(\eta)\bar{\upsilon}_n- (2n+1)\bar{\upsilon}_n=0
\end{equation*}
we conclude that
\begin{equation*}
\mp \varepsilon_{*,n}(\eta)(\upsilon_{*,n}, \bar{\upsilon}_n) =\upsilon'_n(\eta)\bar{\upsilon}_n(\eta)-\upsilon_n(\eta)\bar{\upsilon}'_n(\eta)
\end{equation*}
and as modulo $O\bigl(e^{-\frac{1}{2}\eta^2}\eta^K\bigr)$ 
$(\upsilon_{*,n}, \bar{\upsilon}_n)\equiv 1$ then 
\begin{equation*}
\mp \varepsilon_{*,n}(\eta)\equiv
\upsilon'_{*,n}(\eta)\bar{\upsilon}_n(\eta)-
\upsilon_{*,n}(\eta)\bar{\upsilon}'_n(\eta)
\end{equation*}
modulo $O\bigl(e^{-\frac{3}{2}\eta^2}\eta^K\bigr)$. Then modulo sign
\begin{gather*}
-\varepsilon_{\D,n}(\eta)\equiv \upsilon'_{\D,n} (\eta)\bar{\upsilon}_n(\eta)=
\bigl(-\partial_\eta \varepsilon_{\D,n}\bigr)^{\frac{1}{2}} \bar{\upsilon}_n(\eta),\\[3pt]
\varepsilon_{\N,n}(\eta)\equiv \upsilon_{\N,n} (\eta)\bar{\upsilon}'_n(\eta)=
\bigl(-(\eta^2-2n-1)^{-1} \partial_\eta \varepsilon _{\N,n}\bigr)^{\frac{1}{2}} \bar{\upsilon}'_n(\eta)
\end{gather*}
and modulo $O\bigl(e^{\frac{1}{2}\eta^2}\eta^K\bigr)$
\begin{gather*}
\partial_\eta \varepsilon _{\D,n}^{-1}(\eta)\equiv (\bar{\upsilon}_n(\eta))^{-2},
\\[3pt]
\partial_\eta \varepsilon _{\N,n}^{-1} (\eta)\equiv (\eta^2-2n-1)\bigl(\bar{\upsilon}'_n(\eta)\bigr)^{-2}
\end{gather*}
and 
\begin{phantomequation}\label{15-A-17}\end{phantomequation}
\begin{gather}
\varepsilon _{\D,n}^{-1} (\eta)
\equiv \int^\eta (\bar{\upsilon}_n(\eta))^{-2}\,d\eta,
\tag*{$\textup{(\ref*{15-A-17})}_\D$}\label{15-A-17-D}\\[2pt]
\varepsilon _{\N,n}^{-1} (\eta)
\equiv \int ^\eta (\eta^2-2n-1)(\bar{\upsilon}'_n(\eta))^{-2}\,d\eta
\tag*{$\textup{(\ref*{15-A-17})}_\N$}\label{15-A-17-N}
\end{gather}
and finally we arrive to

\begin{theorem}\label{thm-15-A-4}
As $\eta\to +\infty$ 
\begin{phantomequation}\label{15-A-18}\end{phantomequation}
\begin{gather}
\lambda _{\D,n} (\eta) \equiv 
2n+1+ \Bigl(\int^\eta \bigl(\bar{\upsilon}_n(\eta)\bigl)^{-2}\,d\eta\Bigr)^{-1},
\tag*{$\textup{(\ref*{15-A-18})}_\D$}\label{15-A-18-D}\\[3pt]
\lambda _{\N,n} (\eta)  \equiv 
2n+1 - \Bigl(\int ^\eta (\eta^2-2n-1)\bigl(\bar{\upsilon}'_n(\eta)\bigr)^{-2}\,d\eta \Bigr)^{-1}
\tag*{$\textup{(\ref*{15-A-18})}_\N$}\label{15-A-18-N}
\end{gather}
modulo $O\bigl(e^{-\frac{3}{2}\eta^2}\eta^K\bigr)$ with $K=K(n)$.
\end{theorem}

\begin{remark}\label{rem-15-A-5}
One can see easily from \ref{15-A-18-D}, \ref{15-A-18-N} that 
\begin{equation*}
e^{-\eta^2}\varepsilon_{*,n}^{-1}(\eta)
\sim \sum_{k\ge 0} c'_{*,n,k}\eta^{-2n-1-2k}
\end{equation*}
with the leading coefficient $c'_{*,n,0}= \frac{1}{2}\kappa_n^{-2}$ where $\kappa_n = (n!)^{-\frac{1}{2}}\pi^{-\frac{1}{4}}2^{n/2}$ coefficient at $x^n$ at $\bar{\upsilon}_ne^{x^2/2}$ and thus 
\begin{equation}
e^{-\eta^2} \varepsilon_{*,n}(\eta) = \sim \sum_{k\ge 0} c_{*,n,k}\eta^{2n+1-2k}
\label{15-A-19}
\end{equation}
where 
\begin{equation}
c_{*,n,0}= 2\kappa_n^{2}=   \frac{2^{n+1}}{n!\sqrt{\pi}}.
\label{15-A-20}
\end{equation}
\end{remark}

Now we can estimate recurrently higher-order derivatives of $\varepsilon_{*,n}(\eta)$, $\upsilon_{*,n}(x,\eta)$ as $x=\eta$ arriving to

\begin{corollary}\label{cor-15-A-6}
\begin{equation}
|\partial^j \lambda_{*,n}(\eta)|\asymp \eta^{2n+1+j}e^{-\eta^2}
\label{15-A-21}
\end{equation}
as $\eta >c_{n,j}$.
\end{corollary}

\section{Remarks}
\label{sect-15-A-5}

\begin{remark}\label{rem-15-A-7}
While we do not need it one can prove easily that
\begin{equation}
\lambda_{*,n}(\eta)\sim \eta^2 +\sum_{k\ge 0} c_{*,n,k}\eta^{\frac{2}{3}(1-2k)}
\label{15-A-22}
\end{equation}
as $\eta \to -\infty$ with $c_{*,n,0}$ eigenvalues of operator 
$M(\eta)\Def \bigl(D^2 +2x|\eta|\bigr)$ on $\bR^+$ 
with the corresponding boundary condition ($\D$ or $\N$) as $x=0$ and thus are defined through zeros of Airy function or its derivative. 
\end{remark}

\begin{remark}\label{rem-15-A-8}
(i) For operator $D^2+x^2$ on $\{x\le \eta\}$ albeit with the boundary condition $(u'+\alpha u)(\eta)=0$ with $0\le \alpha<\infty $ one can see easily that
 \ref{15-A-6-N}, \ref{15-A-17-N}   are replaced by 
\begin{gather}
\partial_\eta \lambda_{n} (\eta,\alpha )= 
\bigl(\eta^2-\alpha^2-\lambda _n(\eta,\alpha)\bigr)|
\upsilon_{n}(\eta,\alpha )|^2, 
\tag*{$\textup{(\ref*{15-A-6})}_\alpha$}\label{15-A-6-R}\\
\varepsilon _{n}^{-1} (\eta,\alpha)
\equiv \int ^\eta (\eta^2-2n-1-\alpha^2)(\bar{\upsilon}'_n(\eta,\alpha))^{-2}\,d\eta
\tag*{$\textup{(\ref*{15-A-17})}_\alpha$}\label{15-A-17-R}\\
\end{gather}
respectively.

\medskip\noindent
(ii)  One can prove easily that  $\lambda_{n}(\eta;\alpha)$ monotonically depends on $\alpha$ and $\lambda_{n}(\eta;0)=\lambda_{\N,n}(\eta)$, $\lambda_{n}(\eta;\alpha)\nearrow \lambda_{\D,n}(\eta)$ as $\alpha\to +\infty$. 

However, in virtue of (i) for all $\alpha >0$ there is a single extremum which is a nondegenerate minimum  at $\eta_n=\eta_n(\alpha)>0$ such that 
$\eta_n^2 =\lambda_n(\eta_n) +\alpha^2$ and as $\alpha \to +\infty $ therefore 
$\eta_n (\alpha)\approx \alpha^2 + (2n+1)$ with a small negative error (see remark~\ref{rem-15-A-8}). 

Therefore as  $\alpha\to +\infty$ these minima (see blue line on figure~\ref{eigen-plot-fig} for $\alpha=0$) become more shallow and move to the right and blue-like lines here can approximate red-like lines exactly as on the graphs of $e^{-\eta}-(\eta^2+\alpha^2)^{-1}$.).
\end{remark}

\begin{figure}[h]
\centering
\includegraphics{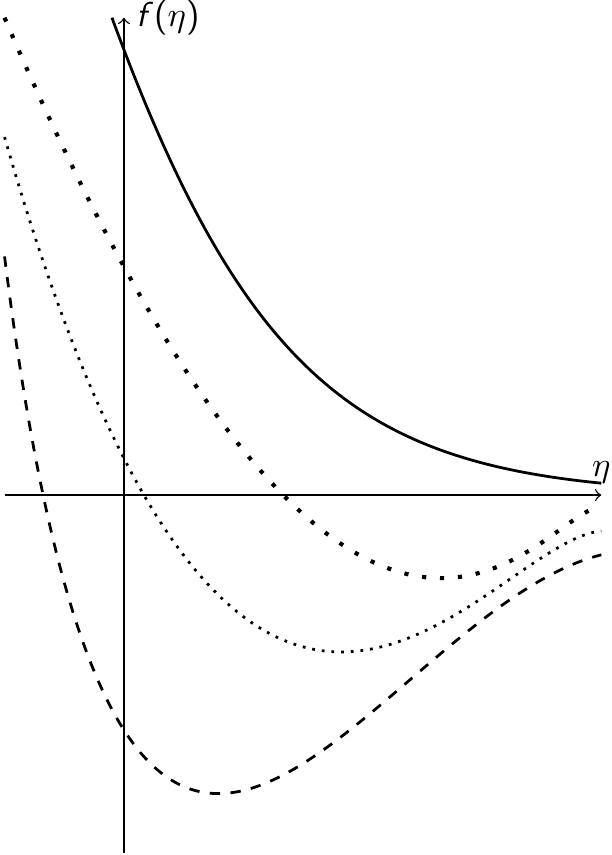}
\caption{Schematics of $f(\eta)=\lambda_n(\eta;\alpha)-(2n+1)$.}
\end{figure}

\begin{conjecture}\label{conj-15-A-9}
For each $n$ function $\lambda_{\D,n}(\eta)$ is convex and $\partial_\eta^2\lambda_{\D,n}(\eta)>0$ while $\lambda_{\N,n}(\eta)$ has a single non-degenerate inflection point $\eta'_n$ and $\partial_\eta^2\lambda_{\D,n}(\eta)\gtrless 0$ as $\eta\lessgtr\eta'_n$.
\end{conjecture}

\newpage

\begin{figure}[h!]
\centering
\includegraphics[width=\linewidth]{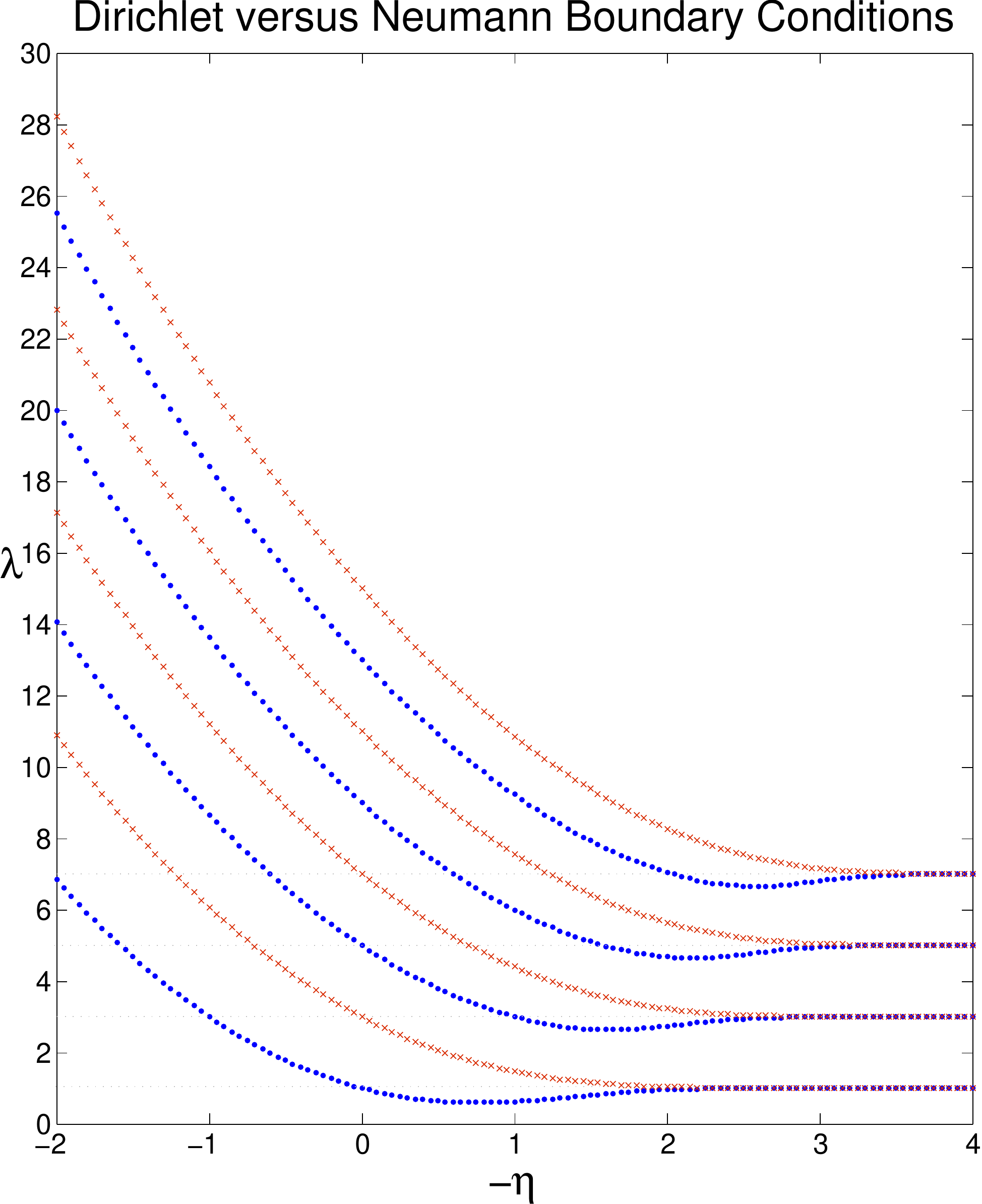}
\caption{\label{eigen-plot-fig} Plots of $\lambda_{\N,n}(\eta)$ and $\lambda_{\D,n}(\eta)$. Calculated by Matlab: Courtesy of Dr.~Marina Chugunova. Here $\eta$ is replaced by $\eta$}
\end{figure}

\end{subappendices}

\bibliographystyle{alpha}

\providecommand{\bysame}{\leavevmode\hbox to3em{\hrulefill}\thinspace}

\vglue .06truein

\begin{tabular}{rrl}
&{\hskip 200 pt} &Department of Mathematics,\cr
&&University of Toronto,\cr
&&40, St.George Str.,\cr
&&Toronto, Ontario M5S 2E4\cr
&&Canada\cr
&&ivrii@math.toronto.edu\cr
&&Fax: (416)978-4107\cr
\end{tabular}

\end{document}